\documentclass[12pt]{amsart}

\usepackage[headheight=1cm, left=2cm, right=2cm, top=3cm, bottom=2cm,
marginparwidth=1.7cm, marginparsep=4mm]{geometry}

\usepackage[utf8]{inputenc}
\usepackage[T1]{fontenc}

\usepackage{amsthm}
\usepackage{amsmath}
\usepackage{amsfonts}
\usepackage{amssymb}

\usepackage{xcolor}

\usepackage[colorlinks,
    colorlinks,
    linkcolor={red!80!black},
    citecolor={blue!80!black},
]{hyperref}

\newcommand{\maxideal}{\mathfrak{m}}
\newcommand{\CC}{\mathbb{C}}
\newcommand{\PP}{\mathbb{P}}

\numberwithin{equation}{section}
\counterwithin{figure}{section}
\newtheorem{theorem}{Theorem}[section]
\newtheorem{corollary}[theorem]{Corollary}
\newtheorem{lemma}[theorem]{Lemma}

\newtheorem{proposition}[theorem]{Proposition}
\theoremstyle{definition}

\newtheorem{remark}[theorem]{Remark}
\newtheorem{example}[theorem]{Example}

\usepackage{todonotes}
\usepackage{xcolor}
\usepackage{booktabs}
\usepackage{setspace}
\usetikzlibrary{matrix,arrows, cd, calc, bending, quotes}

\title[Classification and degenerations\ldots]{Classification and degenerations of small minimal border rank tensors via modules}
\author{Jakub Jagie\l{}\l{}a}
\thanks{Jagie\l{}\l{}a is supported by National Science Centre grant
2020/39/D/ST1/00132}
\author{Joachim Jelisiejew}
\thanks{Jelisiejew is supported by National Science
    Centre grants 2020/39/D/ST1/00132 and 2023/50/E/ST1/00336}

\newcommand{\kk}{\Bbbk}
\newcommand{\OO}{\mathcal{O}}
\newcommand{\spann}[1]{\left\langle #1\right\rangle}
\DeclareMathOperator{\ann}{ann}
\DeclareMathOperator{\degtopure}{\trianglerighteq}%
\DeclareMathOperator{\degto}{\trianglerighteq_{\Sigma}}%
\DeclareMathOperator{\notdegto}{\not\trianglerighteq_{\Sigma}}%
\DeclareMathOperator{\notdegtopure}{\not\trianglerighteq}%
\DeclareMathOperator{\End}{End}
\DeclareMathOperator{\GL}{GL}
\DeclareMathOperator{\charr}{char}
\newcommand{\Tdeg}[1]{T_{\OO_{#1}}}
\newcommand{\Tdegtw}{T_{\widetilde{\OO_{56}}}}
\newcommand{\BBname}{Bia{\l}ynicki-Birula}%
\newcommand{\Gmult}{\mathbb{G}_m}%
\newcommand{\modA}{\underline{A}}
\newcommand{\modB}{\underline{B}}
\newcommand{\modC}{\underline{C}}
\newcommand{\alg}[1]{\mathcal{A}_{111}^{#1}}
\DeclareMathOperator{\OpQuot}{Quot}
\DeclareMathOperator{\Hom}{Hom}
\DeclareMathOperator{\Spec}{Spec}
\DeclareMathOperator{\id}{id}
\DeclareMathOperator{\Gr}{Gr}

\usepackage{marginnote}
\begin{document}

\begin{abstract}
    We give a self-contained classification of $1_*$-generic minimal border rank tensors in $\CC^m \otimes \CC^m \otimes \CC^m$ for $m \leq 5$. Together
    with previous results, this gives a classification of all minimal border
    rank tensors in $\CC^m \otimes \CC^m \otimes \CC^m$ for $m \leq 5$: there
    are $107$ isomorphism classes (only $37$ up to permuting factors).  We fully describe possible degenerations
    among the tensors.  We prove that there are no
    $1$-degenerate minimal border rank tensors in $\CC^m \otimes \CC^m \otimes
    \CC^m $ for $m \leq 4$.
\end{abstract}
\maketitle

\section{Introduction}

We consider tensors in $\CC^m \otimes \CC^m \otimes \CC^m$. 
The \emph{rank} of a tensor $T$ is the smallest integer $r$ for which there exists a decomposition 
$T = \sum_{i=1}^r a_i \otimes b_i \otimes c_i$ and the \emph{border rank} of $T$ is the smallest $r$
such that $T$ can be approximated by rank $r$ tensors.
A tensor is \emph{concise} if it does not lie in any proper subspace
$\CC^{m_1} \otimes \CC^{m_2} \otimes \CC^{m_3}\subsetneq \CC^m\otimes
\CC^m\otimes \CC^m$. Every concise tensor has
border rank at least $m$. A tensor has \emph{minimal border rank} if it is
concise and its border rank is equal to $m$. Buczyński observed that every tensor of
border rank $\leq m$ is a \emph{restriction} of a minimal border rank tensor.
Understanding restrictions is much easier than understanding degenerations, so
minimal border rank tensors shed light on all (not necessarily
concise) border rank $\leq
m$ tensors in $\CC^m \otimes \CC^m\otimes \CC^m$.

Relatively little is known about the geometry of minimal border rank tensors even for small
$m$. The classification for $m=3$ was given in~\cite{third}. The
classification of $1_*$-generic ones (see \S\ref{prelims_tensors} for the definition) 
for $m\leq 5$ was known but rests on an involved
linear algebra computation in the book~\cite{perestanovochnye}. The possible degenerations are
much harder to determine as $m$ grows. These, as far as we know, were not known even for
$m=4$. See~\S\ref{previous} for a detailed discussion of previous work.

Minimal border rank tensors, among them the Coppersmith-Winograd tensors,
appear prominently in complexity theory, see \cite{landsberg_tensors,
BurgisserBook, landsberg_complexity}. The results below can be applied in
particular as follows:
\begin{itemize}
    \item special minimal border rank tensors are an input of the celebrated
        laser method. Typically the big CW tensor is used, but it is subject
        to barrier results, see for example~\cite{Christandl__barriers}.
        Alternative inputs are currently investigated, see for
        example~\cite{Conner_Gesmundo_Landsberg_Ventura,
        Homs_Jelisiejew_Michalek_Seynnaeve, Conner_Huang_Landsberg}.
    \item explicit symbolic degenerations and non-degenerations for minimal
        border rank tensors can be used as testing data for numerous conjectures, such as
        best rank one tensors~\cite{Friedland_Ottaviani}, approximation
        degree~\cite{CGLS} etc. Not much of such explicit symbolic data is
        available in literature, perhaps surprisingly.
    \item to prove non-existence of some tensor degenerations we use new, advanced tools.
        They can be useful also in other contexts such as qubits and
        entanglement and in general in the many fields where tensors are
        employed, see for example~\cite{CGLS, landsberg_tensors} for
        their list.
\end{itemize}

\subsection{Results}

In this article we classify tensors and degenerations of minimal border rank
tensors for $m\leq 5$.
In the introduction we work over $\CC$, although out results are
more general.
We define two tensors to be \emph{isomorphic} (respectively, \emph{isomorphic
up to permutations}) if they differ by a linear coordinate change
(respectively, a linear coordinate change and a permutation of factors).
\begin{theorem}\label{theorem_final}
    Up to isomorphism, there are exactly $1, 2, 6, 21, 107$ minimal border
    rank tensors in $\CC^m \otimes \CC^m \otimes \CC^m$ for $m = 1,2,3,4,5$.
    Up to permutations, the numbers are $1,2,4,11,37$. An
    explicit list is given in \S\ref{classification_tensors}.
\end{theorem}

Minimal border rank tensors subdivide into two classes:
$1_*$-generic and $1$-degenerate ones (see \S\ref{prelims_tensors} for definitions). In this article we directly classify
$1_*$-generic minimal border rank tensors for $m \leq 5$ using modules. This is the content
of Section~\ref{classification}.

\begin{theorem}\label{classification_theorem}
    Up to isomorphism and permutations, there are exactly $1,2,4,11,32$  minimal border rank
    $1_*$-generic tensors in $\CC^m \otimes \CC^m \otimes \CC^m$ for $m =
    1,2,3,4,5$.
\end{theorem}

In Section~\ref{degenerate} we prove that for $m \leq 4$ there are no $1$-degenerate minimal border rank tensors.

\begin{theorem}\label{theorem_degenerate}
    For $m \leq 4$, every minimal border rank tensor in $\CC^m \otimes \CC^m \otimes \CC^m$ is $1_*$-generic. 
\end{theorem}

For $m=5$, the classification of $1$-degenerate minimal border rank tensors is
given in \cite[Theorem 1.7]{concise}. Together with
Theorem~\ref{classification_theorem} and Theorem~\ref{theorem_degenerate},
this
yields the classification from Theorem \ref{theorem_final}, up to
isomorphism and permutations. The classification up to isomorphism, but not
allowing permutations, is done
in~\S\ref{ssec:upToIsomorphism}.

In Section~\ref{sec:degenerations} we determine the possible degenerations,
allowing for permutations. We
found the result quite challenging to obtain. First, it was necessary to construct $66$
minimal degenerations, some of them subtle. Second, and much importantly,
after applying standard invariants, we were still left with showing
nonexistence of $20$ minimal degenerations. To rule them out, we apply subtle module
invariants, the theory of 111-algebras (see~\S\ref{ssec:111algebras}). In two
cases we needed to resort to \BBname{} decompositions, which were not applied
to the tensor setup before.
\begin{theorem}\label{degenerations_theorem}
    The diagram of degenerations for $m=5$ is given in Diagram~\ref{sec:diagram}. There are $66$ minimal degenerations. All of them
    are presented explicitly in the attached \emph{Macaulay2} package, see
    Appendix~\ref{sec:code}.
\end{theorem}
The diagram yields an interesting result on indecomposable tensors. Recall
that $T\in \CC^m \otimes \CC^m\otimes \CC^m$ is a \emph{direct sum} if there
are direct sum decompositions $A'\oplus A'' = \CC^m$, $B'\oplus B'' = \CC^m$,
$C' \oplus C'' = \CC^m$ and nonzero tensors $T'\in A'\otimes B'\otimes C'$,
$T''\in A''\otimes B''\otimes C''$ such that $T = T' + T''$. We say that $T$
is \emph{indecomposable} if it is not a direct sum. In
Diagram~\ref{sec:diagram} the indecomposable tensors are marked $T_{1, *}$ and
$\Tdeg{58}, \ldots ,\Tdeg{54}$. We have the following result.
\begin{corollary}\label{ref:localscheme:cor}
    For $m\leq 5$, every indecomposable tensor of minimal border rank is a degeneration of the multiplication tensor
    of the algebra $\CC[x]/x^m$. (For $m=5$, on Diagram~\ref{sec:diagram} this
    tensor is denoted $T_{1,1}$.)
\end{corollary}
The above result is natural from the point of
irreversibility, as the barriers for matrix
multiplication~\cite[p.3]{Christandl__barriers} obtained for $T_{1,1}$ are
much weaker than those for the big Coppersmith-Winograd tensor (tensor
$T_{1,8}$ in our diagram). For $1_A$-generic tensors,
Corollary~\ref{ref:localscheme:cor} can be rephrased in algebro-geometric
terms by saying that every
smoothable local module of degree $m\leq 5$ is in the closure of the curvilinear
component in the punctual Quot scheme.

The classification of $1_A$-generic minimal border rank tensors up to
isomorphism (without allowing permutations) is
equivalent to the classification of $m$-dimensional subspaces of $\End(\kk^m)$
which are limits of diagonalizable subspaces. We provide this one as well.
\begin{theorem}\label{theorem_subspaces}
    Consider $m$-dimensional subspaces of $\End(\kk^m)$ which are limits of
    diagonalizable subspaces.  Up to isomorphism, there are exactly
    $1,2,5,14,48$ such subspaces for $m = 1,2,3,4,5$. A list of isomorphism
    types of subspaces is given in Subsection \ref{classification_subspaces}.
    Equivalently, consider degree $m$ modules over the polynomial ring
    $\CC[x_1, \ldots ,x_{m-1}]$. Up to affine coordinate changes
    (see~\S\ref{ssec:modulesOne}), there are
    exactly $1,2,5,14,48$ isomorphism classes of such modules.
\end{theorem}
Our methods would likely provide a graph of degenerations also in
this setup, but we refrain from this due to space considerations.
We point out that if we do not allow affine coordinate changes, then there are
infinitely many isomorphism classes of modules and classification is deemed
impossible, see for example~\cite{Moschetti_Ricolfi, Mroz_Zwara}.

\subsection{Methods}\label{ssec:methods}

\newcommand{\Espace}[1]{\mathcal{E}_{\alpha}(#1)}

We refer the reader to~\S\ref{prelims_tensors} for definitions of some of the
notions used below. We let $A$, $B$, $C$ be $m$-dimensional vector spaces and
consider tensors in $A\otimes B\otimes C$.

\subsubsection{Modules} To obtain Theorem~\ref{classification_theorem}, we first classify concise
$\CC[x_1, \dots, x_{m-1}]$-modules of dimension $m \leq 5$, extending the result on
algebras by Poonen~\cite{poonen}. To apply it, we use
 the correspondence between modules, spaces of commuting matrices and
 $1_A$-generic tensors, see~\cite{abelian, concise}, which we recall now. A
 tensor $T$ is \emph{$1_A$-generic} if the image
 of the map $T_A \colon A^\vee \to B \otimes C$ contains an element of maximal
 rank. Any $1_*$-generic tensor becomes $1_A$-generic after permuting factors.

Consider a tensor $T \in A \otimes B \otimes C$ which is $1_A$-generic and has
minimal border rank. Pick an element $\alpha \in A^\vee$ such that
$T_A(\alpha)$ has full rank. Interpret $B \otimes C$ as $\Hom(B^\vee, C)$ and define
\[
    \Espace{T} := T_A(A^\vee) T_A(\alpha)^{-1} \subset \End(C).
\]
The subspace $\Espace{T}$ contains the identity. The tensor $T$ is concise, so
$\Espace{T}$ is $m$-dimensional. Since $T$ has minimal border rank, the space
$\Espace{T}$ consists of pairwise commuting endomorphisms and is closed under
composition of endomorphisms. Therefore
$\Espace{T}$ is a commutative subalgebra of the (noncommutative) algebra $\End(C)$.

Let $S$ denote the polynomial ring $\CC[x_1, \dots, x_{m-1}]$. Choose a basis
$e_0 = \id_C, e_1, \dots, e_{m-1}$~of~$\Espace{T}$. We define an $S$-module
$\modC$ associated to $T$ to be the vector space $C$ with an action of
$S$ given by $x_j\cdot c := e_j(c)$. The module $\modC$ is concise and \emph{End-closed}, i.e., it has the property that for each $f \in S$ there is a linear form $\ell \in S_{\leq 1}$ such that $f - \ell$ annihilates $\modC$.

This procedure can be reversed. Let $M$ be an $S$-module of degree $m$. The multiplication map 
$S_{\leq 1} \otimes M \to M$ gives the tensor $\mu_{M} \in S_{\leq 1}^\vee \otimes M^\vee \otimes M$. 
If $M = \modC$ then $\mu_M$ is isomorphic to $T$.

\begin{example}\label{ex0}
    Let $m = 4$ and fix bases $(a_i)_i, (b_i)_i, (c_i)_i$ of $A, B, C$. Consider the tensor
    \[
        T = a_1 \otimes (b_1 \otimes c_1 + \dots + b_4 \otimes c_4) + a_2 \otimes b_1 \otimes c_2 + a_3 \otimes (b_1 \otimes c_3 + b_3 \otimes c_4) + a_4 \otimes b_1 \otimes c_4.
    \]
    The element $\alpha := a_1^*$ gives the tensor $T_A(\alpha) = b_1 \otimes
    c_1 + \dots + b_4 \otimes c_4$, which corresponds to the identity matrix.
    This shows that $T$ is $1_A$-generic. It is also true that $T$ has minimal
    border rank (it is the tensor $U_{2,4}$ from Appendix~\ref{sec:smallm}), so we assign to $T$ the subspace
    \[
        \Espace{T} = \spann{
        \begin{bmatrix}
            1 & 0 & 0 & 0 \\
            0 & 1 & 0 & 0 \\
            0 & 0 & 1 & 0 \\
            0 & 0 & 0 & 1 \\
        \end{bmatrix},
        \begin{bmatrix}
            0 & 0 & 0 & 0 \\
            1 & 0 & 0 & 0 \\
            0 & 0 & 0 & 0 \\
            0 & 0 & 0 & 0 \\
        \end{bmatrix},
        \begin{bmatrix}
            0 & 0 & 0 & 0 \\
            0 & 0 & 0 & 0 \\
            1 & 0 & 0 & 0 \\
            0 & 0 & 1 & 0 \\
        \end{bmatrix},
        \begin{bmatrix}
            0 & 0 & 0 & 0 \\
            0 & 0 & 0 & 0 \\
            0 & 0 & 0 & 0 \\
            1 & 0 & 0 & 0 \\
        \end{bmatrix}}.
    \]
    As expected, it is 4-dimensional, consists of pairwise commuting matrices and is closed under multiplication.
    Denote the matrices spanning $\Espace{T}$ by $e_0 = \id_C, e_1, e_2, e_3$.
    The underlying vector space of $\modC$ is $C$ and the action of
    $x_1^{a_1} x_2^{a_2} x_3^{a_3}\in S$ on a vector $c \in C$ is given by
    $x_1^{a_1} x_2^{a_2} x_3^{a_3} \cdot c := e_1^{a_1} e_2^{a_2}
    e_3^{a_3}(c)$. For another description of this module see
    Example~\ref{ex:cyclic_not_cocyclic}, Example~\ref{ex1} and Example~\ref{example_tensor}.
\end{example}

The above transformations identify $1_A$-generic minimal border rank tensors
with concise End-closed modules (up to isomorphisms). See Subsection
\ref{prelims_tensors} for more details. The dictionary above is also very
useful to disprove existence of degenerations in $1_A$-generic case,
see Subsection~\ref{ssec:degenerationsPrelims}.

\begin{remark}\label{remark_generalization}
    It is likely that the classification of $1_*$-generic minimal border
    rank tensors can be extended to the case $m = 6$ using our methods.
    Also for $m=6$, Poonen's~\cite{poonen} classification is finite,
    while it becomes infinite for $m=7$.
\end{remark}

\subsubsection{111-algebras}\label{ssec:111algebras}
The proof of Theorem~\ref{theorem_degenerate} and a part of
Theorem~\ref{degenerations_theorem} utilize the correspondence between concise
$111$-abundant tensors and surjective bilinear non-degenerate maps between
concise modules. This correspondence is based on the $111$-algebra,
introduced in \cite{concise}. Below we outline it.

Let $T$ be a concise tensor in $A \otimes B \otimes C$. Let $\alg{T}$ denote the subset of $\End A \times \End B \times \End C$ consisting of triples $(X, Y, Z)$ such that
\[
    (X \otimes \mathrm{id} \otimes \mathrm{id})(T)
    = (\mathrm{id} \otimes Y \otimes \mathrm{id})(T)
    = (\mathrm{id} \otimes \mathrm{id} \otimes Z)(T).
\]
The set $\alg{T}$ is called the \emph{111-algebra} of $T$. It is a commutative unital subalgebra of $\End A \times \End B \times \End C$, see \cite[Theorem 1.11]{concise}.
The tensor $T$ is called \emph{111-abundant} if $\dim_\CC \alg{T} \geq m$. In particular all minimal border rank tensors are $111$-abundant.

Let $T$ be a concise $111$-abundant tensor. Then there exist an associative commutative unital $\CC$-algebra $\mathcal{A}$ of rank at least $m$, concise $\mathcal{A}$-modules $M, N, P$ of degree $m$ and a surjective bilinear non-degenerate map of $\mathcal{A}$-modules
$
    \varphi \colon M \times N \to P
$
such that $T$ corresponds to the composition of the linear maps $M \otimes N \to M \otimes_{\mathcal{A}} N \to P$. Moreover each tensor coming from a map as above is concise and 111-abundant. For proofs of this characterisation see \cite[Theorem 5.5]{concise}.

The map $\varphi$ can be found explicitly. The algebra $\alg{T}$ projects onto each of $\End A ,\End B, \End C$. It gives each of the spaces $A, B, C$ a structure of an $\alg{T}$-module, denoted by $\modA, \modB, \modC$. The linear map $T_C^{\top} \colon A^\vee \otimes B^\vee \to C$ factors through the linear map $A^\vee \otimes B^\vee \to \modA^\vee \otimes_{\alg{T}} \modB^\vee$ and induces a map $\varphi \colon \modA^\vee \otimes_{\alg{T}} \modB^\vee \to \modC$, which is an $\alg{T}$-module homomorphism corresponding to $T$.

The tensor $T$ is $1_*$-generic when at least one of $M, N, P^\vee$ is cyclic,
see \cite[Theorem 5.3]{concise}.
To prove Theorem~\ref{theorem_degenerate}, we use the classification of concise
$S$-modules of degree $m\leq 4$ and show that there are no suitable maps
$\varphi$
between non-cyclic ones.

\subsection{Previous work}\label{previous}

The classification of minimal border rank tensors is motivated by algebraic complexity theory and classical algebraic geometry. Such tensors are essential building blocks used to prove upper bounds on the exponent of matrix multiplication via the Strassen’s laser method. They are also closely related to study of secant varieties in algebraic geometry, as they form a dense open subset of the cone over the $m$-th secant variety of Segre variety $\hat{\sigma}_m(Seg(\PP_\CC^{m-1} \times \PP_\CC^{m-1} \times \PP_\CC^{m-1}))$.

The problem of classification of tensors is hard already in small dimensions.
Only the tensors of border rank at most three are fully classified, see
\cite{third}. The minimal border rank tensors for $m = 4$ are well understood
in terms of equations (but not isomorphism types). The variety of these tensors is described
as the zero set of explicit polynomial equations in \cite{friedland_4}. A
refined set of equations, which conjecturally generates the ideal,
is obtained in~\cite{bates-oeding}, together with numerical evidence.
Defining equations for the set of tensors of minimal border rank for $m = 5$
and the set of minimal border rank $1_*$-generic tensors for $m = 5, 6$ are
described in \cite{concise}. They were obtained via introducing the
111-algebra, which was  motivated by the 111-space, introduced in
\cite{buczynski}. The result that there are no $1$-degenerate minimal border
rank tensors for $m \leq 4$ can be extracted from \cite[Section
3]{friedland_4}, although it is not stated explicitly there and the extraction
is difficult. \emph{Symmetric} tensors of symmetric (or Waring) border rank four are much more
understood, see~\cite{Ballico_Bernardi, Landsberg_Teitler}, however it is
important to remember that a priori the symmetric border rank and border rank of
a symmetric tensor might differ (this is the border version of Comon's
conjecture).

The classification of $1_*$-generic minimal border rank tensors for $m=5$ was
obtained in~\cite[Subsection~6.4]{abelian}, which relies on the
classification of nilpotent commutative subalgebras of matrices obtained in
\cite[Chapter~3.3]{perestanovochnye} via a long explicit calculation.
Landsberg and Micha{\l}ek
manually check that nineteen of the resulting tensors have minimal border rank
and one of them does not have minimal border rank.

In our classification, these 20 tensors correspond to the ones that come from
local modules. Our result agrees with \cite{perestanovochnye}, while
there are some inaccuracies in the result of \cite[Subsection 6.4]{abelian},
which we discuss now. First, the subalgebra corresponding to the tensor
$T_{N_{6,8}}$ from \cite{abelian} is not commutative. It appears that this is
due to a typo introduced in \cite{abelian}. Second, the numbering of tensors $T_{N_{16}}, T_{N_{17}}$ is switched with respect to the numbering from \cite{perestanovochnye}. The tensors $T_{N_{15}}, T_{N_{17}}^\vee$ are isomorphic. The tensors corresponding to $T_1, T_4$ from our classification are missing. The correspondence between the three classifications is summarised in the following tables.

\[
\begin{array}{c|ccccccccc}
    \text{this paper} & T_1 & T_2 & T_3 & T_4 & T_5 & T_6 & T_7 & T_8 & T_9 \\
    \hline
    \text{\cite{abelian}} & - & T_{N_{15}}, T_{N_{17}}^\vee & T_{N_{16}} & - & T_{N_{7,9}}^\vee & T_{N_{10,12}} & T_{N_{11,13}} & T_{N_{14}} & T_{N_{1,4}} \\
    \hline
    \text{\cite{perestanovochnye}} & N_{18} & N_{15}, N_{16}^\vee & N_{17} & N_6^\vee, N_8 & N_7^\vee, N_9 & N_{10}, N_{12}^\vee & N_{11}, N_{13}^\vee & N_{14} & N_1, N_4^\vee \\
\end{array}
\]

\[
\begin{array}{c|ccccccccccc}
    \text{this paper} & T_{10} & T_{11} & T_{12} & T_{13} & T_{14} & T_{15} & T_{16} & T_{17} & T_{18} & T_{19} & T_{20}\\
    \hline
    \text{\cite{abelian}} & T_9 & T_{21} & T_{22} & T_{19} & T_{23} & T_{20} & S_1 & S_2 & S_3 & S_4 & T_{Leit, 5}\\
\end{array}
\]

In contrast with~\cite{perestanovochnye, abelian}, we establish the
classification over any algebraically closed field of characteristic different
from two (although in introduction we assume for sake of simplicity that the
base field is $\CC$). We feel that our approach is self-contained and uses
more conceptual techniques. The only exterior classification that we use is
the classification of commutative rank $m$ algebras over an algebraically
closed field for $m \leq 5$ from Poonen's~\cite{poonen}. Poonen's~paper is short and
self-contained and additionally its results can be recovered using apolarity,
as we illustrate in Example~\ref{ex:poonen}.

Our argument uses general results, such as the
correspondence between tensors and modules and the result from \cite[Theorem
1.4]{concise}, and is conducted mostly in the language of commutative algebra.
Thus the method can in the future yield results for higher $m$ as well, see
Remark \ref{remark_generalization}.

For degenerations of tensors, very important results are contained
in~\cite{Blaser_Lysikov, abelian, concise}. Numerical tools can be
successfully applied to heuristically obtain degenerations and bounds on
ranks, see for example~\cite{Conner_Huang_Landsberg,
Conner_Gesmundo_Landsberg_Ventura}, however transforming this into a symbolic
degeneration is still challenging.

\subsection{Acknowledgements}

The authors are very grateful to Jarosław Buczyński, Austin Conner, Joseph M.~Landsberg, and
Mateusz Micha{\l}ek
for their helpful suggestions to improve earlier drafts, and especially to
Joseph M.~Landsberg for forcing them to deal with the degeneration graph. We
thank an anonymous referee for a thorough and very helpful review.

\section{Preliminaries}
\subsection{Tensors}\label{prelims_tensors}
Let $\kk$ be an algebraically closed field with $\charr \kk \neq 2$ and let $A, B, C$ be copies of $\kk^m$. We will be interested in tensors $T \in A \otimes B \otimes C$. 
We define two tensors $T, T'$ to be \emph{isomorphic up to permutations} if there exists a
permutation $\sigma \in \Sigma_3$ and a triple of linear automorphisms $(g_A,
g_B, g_C) \in \GL(A)\times \GL(B)\times \GL(C)$ such that applying $g_A, g_B,
g_C$ on the corresponding factors of $T$ and the permuting the factors by
$\sigma$ yields $T'$. We say that $T, T'$ are \emph{isomorphic} if a triple
above exists with $\sigma$ the identity permutation. Two tensors are
isomorphic up to permutations if and only if
they lie in the same orbit of the action of $(\GL(A)\times \GL(B)\times
\GL(C)) \rtimes \Sigma_3$ on $A\otimes B\otimes C$.

A tensor $T$ induces linear maps $T_A\colon A^\vee \to B \otimes C$,
$T_B\colon B^\vee \to A \otimes C$ and $T_C\colon C^\vee \to A \otimes B$.
We say that $T$ is \emph{$A$-concise} if the map $T_A$ is injective, and $T$ is \emph{concise} if it is simultaneously $A, B$ and $C$-concise.
A
tensor $T$ is \emph{$1_A$-generic} if the image of $T_A$ contains an element
of rank $m$ and \emph{$1_*$-generic} if it is at least one of $1_A$-, $1_B-$
or $1_C$-generic. If $T$ is $1_A$-generic, then it is $B$ and $C$-concise. Tensors which are not $1_*$-generic are called
\emph{1-degenerate}. For a $1_A$-generic tensor $T$, pick an element $\alpha \in A^\vee$ such that
$T_A(\alpha)$ has full rank. Interpret $B \otimes C$ as $\Hom(B^\vee, C)$ and define
\begin{equation}\label{eq:Espace}
    \Espace{T} := T_A(A^\vee) T_A(\alpha)^{-1} \subset \End(C).
\end{equation}
In this setup, we say that $T$ \emph{satisfies the
$A$-Strassen's equations} if $\Espace{T}$ consists of pairwise commuting
matrices. We say that $T$ is \emph{$A$-End-closed} if the
space $\Espace{T}$ is closed under the composition of endomorphisms.
Minimal border rank tensors are automatically End-closed and satisfy
Strassen's equations.
While it is unimportant for the current article, both conditions can be
expressed in terms of equations on coefficients of $T$, see~\cite{strassen}
and~\cite[\S2.1, \S2.4]{abelian}.

\subsection{Modules I}\label{ssec:modulesOne}

Let $S$ denote the polynomial ring $\kk[x_1, \dots, x_{m-1}]$.
An $S$-module $M$ has \emph{degree $m$} if $\dim_\kk M = m$.
For an $S$-module $M$ and an algebra automorphism $\varphi\colon S\to S$ we define
$M^{\varphi}$ to be the $S$-module with the action given by $f\cdot n :=
\varphi(f)\cdot n$ for every $f\in S$. An automorphism $\varphi\colon S\to S$
is an \emph{affine change of variables} if for every $i$ the image
$\varphi(x_i)$ is a $\kk$-linear combination of $1, x_1, \ldots, x_{m-1}$.
Every linear isomorphism $S_{\leq 1} \to S_{\leq 1}$ that preserves $1$
extends uniquely to an affine change of coordinates.
Two $S$-modules $M, N$ of degree $m$ are \emph{equivalent} if $M \simeq N^\varphi$ for some affine change of variables $\varphi$.

To a degree $m$ module $M$ we associated the \emph{multiplication tensor} $\mu_M \in S_{\leq 1}^\vee \otimes M^\vee \otimes M$. The tensor $\mu_M$ is automatically $1_A$-generic because the image of $1 \in S_{\leq 1}$ in $M^\vee \otimes M = \Hom(M, M)$ is the identity.
The \emph{annihilator} of a module $M$ is $\ann(M) = \left\{ f\in S\ |\ fM = 0
\right\}$. We say that $M$ is \emph{concise} if $\mu_M$ is concise, which is
equivalent to saying that $\ann(M)$ is disjoint from $S_{\leq 1}$. We say that
$M$ is \emph{End-closed} if for each $i, j$ there exists a linear form $y$
such that $(x_i x_j - y) M = 0$.

\begin{example}\label{example_min_brank}
    Let $m = 5$ and consider the $S$-module
    \[
        M = \frac{Se_1 \oplus Se_2}
        {(x_{4}e_{1},\:x_{3}e_{1},\:x_{2}e_{1}-x_{4}e_{2},\:x_{2}e_{2},\:x_{1}e_{2},\:x_{1}^{2}e_{1}-x_{3}e_{2})}.
    \]
    The vector space $M$ has a basis $e_1, e_2, x_1e_1, x_2e_1, x_1^2 e_1$.
    In particular, the element $x_1^2 e_1$ cannot be expressed as a linear combination of 
    $e_1, x_1e_1, x_2e_1$, so there is no $y \in S_{\leq 1}$ such that $(x_1^2 - y) e_1 = 0$,
    hence $M$ is not End-closed.
\end{example}

The \emph{dual module} of $M$ is the $S$-module $M^\vee = \Hom_\kk(M, \kk)$,
where the module structure is $(f\cdot \varphi)(m) := \varphi(fm)$ for every $\varphi\in M^{\vee}$,
$f\in S$, $m\in M$. We have a natural isomorphism $M \to (M^\vee)^\vee$ given
by the usual map. The multiplication tensor of $M^{\vee}$ is obtained from
$\mu_M$ by transposing two factors.

The module $M$ is \emph{cyclic} if there
exists an element $m\in M$ such that $S\cdot m = M$. If this happens, we have
$M \simeq S^{\oplus 1}/\ann(M)$ and we say
that $M$ \emph{comes from an algebra} $S/\ann(M)$. A module is cyclic if and
only if $\mu_M$ is $1_B$-generic.  We say that $M$ is \emph{cocyclic} if
$M^\vee$ is cyclic. This happens if and only if $\mu_M$ is $1_C$-generic.

\begin{example}\label{ex:cyclic_not_cocyclic}
    Let $m = 4$ and consider the $S$-module 
    \[
        M = \kk[x_1, x_2, x_3] / (x_1^2, x_1 x_2, x_2^3, x_3 - x_2^2).
    \]
    An explicit calculation shows that $\mu_M$ is the tensor from Example
    \ref{ex0}, that is
    \[
        \mu_M = a_1 \otimes (b_1 \otimes c_1 + \dots + b_4 \otimes c_4) + a_2 \otimes b_1 \otimes c_2 
        + a_3 \otimes (b_1 \otimes c_3 + b_3 \otimes c_4) + a_4 \otimes b_1 \otimes c_4.
    \]
    This tensor is $1_B$-generic, but not $1_C$-generic, so the module $M$ is cyclic but is not cocyclic 
    (and consequently, the module $M^\vee$ is cocyclic but is not cyclic). Hence, 
    the notions of being cyclic or cocyclic are independent. 
\end{example}

\subsection{Modules and $1_A$-generic tensors}\label{ssec:modulesAnd1Ageneric}

Consider a concise $1_A$-generic tensor $T$ that satisfies the $A$-Strassen's
equations. Take a space of commuting matrices $\Espace{T}$ as
in~\eqref{eq:Espace} and choose its basis $e_0 = \id_C$, $e_1$, \ldots
,$e_{m-1}$. Using this space, we define an action of $S$ on $C$ by $x_i\cdot c
= e_i(c)$ for every $i=1,2, \ldots ,m-1$. The resulting $S$-module is denoted
$\modC$. The multiplication tensor of such $\modC$ is
isomorphic to $T$, so it is $1_A$-generic, concise and satisfies the
$A$-Strassen's equations. Conversely,
for a concise $S$-module $M$ of degree $m$, we obtain a multiplication tensor
$\mu_M$ which is $1_A$-generic, concise and satisfies $A$-Strassen's
equations. The tensor $\mu_M$ is End-closed if and only if $M$ is End-closed. 
The Example~\ref{example_min_brank} shows that this condition is not vacuous.

The following result binds the classification of modules and their
multiplication tensors.
\begin{lemma}
    The multiplication tensors $\mu_M, \mu_N$ of concise $S$-modules $M$, $N$ are isomorphic
    if and only if $M$ and $N$ are equivalent $S$-modules.
\end{lemma}
The argument follows
implicitly from~\cite[\S2]{concise} or~\cite[\S2]{abelian} but we know no
explicit reference.
\begin{proof}
    Suppose first that $N = M^{\varphi}$ is an $S$-module equivalent to $M$
    via an affine change of coordinates $\varphi\colon S_{\leq 1}\to
    S_{\leq 1}$. By definition, their multiplication maps satisfy
    \[
        \begin{tikzcd}[column sep=small]
            S_{\leq 1}\ar[d, "\varphi"] & \otimes & M\ar[d, equal] \ar[rrr,
            "\mu_{M^{\varphi}}"] &&&
            M\ar[d, equal]\\
            S_{\leq 1} & \otimes & M \ar[rrr, "\mu_M"]  &&& M
        \end{tikzcd}
    \]
    so that $(\varphi^{\vee}\otimes \id_{M^{\vee}} \otimes \id_M)(\mu_M) =
    \mu_{M^{\varphi}}$ is the required isomorphism of tensors.

    Suppose conversely that $\mu_M\in S_{\leq 1}^{\vee}\otimes M^{\vee}\otimes M$ and
    $\mu_N\in S_{\leq 1}^{\vee} \otimes N^{\vee} \otimes N$ are isomorphic
    tensors, that is, there are linear isomorphisms $f_S\colon S_{\leq 1}^{\vee}\to S_{\leq
    1}^{\vee}$, $f_{M^{\vee}}\colon M^{\vee}\to N^{\vee}$ and $f_M\colon M\to
    N$ such that
    \begin{equation}\label{eq:iso}
        (f_S\otimes f_{M^{\vee}} \otimes f_M)(\mu_M) = \mu_N.
    \end{equation}
    Take
    $\varphi = f_S^{\vee}\colon S_{\leq 1}\to S_{\leq 1}$. This linear map is
    bijective, so we can view it as an affine change of coordinates. We have
    $f_S = \varphi^{\vee}$. We claim
    that $N$ is isomorphic to $M^{\varphi}$. The multiplication tensor of
    $M^{\varphi}$ satisfies $(f_S\otimes \id_{M^{\vee}}\otimes \id_M)(\mu_M) =
    \mu_{M^{\varphi}}$. Comparing this with~\eqref{eq:iso}, we obtain that
    \[
        \begin{tikzcd}[column sep=small]
            S_{\leq 1}\ar[d, equal] & \otimes & M\ar[d,
            "f_{M^{\vee}}^{\vee}"] \ar[rrr,
            "\mu_{M^{\varphi}}"] &&&
            M\ar[d, "f_M"]\\
            S_{\leq 1} & \otimes & N \ar[rrr, "\mu_N"]  &&& N
        \end{tikzcd}
    \]
    is commutative. Evaluating at $1\in S_{\leq 1}$ and using that $\mu_N(1, -) =
    \id_N$, $\mu_{M^{\varphi}}(1, -) = \id_M$, we obtain that
    $f_{M^{\vee}}^{\vee} = f_M$. The map $f_M$ is the required isomorphism.
\end{proof}

\subsection{Modules II}\label{prelim:modules}
Recall that $S = \kk[x_1, \dots, x_{m-1}]$ and let $M$ be an $S$-module of
finite degree. For a maximal ideal $\maxideal \subset S$, let $M_\maxideal$
denote the localization of $M$ with respect to the multiplicatively closed set
$S \setminus \maxideal$. Equivalently, this is the quotient module
$M/\maxideal^N M$ for any $N\gg 0$.

The module $M$ has a finite length, so there exist unique maximal ideals $\maxideal_1, \dots, \maxideal_s \subset S$ such that 
\[
    M \simeq M_{\maxideal_1} \oplus \dots \oplus M_{\maxideal_s}
\]
with $M_{\maxideal_i}$ nonzero for $i=1, \ldots , s$. We say that
$\{\maxideal_1, \ldots ,\maxideal_s\}$ is the \emph{support} of $M$. We will
refer to the numbers $\dim_{\kk} M_{\maxideal_1}$, \ldots, $\dim_{\kk}
M_{\maxideal_s}$ as the \emph{degree decomposition} of $M$. If $r=1$
we say that $M$ is \emph{local}. A nonzero module $M$ is local if and only if $\maxideal^N M = 0$ 
for some $N \geq 1$. For proofs of these claims see \cite[Theorem~2.13]{eisenbud}.

The classification of local modules is easier than the classification of general modules because 
each local $S$-module has a natural structure of an $S_\maxideal$-module, so we can work over 
a local ring and utilize results such as Nakayama's lemma.

\begin{lemma}[Nakayama's Lemma, {\cite[Corollary~2.7]{Atiyah_Macdonald}}]\label{ref:Nakayama:lem}
    Let $M$ be a local module as above and $N\subseteq M$ be a submodule. If
    $M = \maxideal M + N$, then $M = N$.
\end{lemma}

We work over an algebraically closed field, so by Hilbert's Nullstellensatz
$\maxideal = (x_1 - a_1, \dots, x_{m-1} - a_{m-1})$ for some $a_1, \dots, a_{m-1} \in \kk$.
We classify modules only up to affine changes of variables, so we can further assume that 
$\maxideal = (x_1, \dots, x_{m-1}) = S_+$. 
Note that $\maxideal$ is fixed only under linear changes of variables.

Let $M$ be a cyclic $S$-module. There is an isomorphism of $S$-modules $M
\simeq S^{\oplus 1} / \ann(M)$, so $M$ has a natural structure of an unital commutative
$S$-algebra, which is in particular a $\kk$-algebra. Each unital commutative
$\kk$-algebra of rank $m$ yields a concise cyclic $S$-module of degree $m$
(unique up to a change of variables in $S$, see Corollary \ref{corr_cyclic}).
This correspondence restricts to local algebras and local modules. Such
algebras were classified for small $m$ by Poonen~\cite{poonen}. We follow his
naming convention and represent such $\kk$-algebras as quotients of the polynomial ring in variables~$x, y, z, \dots$.

\begin{example}\label{ex1}
    Consider the unital commutative $\kk$-algebra $\kk[x, y] / (x^2, xy, y^3)$
    of degree $4$. This algebra is a quotient of the polynomial ring in 2
    variables, so a priori it has a structure of a $\kk[x_1,x_2]$-module.
    We can make it into a concise $\kk[x_1, x_2, x_3]$-module by choosing a basis $1, x, y, y^2$ and 
    declaring that $1, x_1, x_2, x_3$ act as multiplication by corresponding elements of this basis:
    \begin{center}
    \[
        \begin{array}{c|cccc}
                & 1 & x & y & y^2 \\
            \hline
            1   & 1 & x & y & y^2 \\
            x_1 & x & 0 & 0 & 0   \\
            x_2 & y & 0 & y^2 & 0 \\
            x_3 & y^2 & 0 & 0 & 0 \\
        \end{array}
    \]
    \end{center}
    This is an explicit description of the module from Example~\ref{ex:cyclic_not_cocyclic}.

    The multiplication tensor in the $\kk[x_1, x_2]$-module is the restriction
    of the multiplication tensor in the $\kk[x_1, x_2, x_3]$-module via the
    inclusion
    $\kk[x_1, x_2]_{\leq 1}\hookrightarrow \kk[x_1, x_2, x_3]_{\leq 1}$.
\end{example}
In the following we will repeatedly use the ``add additional variable'' construction from
Example~\ref{ex1}. The following lemmas address the issue of when we can extend modules to concise modules over polynomial rings in more variables and whether these extensions are unique.
For a homomorphism of rings $\psi\colon S'\to S$ and an $S$-module $M$, the
\emph{restriction of $M$} (via $\psi$) is the $S'$-module $M$ such that for
$s'\in S'$ and $n\in M$ we have $s'\cdot n := \psi(s')n$.

\begin{lemma}\label{concise+endclosed}
    Let $S'$ be a polynomial ring over $\kk$. Let $M'$ be an $S'$-module of
    degree $m$ such that the space $\End_{S'}(M')$ is $m$-dimensional and
    its elements pairwise commute. Then there is a \emph{concise} $S$-module
    $M$ that restricts to $M'$ via a linear inclusion $S'\hookrightarrow S$.
    The module $M$ is unique up to equivalence and it is End-closed.
\end{lemma}
\begin{proof}
    Choose a basis $e_0 = \id_{M'}, e_1 \dots, e_{m-1}$ of $\End_{S'}(M')$.
    We can assume that $e_1, \ldots ,e_k$ correspond to multiplications by
    elements of a basis of $S'_1$. We define a structure of a concise $S$-module $M$ 
    on the underlying vector space of $M'$ by setting $x_i n := e_i(n)$ for $i=1,2, \ldots ,m-1$. 
    This identifies $S'$ with a subring $\kk[x_1, \ldots ,x_k]\subseteq S$, so $M'$ is indeed
    a restriction of $M$. To prove uniqueness, observe that for every other
    $M$, by conciseness we obtain an inclusion $S_{\leq 1} \hookrightarrow
    \End_{S}(M)$. Moreover, $\End_{S}(M)$ is a subspace of $\End_{S'}(M')$.
    We have $\dim_{\kk} S_{\leq 1} = m = \dim_\kk \End_{S'}(M')$, so in particular
    the inclusion $S_{\leq 1} \hookrightarrow\End_{S}(M)$ is an
    isomorphism and $M$ differs from the choice above only by a change of basis.
    Finally, the $S$-module $M$ is End-closed, because the composition of any two endomorphisms 
    is again an endomorphism, so it corresponds to multiplication by some linear form. 
\end{proof}

\begin{corollary}\label{corr_cyclic}
    Let $M'$ be a cyclic or cocyclic $S$-module of degree $m$. Up to a linear change of variables in $S$, there is a unique way to give $M'$ a structure of a \emph{concise} $S$-module $M$. The module $M$ is automatically End-closed.
\end{corollary}
\begin{proof}
    Assume that $M'$ is cyclic. Let $1$ denote the unit of the associated
    $S$-algebra $S / \ann(M')$. The endomorphisms of the $S$-module $M'$ are
    determined by the choice of the image of $1$, so there is a natural
    isomorphism of vector spaces $\End_{S}(M')$ and $S / \ann(M')$. Therefore $\dim_\kk \End_{S}(M') = m$ and Lemma \ref{concise+endclosed} applies. 

    Assume that $M'$ is cocyclic. Endomorphisms of $M' = \Hom_\kk((M')^\vee, \kk)$ are given by precomposition with endomorphisms of the cyclic module $(M')^\vee$, so $\dim_\kk \End_S(M') = m$ an we conclude by Lemma \ref{concise+endclosed}.
\end{proof}
Corollary~\ref{corr_cyclic} can be generalized to disjoint sums.  It will be
useful for recovering the classification of all concise modules from the
classification of local concise modules in Subsection
\ref{classification_concise}, see Example~\ref{example_sum}.
\begin{corollary}\label{cor_unique}
    Let $S'\hookrightarrow S$ be a polynomial subring and
    $N'_1, \dots, N'_r$ be $S'$-modules such that their supports are pairwise
    disjoint, that each of them is cyclic or cocyclic, and that
    $\sum_{i=1}^r \dim_{\kk} N'_i = m$. Then there exist
    $S$-modules $N_1, \dots, N_r$ such that $N_i$ restricts to $N_i'$ for
    $i=1,2, \ldots ,r$ and that $N_1\oplus  \ldots \oplus N_r$ is concise.
    Such an $S$-module is unique up to equivalence and End-closed.
\end{corollary}
\begin{proof}
    Let $N' = N_1'\oplus  \ldots \oplus N'_r$. Since supports are disjoint, we
    have $\Hom_{S'}(N_i', N_j') = 0$ for $i\neq j$, so that
    \[
    \End_{S'}(N') = \End_{S'}(N_1')\oplus \ldots \oplus \End_{S'}(N'_r).
    \]
    By cyclicity, arguing as in Corollary~\ref{corr_cyclic}, we obtain that
    $\dim_{\kk} \End_{S'}(N') = m$, so we can apply
    Lemma~\ref{concise+endclosed} and obtain a concise, End-closed $S$-module
    $N$ that restricts to $N'$. Either by direct check or by disjointness we
    have $N = N_1\oplus \ldots \oplus N_r$, where $N_i$ restricts to $N_i'$
    for $i=1, \ldots ,r$.
\end{proof}

Alone, the condition that $\dim_\kk \End_S N_1 + \dots + \dim_\kk \End_S N_r = m$ does not guarantee that the $S$-module $N_1 \oplus \dots \oplus N_r$ can be made concise. 

\begin{example}
    Let $S = \kk[x_1, x_2, x_3]$. Consider the $\kk[x]$-module $N'_1 =
    \kk[x]/(x^2)$ and the $S$-modules $N_2 = N_3 = S/(x_1, x_2, x_3)$. Clearly
    $\dim_\kk \End N'_1 + \dim_\kk \End N_2 + \dim_\kk \End N_3 = 2 + 1 + 1 =
    4$. Consider any $S$-module $N_1$ such that there exists a linear form $x
    \in S_1$ such that the restriction of scalars via the inclusion map
    $\kk[x] \subset S$ yields $N'_1$ and $\End N_1 = \End N'_1$. Let $M = N_1
    \oplus N_2 \oplus N_3$. The factor $N_2 \oplus N_3$ is annihilated by
    $S_+$, so $S_{\leq 1} \cap \ann(M)= S_1 \cap \ann(N_1)$. The space $\End N_1$ is $2$-dimensional, so the endomorphism corresponding to multiplication by $x_1, x_2, x_3 \in S_1$ are linearly dependent, so $S_{\leq 1} \cap \ann(M)\neq 0$.
\end{example}

The following lemma generalizes this example in the local case.
\begin{lemma}\label{not_concise}
    Let $M$ be a local $S$-module of degree $m$. If there exists a cyclic or cocyclic $S$-module $N$ and an integer $l \geq 1$ such that $M \simeq (S/\maxideal)^{\oplus l} \oplus N$, then $M$ is not concise.
\end{lemma}
\begin{proof}
    Note that $S_{\leq 1} \cap \ann(M)= S_1 \cap \ann(M)= S_1 \cap \ann(N)$.
    If $N$ is cyclic, then $\ann(N)$ coincides with the annihilator of the
    unit $1$ of the algebra corresponding to $N$. The module $N$ is local, so
    it is annihilated by large powers of $\maxideal$, so the subspace $S_1
    \cdot 1 \subset N$ cannot contain $1$. It follows that $\dim_\kk S_1 \cdot
    1 < \dim_\kk N \leq \dim_\kk S_1$, so the map $S_1 \to S_1 \cdot 1$ has a
    non-zero kernel, so $S_1 \cap \ann(N)$ is non-zero. The identity $\ann(N)=
    \ann(N^\vee)$ asserts that the result holds also for cocyclic modules.
\end{proof}

Above, we introduced lemmas which can be used to obtain new modules from already classified ones. The following enables us to determine some modules satisfying the assumptions of these lemmas. Let $(0 : \maxideal)_M$ denote the \emph{socle} of a local module $M$, i.e., the maximal submodule of $M$ annihilated by $\maxideal$.

\begin{lemma}\label{redundant}
    Let $M$ be a local $S$-module of finite degree. Assume there exists an element $m \in M$ such that $m \in (0 : \maxideal)_M$ and $m \notin \maxideal M$. Then there exists a local $S$-module $N$ such that $M \simeq S/\maxideal \oplus N$.
\end{lemma}
\begin{proof}
    Let $r = \dim_\kk M / \maxideal M$. Consider the free modules $F' := S e_2 \oplus \dots \oplus S e_r$ and $F := Se_1 \oplus F'$. The element $m$ is a minimal generator, so we can choose a surjection $F \to M$ such that $e_1 \mapsto m$. Let $K$ be the kernel of this surjection and let $K' = K \cap F'$. The element $m$ lies in $(0 : \maxideal )_M$, so $K \cap S e_1 = \maxideal e_1$ and consequently we obtain $K = K \cap (S e_1 \oplus F') =  \maxideal e_1 \cap K'$. Therefore $M \simeq F/K = (Se_1 \oplus F') / (\maxideal e_1 \oplus K') \simeq S / \maxideal \oplus F' / K'$, so we can take $N = F' / K'$.
\end{proof}

\begin{lemma}\label{cyclic}
    Let $M$ be a local $S$-module of finite degree. The following hold:
    \begin{enumerate}
        \item The module $M$ is cyclic if and only if $\dim_{\kk} M / \maxideal M = 1$.
        \item The module $M^\vee$ is cyclic if and only if $\dim_{\kk} (0 : \maxideal)_M = 1$.
    \end{enumerate}
\end{lemma}

\begin{proof}
    \begin{enumerate}
        \item Assume that $\dim_{\kk} M / \maxideal M = 1$. Choose an element
            $n \in M$ whose image spans $M / \maxideal M$. The map $Sn \to M /
            \maxideal M$ is surjective, so $M = \maxideal M + Sn$, hence $M =
            Sn$ by Nakayama's Lemma~\ref{ref:Nakayama:lem}. If $M$ is cyclic, then it inherits its $S$-module structure from the structure of the corresponding $S$-algebra, so $M / \maxideal M$ is spanned by the class of the unit of this $S$-algebra.
        \item The perfect pairing $M^\vee \times (M^\vee)^\vee \to \kk$ induced by evaluation gives an isomorphism between the space $M^\vee / \maxideal M^\vee$ and the subspace of $(M^\vee)^\vee$ consisting of functionals vanishing on $\maxideal M^\vee$. The latter space is equal to $(0 : \maxideal)_{(M^\vee)^\vee}$, which has the same dimension as $(0 : \maxideal)_M$ since the natural map $M \to (M^\vee)^\vee$ is an isomorphism. Therefore $\dim_{\kk} M^\vee / \maxideal M^\vee = \dim_\kk (0 : \maxideal)_M$, so~the conclusion follows from the previous case.\qedhere
    \end{enumerate}
\end{proof}

\subsection{Apolarity for modules}\label{subsection_apolarity}
We briefly recall apolarity for modules, which is a very useful tool in the
classification and for finding degenerations. A more detailed survey and proofs can be found in \cite[Subsection 4.1]{components}.

Let $F$ be a finitely generated free $S$-module. A submodule $L\subseteq F$ is
\emph{cofinite} if $\dim_\kk F / L < \infty$. For such an $L$ we define the subspace
\[
    L^\perp := \{ \varphi \in F^\vee \colon \varphi(L) = 0 \} \subset F^\vee.
\]
Conversely, for a submodule $M \subset F^\vee$ of finite degree we define the subspace 
\[
    M^\perp := \{f \in F \colon \varphi(f) = 0 \text{ for every } \varphi \in M \} \subset F.
\]
Both subspaces are in fact submodules. Applying $(-)^\vee$ to the natural inclusion $M \subset F^\vee$ yields a surjective map $(F^\vee)^\vee \to M^\vee$. Note that $M^\perp$ is the kernel of the composed map $F \to (F^\vee)^\vee \to M^\vee$ which is still surjective, so we get an isomorphism of vector spaces
$
    F / M^\perp \to M^\vee,
$
which in fact is an isomorphism of modules.

The maps $L \mapsto L^\perp, M \mapsto M^\perp$ give a bijection between cofinite submodules of $F$ and finite degree submodules of $F^\vee$. This correspondence is called \emph{apolarity for modules}, see \cite[Proposition 4.3]{components} for a proof.

There is also a local version of this correspondence, more useful for applications. If we restrict our attention to cofinite submodules of $F$ which yield local quotient, then we can replace $F^\vee$ with a much smaller submodule $F^* \subset F^\vee$. 

Define $F^* $ to be $ \bigoplus_i F_i^\vee \subset F^\vee$. It is the
submodule of $F^\vee$ consisting of functionals that vanish on some
$\maxideal^N F$, where $\maxideal = S_{+}$. Consider $S^* := \kk[y_1, \dots,
y_n]$ with an $S$-module structure given by \emph{contraction}, that is
\begin{equation}\label{eq:contraction}
    x_i \cdot (y_1^{a_1} \dots y_n^{a_n}) =
    \begin{cases}
        y_1^{a_1} \dots y_{i-1}^{a_{i-1}} y_i^{a_i - 1} y_{i+1}^{a_{i+1}}\dots y_n^{a_n} & \text{if } a_i > 0 \\
        0 & \text{otherwise}.
    \end{cases}
\end{equation}
If we fix a basis $e_1, \dots, e_r$ of $F$, then $F^*$ can be identified with
the space $\bigoplus_{j=1}^r S^* e_j^*$. We view $S^*$ purely as a vector
space, although it can be viewed invariantly as a graded dual
of $S$ and has a divided power ring structure, see for
example~\cite[Appendix~A]{Iarrobino_Kanev}.

Finally, we have the local version of apolarity for modules:
    The maps $L \mapsto L^\perp, M \mapsto M^\perp$ give a bijection between
    cofinite submodules of $F$ such that $F/L$ is local with support
    $\{\maxideal\}$ and finite degree submodules of $F^*$. See \cite[Proposition 4.4]{components} for a proof.

\begin{example}[A sketch of classification of algebras]\label{ex:poonen}
    Algebras correspond to cyclic modules, so we take $F = S$ 
    (and thus $F^* = S^*$). The correspondence above gives a
    bijection between the quotient algebras $S/L$ and submodules $L^{\perp}
    \subseteq S^*$. A submodule $L^{\perp}$ is a subspace closed under the
    contraction action~\eqref{eq:contraction}. For small value of $m = \dim_{\kk} S/L = \dim_{\kk}
    L^{\perp}$, these subspaces are fairly easy to classify directly,
    especially if we allow coordinate changes on $S$.

    For example, for $m=1$, we notice that $1\in L^{\perp}$, hence $L^{\perp} = \spann{1}$
    and $S/L = S/\maxideal$.

    For $m = 2$, apart from $1$ we need to have a linear form in $L^{\perp}$, so up to
    coordinate change $L^{\perp} = \spann{1, x_1}$ and $S/L = S/(x_1^2, x_2,
    x_3, \ldots )$. This algebra is isomorphic to $\kk[x]/(x^2)$.

    For $m=3$ we have either a one-dimensional or a two-dimensional space of
    linear forms in $L^{\perp}$. The two-dimensional case
    yields $L^{\perp} = \spann{1, x_1, x_2}$ and $S/L = S/(x_1^2, x_1x_2,
    x_2^2, x_3, \ldots )$, which is an algebra isomorphic to $\kk[x, y]/(x,
    y)^2$. In the one-dimensional case, the space of linear
    forms is spanned by, say, $x_1$. Thus, the
    leading form of any polynomial in $L^{\perp}$ is necessarily a pure power
    of $x_1$, so up to coordinate change, we have $L^{\perp} = \spann{x_1^2 + x_2, x_1, 1}$ or $L^{\perp} =
    \spann{x_1^2, x_1, 1}$. Both choices yield $S/L$ isomorphic to the algebra
    $\kk[x]/(x^3)$.
\end{example}

\subsection{Maps between modules and 111-equations}\label{maps}
The definitions and general results about $111$-algebra introduced in
\cite{concise} are stated for $\kk = \CC$. This is not a necessary assumption (the same proofs work), 
so we state it over $\kk$.

Let $T \in A \otimes B \otimes C$ be a concise tensor. 
Recall that $m = \dim_\kk A = \dim_\kk B = \dim_\kk C$. Each linear endomorphism $X \in \End(A)$ 
yields a new tensor $X \circ_A T := (X \otimes \id_B \otimes \id_C)(T)$. The set of triples 
$(X,Y,Z) \in \End(A) \times \End(B) \times \End(C)$ such that $X\circ_A T = Y \circ_B T = Z \circ_C T$
is a commutative unital subalgebra of $\End(A) \times \End(B) \times \End(C)$, see \cite[Theorem 1.11]{concise}. 
We denote this algebra by $\alg{T}$ and call it the \emph{111-algebra} of $T$.
We say that $T$ is \emph{$111$-abundant} if $\dim_\kk \alg{T} \geq m$ and \emph{$111$-sharp} 
if the equality holds.

We will describe a correspondence between concise $111$-abundant tensors and bilinear maps between modules. 
Let $\mathcal{A}$ be a commutative unital $\kk$-algebra of degree at least $
m$ and let $M, N, P$ be $\mathcal{A}$-modules of degree $m$. An
$\mathcal{A}$-linear map $\varphi
\colon M \otimes_\mathcal{A} N \to P$ is \emph{non-degenerate}, if it is
surjective and for each $m \in M$ and $n\in N$ the restrictions $\varphi(m,
-)$, $\varphi(-, n)$ are nonzero. The linear map $M \otimes N \to M \otimes_\mathcal{A} N$ and 
the $\mathcal{A}$-module homomorphism $M \otimes_\mathcal{A} N \to P$ compose to a linear map 
$M \otimes N \to P$, corresponding to a tensor denoted by $T_\varphi$. The conditions imposed on 
$\varphi$ imply that $T_\varphi$ is concise and $111$-abundant, see \cite[Theorem 5.5]{concise}. 

Let $T \in A \otimes B \otimes C$ be a concise $111$-abundant tensor. 
Each of the spaces $A, B, C$ has a structure of an $\alg{T}$-module, coming from projections from 
$\alg{T}$ to the corresponding factors. We denote these modules by $\modA, \modB, \modC$. 
The map $T_C^{\top} \colon A^\vee \otimes B^\vee \to C$ factors through the natural surjection 
$A^\vee \otimes B^\vee \to \modA^\vee \otimes_{\alg{T}} \modB^\vee$ and induces an $\alg{T}$-module 
homomorphism $\varphi \colon \modA^\vee \otimes_{\alg{T}} \modB^\vee \to \modC$. 
The map $\varphi$ satisfies conditions described above, so it induces a concise $111$-abundant 
tensor $T_\varphi$ which coincides with $T$. For proofs of these claims see~\cite[Theorem 5.4]{concise}. 

Conciseness of $T$ implies that the projections of $\alg{T}$ to $\End(A)$, $\End(B)$, and $\End(C)$ 
are one-to-one, see~\cite[Theorem 1.1]{concise}.
In particular, no non-zero element of the $\alg{T}$ annihilates $\modA, \modB, \modC$ 
or their duals. We call such modules \emph{concise}. This notion of conciseness of modules over 
finite algebras is closely related to the notion of conciseness of modules over polynomial rings.
If we take a surjection from a polynomial ring $S'$ in $\dim_\kk \alg{T} - 1$ variables mapping 
$S'_{\leq 1}$ isomorphically to $\alg{T}$ and consider the $\alg{T}$-modules as $S'$-modules,
then these two notions coincide.

It is easy to determine whether a concise $111$-abundant tensor is $1$-degenerate. 
It is the case precisely when none of the $\alg{T}$-modules $\modA^\vee, \modB^\vee, \modC^\vee$ 
is cyclic, see \cite[Proposition 5.3]{concise}. Therefore, the concise $111$-abundant tensor 
$T_\varphi$ coming from a bilinear map $\varphi \colon M \times N \to P$ is $1$-degenerate precisely 
if $M, N$ are not cyclic and $P$ is not cocyclic.

Every tensor of minimal border rank is also $111$-abundant, this follows by
semicontinuity from the fact that the unit tensor is $111$-abundant, see
also~\cite[Example~4.5, Lemma~5.7]{concise}. For $m \leq 5$ the
converse is true and $111$-abundant tensors are in fact $111$-sharp, see \cite[Theorem 1.6]{concise}. 
It follows that $\dim_\kk \alg{T} = m$ and we can therefore choose a surjection $S \to \alg{T}$ 
which maps $S_{\leq 1}$ isomorphically to $\alg{T}$. 
This allows us to work with $S$-modules instead of $\alg{T}$-modules. 
The fact that $\modA, \modB, \modC$ and their duals are concise as $\alg{T}$-modules translates 
to the fact that the corresponding $S$-modules are concise.

Summing up, each tensor of minimal border rank for $m \leq 5$ gives a surjective non-degenerate map 
of $S$-modules $M \otimes_S N \to P$, where $M, N, P$ are concise $S$-modules of degree $m$. 
If the tensor is additionally $1$-degenerate, then $M, N$ are not cyclic and $P$ is not cocyclic. 
We will use the classification of local concise modules of degree $\leq 4$ to show that there 
are no such maps for $m \leq 4$ and thus there are no $1$-degenerate tensors of minimal border rank 
for $m \leq 4$.

We can decompose bilinear maps between any modules of finite degree to bilinear maps between local modules. It will enable us to utilize results such as Nakayama's lemma and use the classification of local concise $S$-modules of degrees $\leq 4$ obtained in Subsection \ref{classification_local}.

\begin{lemma}\label{maps_decomposition}
    Let $M, N, P$ be $S$-modules of degree $m$. Let $\{\maxideal_1, \dots,
    \maxideal_r\}$ be the union of their supports (see~\S\ref{prelim:modules}
    for definition). Then every map of
    $S$-modules $\varphi\colon M \otimes_S N \to P$ is a direct sum of maps $M_{\maxideal_i} \otimes_S N_{\maxideal_i} \to P_{\maxideal_i}$.
\end{lemma}
\begin{proof}
    We have $M \otimes_S N = \bigoplus_{1 \leq i,j \leq r} M_{\maxideal_i} \otimes_S N_{\maxideal_j}$, so $\varphi$ decomposes as a direct sum of homomorphisms $\varphi_{i,j} \colon M_{\maxideal_i} \otimes_S N_{\maxideal_j} \to P$. The module $M_{\maxideal_i} \otimes_S N_{\maxideal_j}$ is annihliated by sufficiently large powers of $\maxideal_i$ and $\maxideal_j$, so the same holds for the image of $\varphi_{i,j}$. Therefore the image of $\varphi_{i,j}$ is contained in $P_{\maxideal_i} \cap P_{\maxideal_j}$, so $\varphi_{i,j} = 0$ for $i \neq j$ and each $\varphi_{i,i}$ factors through a map $M_{\maxideal_i} \otimes_S N_{\maxideal_i} \to P_{\maxideal_i}$.
\end{proof}

In general, there is no reason for $M_{\maxideal}, N_{\maxideal},
P_{\maxideal}$ to have equal degrees. We will show in Section \ref{degenerate}
that if the map $M \times N \to P$ corresponds to a concise $111$-abundant
tensor $T$ with 111-algebra $\mathcal{A} = \alg{T}$ and $m \leq 4$, then we have $\dim_\kk M_{\maxideal} = \dim_\kk N_{\maxideal} = \dim_\kk P_{\maxideal} = \dim_\kk \mathcal{A}_\maxideal$ and that $M_{\maxideal}, N_{\maxideal}, P_{\maxideal}$ are concise $\mathcal{A}_\maxideal$-modules.

In this setting the map $M_\maxideal \otimes_{\mathcal{A}_\maxideal}
N_\maxideal \to P_\maxideal$ is also non-degenerate, so it
yields a $111$-abundant tensor. In general, there is no reason why should it be $1$-degenerate if the original tensor was $1$-degenerate, but we will show in Section \ref{degenerate} that it is the case for $m \leq 4$. We will need the following weaker result.

\begin{lemma}\label{maps_(co)cyclic}
    Let $M$ be an $S$-module of finite degree. If $M_\maxideal$ is cyclic (respectively, cocyclic) for each maximal ideal, then $M$ is (respectively, cocyclic).
\end{lemma}
\begin{proof}
    We can decompose $M$ as a finite direct sum $M_{\maxideal_1} \oplus \dots
    \oplus M_{\maxideal_r}$ of local modules. Assume that each
    $M_{\maxideal_i}$ is cyclic, hence isomorphic to $S /
    \ann(M_{\maxideal_i})$. There exist $N_i$ such that $\maxideal_i^{N_i}
    \subset \ann(M_{\maxideal_i})$. For each $i \neq j$ the ideals
    $\maxideal_i, \maxideal_j$ are coprime, so $\maxideal_i^{N_i},
    \maxideal_j^{N_j}$ are also coprime and consequently
    $\ann(M_{\maxideal_i}), \ann(M_{\maxideal_j})$ are coprime. By Chinese
    remainder theorem $M$ is isomorphic to $S / \bigcap_i
    \ann(M_{\maxideal_i})$, so it is cyclic. If each $M_{\maxideal_i}$ is cocyclic, then $(M_{\maxideal_1}^\vee \oplus \dots \oplus M_{\maxideal_r}^\vee)^\vee = M$ is also cocyclic.
\end{proof}

Below we give technical lemmas that will be used in the proof that for $m \leq
4$ there are no bilinear maps that could yield $1$-degenerate minimal border
rank tensors.

\begin{lemma}\label{maps_cyclic}
    Let $M, N, P$ be concise $S$-modules of degree $m$ and let $\varphi \colon M \times N \to P$ be a bilinear map of $S$-modules. If there exists an element $m \in M$ such that the map $\varphi(m, -) \colon N \to P$ is surjective, then $M$ is cyclic.
\end{lemma}
\begin{proof}
    If $f \in \ann(m)$, then for every $n \in N$ we have $f \varphi(m, n) =
    \varphi(fm, n) = 0$, so $\ann(m)\subset \ann(\varphi(m, N))$. The map
    $\varphi(m, -)$ is assumed to be surjective, so $\ann(m)\subset \ann(P)$.
    The module $P$ is concise, so $S_{\leq 1} \cap \ann(P)= 0$ which implies
    that $S_{\leq 1} \cap \ann(m)= 0$. It follows that $S_{\leq 1}m$ is
    $m$-dimensional, so $Sm = M$ and so $m$ generates $M$.
\end{proof}

\begin{lemma}\label{maps_nakayama}
    Let $N, P$ be local $S$-modules of degree $m$, supported at the same maximal ideal $\maxideal \subset S$, and let $\varphi\colon N \to P$ be a map of $S$-modules. If the induced map $\overline{\varphi} \colon N / \maxideal N \to P / \maxideal P$ is surjective, then $\varphi$ is surjective.
\end{lemma}
\begin{proof}
    The map $\overline{\varphi}$ is surjective, so $P = \varphi(P) + \maxideal
    P$ and we conclude by Nakayama's Lemma~\ref{ref:Nakayama:lem}.
\end{proof}

\begin{lemma}\label{maps_not_surjective}
    Let $M, N, P$ be local concise $S$-modules of degree $m$ supported at the same maximal ideal $\maxideal \subset S$ and let $\varphi \colon M \times N \to P$ be a bilinear surjective map of $S$-modules. If $P$ is cyclic, then $M, N$ are cyclic as well.
\end{lemma}
\begin{proof}
    Assume that $M$ is not cyclic. By Lemmas \ref{maps_nakayama} and
    \ref{maps_cyclic} for each $m \in M$ the induced map
    $\overline{\varphi}(m, -) \colon N/\maxideal N \to P/\maxideal P$ is not
    surjective. The module $P$ is cyclic, so by Lemma \ref{cyclic} the space
    $P / \maxideal P$ is 1-dimensional, so all of these maps are in fact zero.
    It follows that the image of $\varphi$ is contained in $\maxideal P$. By
    Nakayama's Lemma~\ref{ref:Nakayama:lem} the submodule $\maxideal P$ is not
    the whole $P$, so $\varphi$ is not surjective. The same argument shows that $N$ is not cyclic.
\end{proof}

\begin{corollary}\label{maps_symmetry}
    Let $M, N, P$ be concise $S$-modules of degree and let $\varphi \colon M \times N \to P$ be a non-degenerate bilinear map of $S$-modules. Assume that for every maximal ideal $\maxideal \subset S$ the modules $M_\maxideal, N_\maxideal, P_\maxideal$ have equal dimensions. If at least one of $M, N, P$ is cyclic or cocyclic, then $T_\varphi$ is $1_*$-generic.
\end{corollary}
\begin{proof}
    Assume that $T_\varphi$ is $1$-degenerate. By the general characterization
    of concise 111-abundant tensors we know that $P$ cannot be cocyclic and $M, N$
    cannot be cyclic. If $P$ is cyclic, then each $P_\maxideal$ is cyclic
    because $(S / \ann(P))_\maxideal = S_\maxideal / (\ann(P))_\maxideal$, so
    by Lemma \ref{maps_not_surjective} in particular each $M_\maxideal$ is
    cyclic, so by Lemma \ref{maps_(co)cyclic} the module $M$ is cyclic and we
    get a contradiction. Therefore $P$ is neither cyclic nor cocyclic. The
    permuted map $\varphi'\colon M\times P^{\vee}\to N^{\vee}$ is
    nondegenerate as well. Repeating the above argument, we get that
    $N$ is not cyclic. The argument for $M$ is the same.
\end{proof}

\subsection{A special case of Kronecker's normal form}\label{ssec:Kronecker}

For certain cases of classification below (in~\S\ref{ssec:m5}), a very special case of Kronecker's
normal form will be useful. It seems nontrivial to find a reference over an
arbitrary field and, moreover, we need only a little, so we prove it below.
\newcommand{\Mats}[1]{\mathbb{M}_{#1}}%
\newcommand{\tr}{\operatorname{tr}}%

Consider vector spaces $B', C'$ and a subspace $V\subseteq
\Mats{b\times c}$ of matrices. For
fixed dimensions $\dim V$, $b, c$ we can ask what are the orbits
of $\GL_b\times \GL_c$ acting on the Grassmannian $\Gr(\dim V,
\Mats{b\times c})$. In this section we recall the answer for very small cases.

The \emph{trace pairing} $\Mats{b\times c}\times \Mats{b\times c}\to \kk$ given by
the formula $(X, Y)\mapsto \tr(XY^{\top})$ is nondegenerate for every $b$, $c$
and any field $\kk$. It yields an isomorphism $\Gr(\dim V, \Mats{b\times c})
\simeq \Gr(bc - \dim V, \Mats{b\times c})$ given by $V\mapsto V^{\perp}$.
The $\GL_b\times \GL_c$ orbits on both spaces correspond.

\begin{example}\label{ex:2x2}
    Consider the case $b=c=2$, $\dim V = 3$. By the trace pairing, we reduce
    to considering $1$-dimensional subspaces of $\Mats{2\times 2}$. These are
    classified by rank with two isomorphism classes.
    \[
        \begin{bmatrix}
            1 & 0 \\
            0 & -1 \\
        \end{bmatrix},
        \begin{bmatrix}
            0 & 1 \\
            0 & 0 \\
        \end{bmatrix},
        \begin{bmatrix}
            0 & 0 \\
            1 & 0 \\
        \end{bmatrix}
        \text{\quad or \quad}
        \begin{bmatrix}
            1 & 0 \\
            0 & 0 \\
        \end{bmatrix},
        \begin{bmatrix}
            0 & 1 \\
            0 & 0 \\
        \end{bmatrix},
        \begin{bmatrix}
            0 & 0 \\
            1 & 0 \\
        \end{bmatrix}.
    \]
\end{example}

\newcommand{\boldedcust}[2]{\mathbf{#1}\marginnote{$\mathbf{#2}$}}
\newcommand{\bolded}[1]{\mathbf{#1}\marginnote{$\mathbf{#1}$}}
\begin{example}\label{ex:2x3}
    Consider the case $b=2, c=3, \dim V = 4$. By the trace pairing, we reduce
    to $2$-dimensional subspaces $W\subseteq \Mats{2\times 3}$. In the classical
    language, these are Kronecker's pencils of matrices. We classify them below.

    \begin{enumerate}
        \item Assume that $W$ contains only matrices of rank $2$. Let
            $w_1, w_2$ be a basis of $W$. If the kernels of $w_1$,
            $w_2$ intersect non-trivially, then the space $W$ lies in $\Mats{2\times
            2}$, so it has a rank one element. Hence the kernels are disjoint. Change the basis
            $e_1, e_2, e_3$ so that $w_1 e_1 = 0, w_2 e_3 = 0$.

            Take $f_1 = w_1 e_3, f_2 = w_2 e_1$.
            The matrices $w_1, w_2$ have rank 2, so $f_1, f_2$ are non-zero.
            Suppose that $f_1, f_2$ are linearly dependent. Take a nonzero $w'\in \spann{w_1,
            w_2}$ such that $w'e_2 = 0$, then the image
            of $w'$ is $\spann{f_1} = \spann{f_2}$, so $w'$ has rank one, a contradiction. Therefore $f_1, f_2$ are linearly independent.

            The pairs of vectors $f_1,  w_1 e_2$ and $f_2, w_2 e_2$ are
            linearly independent, so after rescaling $f_1, f_2$ and adding
            multiples of $e_1, e_3$ to $e_2$ we can assume that $w_1 e_2 =
            f_2, w_2 e_2 = f_1$. In bases $e_1$, $e_2$, $e_3$ and $f_1$,
            $f_2$ we obtain the subspace $\bolded{W_{11}} = \spann{w_1, w_2}$,
            where
            \[
                w_1 = \begin{bmatrix}
                    0 & 1 & 0 \\
                    0 & 0 & 1 \\
                \end{bmatrix},\quad
                w_2 = \begin{bmatrix}
                    1 & 0 & 0 \\
                    0 & 1 & 0 \\
                \end{bmatrix}.
            \]
        \item Assume that $W$ contains both rank 1 and rank 2 matrices. Let $w_1, w_2$ be a basis of $W$ such that $w_1$ has rank 2 and $w_2$ has rank 1.
            \begin{enumerate}
                \item Assume that the kernels of $w_1, w_2$ intersect
                    trivially. Change the basis so that $\ker w_1 =
                    \spann{e_2}$, $\ker w_2 = \spann{e_1, e_3}$, and $w_1 e_1
                    = w_2 e_2$. In bases $e_1$, $e_2$, $e_3$ and  $w_1
                    e_1, w_1 e_3$ we obtain the subspace $\bolded{W_{12}} = \spann{w_1, w_2}$,
            where
                    \[
                        w_1 = \begin{bmatrix}
                            1 & 0 & 0 \\
                            0 & 0 & 1 \\
                        \end{bmatrix},\quad
                        w_2 = \begin{bmatrix}
                            0 & 1 & 0 \\
                            0 & 0 & 0 \\
                        \end{bmatrix}.
                    \]
                \item Assume that the kernels of $w_1, w_2$ intersect
                    non-trivially.
                    Change the basis so that $\ker w_1 = \spann{e_1}$ and
                    $\ker w_2 = \spann{e_1, e_2}$. Let $f_1 = w_1 e_2, f_2 =
                    w_1 e_3, f_3 = M_3 e_3$. The vectors $f_1, f_2$ form a
                    basis of $\kk^2$, so $f_3$ is a linear combination of
                    $f_1, f_2$.

                    If there exists an element $\lambda \in \kk$ such that
                    $f_2 + \lambda f_1$ is a multiple of $f_3$, then after
                    adding a multiple of $e_2$ to $e_3$ and rescaling $w_2$ we
                    can assume that $f_2 = f_3$. In bases $e_1$, $e_2$, $e_3$
                    and  $f_1, f_2$ we obtain the subspace
                    $\bolded{W_{13}} = \spann{w_1, w_2}$, where
                    \[
                        w_1 - w_2 = \begin{bmatrix}
                            0 & 1 & 0 \\
                            0 & 0 & 0 \\
                        \end{bmatrix},\quad
                        w_2 = \begin{bmatrix}
                            0 & 0 & 0 \\
                            0 & 0 & 1 \\
                        \end{bmatrix}.
                    \]

                    If there is no such element $\lambda \in \kk$, then $f_3$
                    is a multiple of $f_1$. An analogous argument yields the
                    subspace $\bolded{W_{14}} = \spann{w_1, w_2}$ where
                    \[
                        w_1 = \begin{bmatrix}
                            0 & 1 & 0 \\
                            0 & 0 & 1 \\
                        \end{bmatrix},\quad
                        w_2 = \begin{bmatrix}
                            0 & 0 & 1 \\
                            0 & 0 & 0 \\
                        \end{bmatrix}.
                    \]
            \end{enumerate}

        \item Assume that $W$ contains only matrices of rank at most $1$. Let $w_1, w_2$ be a basis of $W$.
            \begin{enumerate}
                \item Assume that $\ker w_1 \neq \ker w_2$. Change the basis
                    $e_1, e_2, e_3$ so that $\ker w_1 = \spann{e_1, e_2}$ and
                    $\ker w_2 = \spann{e_1, e_3}$. Each linear combination of
                    $w_1, w_2$ has rank one, so the vectors $w_1 e_3, w_2 e_2$
                    are non-zero and linearly dependent. After bases changes
                    we obtain the subspace $\bolded{W_{15}}$ spanned by
                    \[
                        \begin{bmatrix}
                            0 & 0 & 0 \\
                            0 & 0 & 1 \\
                        \end{bmatrix},\quad
                        \begin{bmatrix}
                            0 & 0 & 0 \\
                            0 & 1 & 0 \\
                        \end{bmatrix}.
                    \]
                \item Assume that $\ker w_1 = \ker w_2$. After change of bases, we
                    obtained $\bolded{W_{10}}$ the subspace spanned by
                    \[
                        \begin{bmatrix}
                            0 & 0 & 1 \\
                            0 & 0 & 0 \\
                        \end{bmatrix},\quad
                        \begin{bmatrix}
                            0 & 0 & 0 \\
                            0 & 0 & 1 \\
                        \end{bmatrix}.
                    \]
            \end{enumerate}
    \end{enumerate}
\end{example}

\section{Classification}\label{classification}
\subsection{Local concise modules}\label{classification_local}
Let $S = \kk[x_1, \ldots ,x_{m-1}]$.
In this section we classify \emph{concise} local $S$-modules $M$ of degree $m \leq 5$, up to an affine change of variables and duality.
We fix the support of the module at $\maxideal = S_{\geq 1}$. The affine
changes of variables that preserve $\maxideal$ are precisely the linear
changes of variables in $S$, so two such modules are equivalent if one becomes isomorphic to the other after a linear change of variables.

The classification up to degree three is nearly trivial and degree four is
also quite approachable. We give all the details, because we will use the
result to classify local concise modules of degree five and to prove that
there are no $1$-degenerate tensors of degree at most four.

\subsubsection{Case $m=1$.}
    In this case $S = \kk$ and the only $\kk$-module of degree $1$ is $\kk$.

\subsubsection{Case $m=2$.}
    Let $M$ be an $S$-module of degree $2$. Consider the space $M /\maxideal M$.
    Since $M$ is nonzero, by Nakayama's Lemma~\ref{ref:Nakayama:lem} we have $\dim_{\kk} M/\maxideal
    M \geq 1$.
    If $\dim_\kk M /\maxideal M = 2$, then $\maxideal M = 0$, so $M = \kk^2$,
    which is not concise. If $\dim_\kk M /\maxideal M = 1$, then by Lemma
    \ref{cyclic} the module $M$ is cyclic. By Corollary~\ref{corr_cyclic} and
    Example~\ref{ex:poonen} we get that the only concise cyclic $S$-module is the one coming from the $\kk$-algebra $\kk[x]/(x^2)$.

\subsubsection{Case $m=3$.}
    To deal with the cyclic and cocyclic modules, we use
    Corollary~\ref{corr_cyclic} and the classification in
    Example~\ref{ex:poonen}.
    There are two concise cyclic modules, coming from the $\kk$-algebras
    $\kk[x]/(x^3)$ and $\kk[x,y]/(x,y)^2$, and two concise cocyclic modules,
    dual to these ones. By Lemma~\ref{cyclic}, $(\kk[x]/(x^3))^\vee$ is cyclic
    so we discard it, but $(\kk[x,y]/(x,y)^2)^\vee$ is a new concise module.

    Let $M$ be neither cyclic nor cocyclic. By Lemma~\ref{cyclic} we have
    \[
        \dim_\kk M / \maxideal M \geq 2,\quad\mbox{and}\quad \dim_\kk (0 : \maxideal)_M \geq 2
    \]
    so by dimensional reasons there exists a minimal generator of $M$ which
    lies in the socle. By Lemma~\ref{redundant} there exists an $S$-module $N$
    of degree $2$ such that $M = \kk \oplus N$. By classification for $m = 1,
    2$ and Lemma~\ref{not_concise}, the module $M$ is not concise.

\subsubsection{Case $m=4$.}\label{deg4}
    As above, we first deal with cyclic and cocyclic modules. The concise
    cyclic modules come from $\kk$-algebras $\kk[x]/(x^4), \kk[x,y]/(x^2, xy,
    y^3), \kk[x,y]/(x^2, y^2)$ and $\kk[x,y,z]/(x,y,z)^2$, by~\cite{poonen} or
    arguing similarly as in Example~\ref{ex:poonen}. There are also two new concise cocyclic modules $(\kk[x,y]/(x^2, xy, y^3))^\vee$ and $(\kk[x,y,z]/(x,y,z)^2)^\vee$.

    Let $M$ be an $S$-module of degree $4$ which is not cyclic or cocyclic.
    Lemma~\ref{cyclic} yields $\dim_\kk M / \maxideal M \geq 2$ and $\dim_\kk (0 :
        \maxideal)_M \geq 2$.
    If there exists a minimal generator from the socle, then by the
    classification for $m=1,2,3$ and Lemma~\ref{not_concise}, the considered
    module cannot be concise. We thus obtain that $(0 : \maxideal)_M \subset
    \maxideal M$. By the above, we have $\dim_{\kk} \maxideal M \leq 2$ while
    $\dim_{\kk} (0:\maxideal)_M\geq 2$, so
    \[
        (0 : \maxideal)_M = \maxideal M
    \]
    is a two-dimensional subspace.
    Choose a basis $e_1, e_2, e_3, e_4$ of $M$ such that $e_3, e_4$ span this
    subspace. In this basis, the multiplications by $x_1, x_2, x_3$ on $M$
    have matrices of the form
        \[
        \begin{bmatrix}
            0 & 0 & 0 & 0 \\
            0 & 0 & 0 & 0 \\
            * & * & 0 & 0 \\
            * & * & 0 & 0 \\
        \end{bmatrix},
        \]
        so they yield a $3$-dimension subspace of $2 \times 2$ matrices.
        Using Example~\ref{ex:2x2}, we obtain two possible modules
        $\boldedcust{N_7}{N_{7, 8}}$,
        $\mathbf{N_8}$. They are isomorphic to their duals.

        \subsubsection{Case $m=5$}\label{ssec:m5}
We start with cyclic and cocyclic modules. Corollary~\ref{corr_cyclic} and the
classification from~\cite{poonen} yields modules corresponding to the following $\kk$-algebras:
\begin{equation}\label{eq:localalgebras}
\begin{array}{c|c}
    \bolded{M_1} & \kk[x]/(x^5) \\
     \bolded{M_2} & \kk[x, y]/(x^2, xy, y^4) \\
     \bolded{M_3} & \kk[x,y]/(x^2 + y^3, xy) \\
     \bolded{M_4} & \kk[x, y]/(xy, x^3, y^3) \\
     \bolded{M_5} & \kk[x,y]/(x^2, xy^2, y^3) \\
     \bolded{M_6} & \kk[x,y,z]/(x^2, y^2, xy, xz, yz, z^3) \\
     \bolded{M_7} & \kk[x,y,z]/(x^2, y^2, z^2, xy, xz) \\
     \bolded{M_8} & \kk[x,y,z]/(x^2, y^2, xz, yz, xy + z^2) \\
     \bolded{M_9} & \kk[x, y, z, w]/(x, y, z, w)^2 \\
\end{array}
\end{equation}
All of these modules are End-closed by Corollary \ref{corr_cyclic}. By
duality, this also concludes the case of cocyclic modules.

We move to the case of modules concise $S$-modules of the form $M = \kk \oplus N$. By the classification of local modules of degree $\leq 4$ and Lemma~\ref{not_concise} we get that $M$ can be concise only if $N$ comes from one of the modules $N_7, N_8$.
In this case the action of $x_1, x_2, x_3$ corresponded to $3$ linearly
independent matrices in the space of $2 \times 2$ matrices. A concise module
is obtained only if $x_4$ acts by a matrix which completes it to a basis of
the space of $2 \times 2$ matrices. In both cases, we obtain a concise module
$\bolded{M_{10}}$ whose multiplication tensor after a linear change of variables is
\[
\begin{bmatrix}
    x_0 & 0 & 0 & 0 & 0 \\
    0 & x_0 & 0 & 0 & 0 \\
    0 & 0 & x_0 & 0 & 0 \\
    0 & x_1 & x_3 & x_0 & 0 \\
    0 & x_2 & x_4 & 0 & x_0 \\
\end{bmatrix}
\]
The fact that $M_{10}$ is annihilated by $\maxideal^2$ implies that
$M_{10}$ is End-closed.

If a local concise $S$-module $M$ is not of one of these forms, then by
Lemma~\ref{cyclic} and Lemma~\ref{redundant} we have
\[
    \dim_\kk M / \maxideal M \geq 2,\quad \dim_\kk (0 : \maxideal)_M \geq
    2,\quad (0 : \maxideal)_M \subset \maxideal M.
\]
If the inclusion $(0 : \maxideal)_M \subset \maxideal M$ is proper, then
$\maxideal^2 M \neq 0$. There are three further subcases:
\begin{enumerate}
    \item\label{case1} $(0 : \maxideal)_M = \maxideal M$ are subspaces of dimensions $3$ and $\maxideal^2 M = 0$,
    \item\label{case2} $(0 : \maxideal)_M = \maxideal M$ are subspaces of dimensions $2$ and $\maxideal^2 M = 0$,
    \item\label{case3} $(0 : \maxideal)_M \subsetneq \maxideal M$ are subspaces of dimensions $2, 3$ respectively, $\dim_\kk \maxideal^2 M = 1$ and $\maxideal^3 M = 0$.
\end{enumerate}
The modules from the subcase \ref{case1} are exactly the dual modules of the
modules from the subcase \ref{case2} because taking the dual module
corresponds to transposing the matrices describing the actions of variables
$x_1, x_2, x_3, x_4$. Therefore there are in fact only two genuinely new subcases.

We start by investigating the \textbf{subcase~\ref{case2}}. Choose a basis $e_1, e_2,
e_3, e_4, e_5$ of the underlying vector space of $M$ such that $e_4, e_5$
spans $(0 : \maxideal)_M = \maxideal M$. The multiplications by $x_1,  \ldots , x_4$ correspond to matrices of the form
\[
\begin{bmatrix}
    0 & 0 & 0 & 0 & 0 \\
    0 & 0 & 0 & 0 & 0 \\
    0 & 0 & 0 & 0 & 0 \\
    * & * & * & 0 & 0 \\
    * & * & * & 0 & 0 \\
\end{bmatrix}.
\]
We require $M$ to be concise, so the obtained $2 \times 3$ matrices  span a 4-dimensional subspace $W \subset \mathrm{M}_{2
\times 3}(\kk)$. We employ Example~\ref{ex:2x3}. From $W_{11}, \ldots , W_{15}$, we get new modules
$\boldedcust{M_{11}}{M_{11, \ldots , 15}}, \ldots , \mathbf{M_{15}}$ which indeed satisfy the conditions of subcase
\ref{case2}, so in particular they are not isomorphic to any of previously
obtained modules. They are End-closed because they are annihilated by
$\maxideal^2$ and none of them is self-dual as their dual modules are not
generated by two elements.

We move to the \textbf{subcase~\ref{case3}}. We will use apolarity for modules, see Subsection \ref{subsection_apolarity}.
We have $\dim_\kk M / \maxideal M = 2$ and $\dim_\kk M^\vee / \maxideal M^\vee
= \dim_\kk (0 : \maxideal)_M = 2$, so there exist $\sigma_1, \sigma_2 \in S^*
e_1^* \oplus S^* e_2^*$ such that $M^{\vee} = S\sigma_1 +  S\sigma_2$ and this is a minimal presentation.

We have $\dim_{\kk} (0 : \maxideal^2)_{M^{\vee}} = \dim_{\kk} M/\maxideal^2 M =
m-1$, so we may assume that $\sigma_2$ has degree at most one.
We have $\dim_{\kk} \maxideal^2 M^{\vee} = \dim_{\kk} M/(0 : \maxideal^2)_M =
1$, so we may assume that $\sigma_1 = qe_1^* + \ell_{1,2} e_2^*$ with $q$ of
degree two, $\ell_{1,2}$ of degree one.
Again by $\dim_{\kk} (0 : \maxideal^2)_{M^{\vee}} = m-1$, we can assume that
the degree-two part of $q$ is equal to
$y_1^2$.

In summary, we obtain the following normal form
\[
    \begin{split}
        \sigma_1 = (y_1^2 + \ell_{1,1})e_1^* &+ \ell_{1,2} e_2^* \\
        \sigma_2 = \ell_{2,1}e_1^* &+ \ell_{2,2}e_2^*,
    \end{split}
\]
where $\ell_{i,j}\in \spann{y_1, \ldots , y_4}$ are linear forms in dual
variables.

Let us summarize the transformation which we have at out disposal. We may
\begin{itemize}
    \item add a multiple of $\sigma_2$ to $\sigma_1$, add a multiple of
        $x_1\sigma_1 = y_1 e_1^*$ to $\sigma_1$ and $\sigma_2$, multiply
        $\sigma_1, \sigma_2$ by non-zero constants: this corresponds to different choices of generators $\sigma_1, \sigma_2$,
    \item add some multiple of $e_1$ to $e_2$: this corresponds to different choices of the basis $e_1, e_2$,
    \item perform linear changes $\varphi$ of variables $y_1, y_2, y_3, y_4$
        provided that $\varphi(y_1)\in \kk y_1$, this corresponds to the
        changes of the (dual) variables $x_1, \ldots , x_4$,
    \item interchange $\ell_{1,2}$ and $\ell_{2,1}$: this corresponds to
        taking the dual module, see
        also~\cite[Theorem~2.8]{Kunte__Gorenstein_modules_of_finite_length},
        \cite[Theorem~4.3]{Wojtala}.
\end{itemize}
These operations do not affect End-closedness or conciseness.
\begin{enumerate}
    \item Assume that $M$ is not End-closed. Then there is no linear form $f$
        such that $y_1^2 - f \in \ann(M)$. This means that every linear form
        which annihilates $\ell_{2,1}$, annihilates also $\ell_{1,1}$, so
        $\ell_{1,1}$ is a scalar multiple of $\ell_{2,1}$. By
        adding some multiple of $\sigma_2$ to $\sigma_1$, we may assume that
        $\ell_{1,1} = 0$.

        By conciseness of $M$, the forms $y_1$, $\ell_{1,2}$, $\ell_{2,1}$,
        $\ell_{2,2}$ are linearly independent, so we can change variables to
        obtain $\ell_{1,2} = y_2$, $\ell_{2,1} = y_3$, $\ell_{2,2} = y_4$. We
        obtain the module
        \[
            \bolded{M_{20}} = \frac{Se_1 \oplus Se_2}{(y_1^2 e_1^* + y_2
            e_2^*,\, y_3 e_1^* + y_4 e_2^*)^\perp} = \frac{Se_1 \oplus
            Se_2}{(y_{4}e_{1},\:y_{3}e_{1},\:y_{2}e_{1}-y_{4}e_{2},\:y_{2}e_{2},\:y_{1}e_{2},\:y_{1}^{2}e_{1}-y_{3}e_{2})}
        \]
    \item Assume that $\ell_{2,2}$ is a scalar multiple of $y_1$.
        By conciseness of $M$, the forms $y_1$, $\ell_{1,1}$, $\ell_{1,2}$,
        $\ell_{2,1}$ are linearly independent, so we may assume
        $\ell_{1,1} = y_2$, $\ell_{1,2} = y_4$, $\ell_{1,3} = y_3$.

        Rescaling $\sigma_1$, $\sigma_2$, $y_1$, we reduce to two cases:
        $\ell_{2,2} = y_1$, $\ell_{2,2} = 0$. We have two corresponding
        modules
        \[
            \bolded{M_{16}} = \frac{Se_1 \oplus Se_2}{((y_1^2 + y_2)e_1^* +
            y_4 e_2^*,\, y_3 e_1^*)^\perp}
        \]
        and 
        \[
            \bolded{M_{17}} = \frac{Se_1 \oplus Se_2}{((y_1^2 + y_2)e_1^* +
            y_4 e_2^*,\, y_3 e_1^* + y_1 e_2^*)^\perp}.
        \]
        Both modules are concise and End-closed.
    \item Assume that $\ell_{2,2}$ is not a scalar multiple of $y_1$ and
        assume that $M$ is End-closed. Change the variables $x_2, x_3, x_4$
        so that $x_1^2 - x_2$ annihilates $M$. It follows that the only form
        $\ell_{i,j}$ not annihilated by $x_2$ is $\ell_{1,1}$.

        Up to linear change of coordinates we may assume that $\ell_{2,2} =
        y_4$.  Consider $\ell_{1,2}$, $\ell_{2,1}$. They are both annihilated
        by $x_2$. By adding multiple of $y_1e_1^*$ to $\sigma_2$ and passing
        to the dual module we may assume that $x_1$ annihilates $\ell_{1,2}$,
        $\ell_{2,1}$. Next, by adding a multiple of $\sigma_2$ to $\sigma_1$
        and linear coordinate changes in $e_1$, $e_2$, we may assume that
        $x_4$ annihilates $\ell_{1,2}$, $\ell_{2,1}$. So they are both
        multiples of $y_3$. Finally, we may assume $\ell_{1,1} = y_2$.
        At this point we have
        \[
            \sigma_1 = (y_1^2 + y_2) e_1^* + \lambda_{1,2} y_3 e_2^*, \quad
            \sigma_2 = \lambda_{2,1} y_3 e_1^* + y_4 e_2^*,
        \]
        for some scalars $\lambda_{1,2}$, $\lambda_{2,1}$, at least one of them
        nonzero. Up to taking the dual module and rescaling, we have
        $\lambda_{2,1} = 1$. Further rescaling of variables and $e_2$ reduces
        us to two cases: $\lambda_{2,1} = 1$ and $\lambda_{2,1} = 0$. The corresponding
        modules
        \[
            \bolded{M_{18}} = \frac{Se_1 \oplus Se_2}{((y_1^2 + y_2)e_1^*,\,
            y_3 e_1^* + y_4 e_2^*)^\perp}
        \]
        and
        \[
            \bolded{M_{19}} = \frac{Se_1 \oplus Se_2}{((y_1^2 + y_2)e_1^* +
            y_3 e_2^*,\, y_3 e_1^* + y_4 e_2^*)^\perp}.
        \]
        are concise and End-closed.
\end{enumerate}

Every one of the modules $M_{16}$, $M_{17}$, $M_{19}$ is $M_{20}$ is
self-dual, because swapping $\ell_{1,2}$ and $\ell_{2,1}$ leading to an
equivalent system $\sigma_1$, $\sigma_2$.
The module $M_{18}$ is not isomorphic to its dual, because the module
$M_{18}$ admits a quotient module $(Se_1 \oplus Se_2)/((y_1^2 +
y_2)e_1^*)^{\perp}$, which has degree three and which is not annihilated by
$\maxideal^2$. The dual module $M_{18}^{\vee}$ admits no such submodule.
The modules $M_{16}, \ldots , M_{20}$ are
pairwise non-equivalent, which can be seen most easily by considering the
dimension of the stabilizer, see Section~\ref{ssec:stabilizer} below. The
direct proof is also quite easy, but we leave it to the reader.
We conclude subcase~\ref{case3} and the whole classification for $m=5$.

\subsection{Concise modules}\label{classification_concise}
We retrieve the classification of all concise modules of degree $m \leq 5$ from the local case and determine which of them are End-closed. Then we translate this result into the classification of $1_A$-generic concise tensors satisfying the $A$-Strassen's equations and determine which of them have minimal border rank. We classify modules up to affine changes of variables and duality and tensors up to the action of $\operatorname{GL}_m(\kk)^3 \rtimes \Sigma_3$.

Before we begin, we fix the convention on how one obtain concise modules from
an algebra. When all local modules are cyclic or cocyclic, this procedure is unique up to 
isomorphism by Corollary~\ref{cor_unique}. The convention allows us to pick concrete
representatives systematically. Corollary~\ref{cor_unique} does not directly
apply for the cases using modules $N_7, N_8$ from Subsection~\ref{deg4}, but we handle
these cases manually.

Let $\maxideal_1, \ldots , \maxideal_m\subseteq S$ denote the vertices of the standard
$(m-1)$-simplex. For a finite algebra $R$ with $r$ maximal ideals we will
fix an
isomorphism $R \simeq S/I$ such that $S/I$ is supported at $\maxideal_m$,
\ldots ,$\maxideal_{m-r+1}$. The details are best illustrated by the example below.

\begin{example}\label{example_sum}
    Consider the $\kk[x_1, x_2, x_3, x_4]$-module $\kk[x_1]/(x_1)^3 \oplus \kk[x_3]/(x_3-1)^2$. 
    The first factor comes from the $3$-dimensional algebra $\kk[x]/(x)^3$, so we set $x_1, x_2$ 
    to act as multiplication by $x$ and $x^2$. The second factor comes from the $2$-dimensional 
    algebra $\kk[x]/(x-1)^2$, so we set $x_3$ to act as multiplication by $x-1$. The last variable 
    has to act in such a way that we can retrieve endomorphisms corresponding to the identity 
    on one of the summands and zero on the other one. Therefore, we declare $x_4$ to act on the 
    direct sum by $x_4(m_1 + m_2) = m_2$. We retrieve the identity on the first summand 
    as $1 - x_4$. Note that this module is concise. By Corollary~\ref{cor_unique}, we conclude 
    that this is essentially the only way to make it concise and that it is End-closed.
\end{example}

We proceed with the classification for $m = 5$ and here we give a proof of
Theorem~\ref{classification_theorem}.
Let $s$ denote the number of maximal ideals in the support of our module,
see~\S\ref{prelim:modules}. We will use the classifications from Subsection \ref{classification_local}. We do not include the computations showing which modules are simultaneously cyclic and cocyclic. It can be checked using Lemma~\ref{cyclic} combined with the fact that $\ann(N)= \ann(N^\vee)$. The modules without labels are dual to other modules on the list.

\subsubsection{Support of cardinality $s = 1$.}
 This is exactly the classification of local concise modules of degree $5$ from Subsection \ref{classification_local}.

 \subsubsection{Support of cardinality $s = 2$.}\label{sssec:s2}

    The degree decomposition $5 = 3 + 2$ yields
    \[
    \begin{array}{c|c}
        \bolded{M_{2,1}} & \kk[x_1]/(x_1)^3 \oplus \kk[x_3]/(x_3-1)^2 \\
        \bolded{M_{2,2}} & \kk[x_1, x_2]/(x_1, x_2)^2 \oplus \kk[x_3]/(x_3 - 1)^2 \\
            & (\kk[x_1, x_2]/(x_1, x_2)^2)^\vee \oplus \kk[x_3]/(x_3-1)^2 \\
    \end{array}
    \]
    where $x_4$ acts on the direct sum by $x_4(m_1 + m_2) = m_2$. The degree decomposition $5 = 4 + 1$ using only cyclic and cocyclic modules gives
    \begin{equation}\label{eq:modulesM23M26}
    \begin{array}{c|c}
        \bolded{M_{2, 3}}  & \kk[x_1]/(x_1)^4 \oplus S / \maxideal_4 \\
        \bolded{M_{2, 4}} & \kk[x_1, x_2]/(x_1^2, x_1 x_2, x_2^3) \oplus S / \maxideal_4 \\
         & (\kk[x_1, x_2]/(x_1^2, x_1 x_2, x_2^3))^\vee \oplus S / \maxideal_4 \\
        \bolded{M_{2, 5}} & \kk[x_1, x_2]/(x_1^2, x_2^2) \oplus S / \maxideal_4 \\
        \bolded{M_{2, 6}} & \kk[x_1, x_2, x_3]/(x_1, x_2, x_3)^2\oplus S / \maxideal_4 \\
         & (\kk[x_1, x_2, x_3]/(x_1, x_2, x_3)^2)^\vee \oplus S / \maxideal_4 \\
    \end{array}
\end{equation}
    The modules without labels are dual to other modules on the list because in the other summand 
    is simultaneously cyclic and cocyclic. We directly check that they are concise, so they are 
    also End-closed by Corollary~\ref{cor_unique}.
    
    There are also two modules $N_7, N_8$ of degree $4$ described in
    Subsection~\ref{deg4}. We directly compute that $\End(N_7)$, $\End(N_8)$ are $5$-dimensional, spanned by the identity and multiplication by matrices with non-zero elements only in the top right $2\times2$ minor.
    Let $N$ be an $S$-module restricting to $N_7$ or to $N_8$ and let $M = N \oplus S / \maxideal_4$. Consider the map $S_{\leq 1} \to \End(N)$. It has rank at least $4$ because $N_i$ are concise.

    If $S_{\leq 1}\to \End(N)$ has rank $4$, then we can assume (after a linear change of
            variables in $S$) that $x_4$ spans its kernel and $x_1, x_2, x_3$
            act as in Subsection \ref{deg4}. This yields two new modules
            $\bolded{M_{2,7}}$ and $\bolded{M_{2,8}}$ corresponding to $N_7$
            and $N_8$. It can be checked directly that they are concise,
            End-closed and isomorphic to their dual modules.

        If $S_{\leq 1}\to \End(N)$ has rank $5$, then it is injective.
            The module $N$ is supported at $0$, so $x_4$ acts nilpotently on
            $N$ and it acts as an isomorphism on $S/\maxideal_4$. It follows
            that some large power $x_4^e$ acts by zero on $N$ and as an
            isomorphism on $S/\maxideal_4$. Therefore, the multiplication by
            $x_4^e$ does not coincide with the multiplication by any element of
            $S_{\leq 1}$ and the resulting module is not End-closed.

\subsubsection{Support of cardinality $s = 3$.}
    The decomposition $5 = 3+1+1$ yields 
    \[
    \begin{array}{c|c}
        \bolded{M_{3,1}} & \kk[x_1]/(x_1^3) \oplus S / \maxideal_3 \oplus S /\maxideal_4 \\
        \bolded{M_{3,2}} & \kk[x_1,x_2]/(x_1, x_2)^2 \oplus S / \maxideal_3 \oplus S /\maxideal_4 \\
         & (\kk[x_1,x_2]/(x_1, x_2)^2)^\vee \oplus S / \maxideal_3 \oplus S /\maxideal_4 \\
    \end{array}
    \]
    The second decomposition $5 = 2+2+1$ gives 
    \[
    \bolded{M_{3,3}} = \kk[x_1]/(x_1)^2 \oplus \kk[x_2]/(x_2 - 1)^2 \oplus S / \maxideal_4,
    \]
    where $x_3$ acts on $\kk[x_1]/(x_1)^2 \oplus \kk[x_2]/(x_2 - 1)^2$ by $x_3(m_1 + m_2) = m_2$. 
    All the modules are concise by direct inspection and  thus End-closed by Corollary~\ref{cor_unique}.

\subsubsection{Support of cardinality $s = 4,5$.}
    For $s=4$, the only module is 
    \[
        \bolded{M_{4,1}} = \kk[x_1]/(x_1^2) \oplus S /\maxideal_2 \oplus S / \maxideal_3 \oplus S / \maxideal_4.
    \]
    It is concise and hence End-closed.
    For $s=5$, the only module is 
    \[
        \bolded{M_{5,1}} = S / \maxideal \oplus S /\maxideal_1 \oplus \dots \oplus S / \maxideal_4.
    \]
    It is concise and hence End-closed.

\subsubsection{The case $m\leq 4$.}
The situation for $m \leq 4$ is simple, but for sake of completeness we give an explicit list 
of modules, up to duality. All the obtained modules are End-closed by Corollary~\ref{cor_unique}.

For $m=2$ we obtain
    \[
        \kk[x_1]/(x_1)^2, \quad S/ \maxideal \oplus S / \maxideal_1
    \]
For $m=3$ we obtain
    \[
        \kk[x_1]/(x_1)^3, \quad \kk[x_1, x_2]/(x_1, x_2)^2, \quad \kk[x_1]/(x_1)^2 \oplus S / \maxideal_2, \quad S / \maxideal \oplus S / \maxideal_1 \oplus S / \maxideal_2
    \]
For $m=4$ we obtain
    \[
    \begin{split}
        &\kk[x_1]/(x_1)^4, \quad \kk[x_1, x_2]/(x_1^2, x_1 x_2, x_2^3), \quad \kk[x_1, x_2]/(x_1^2, x_2^2), \quad \kk[x_1, x_2, x_3]/(x_1, x_2, x_3)^2,\\
        &N_7, \quad N_8, \quad \kk[x_1]/(x_1)^3 \oplus S / \maxideal_3, \quad \kk[x_1, x_2]/(x_1, x_2)^2 \oplus S / \maxideal_3, \quad\kk[x_1]/(x_1)^2 \oplus \kk[x_2]/(x_2)^2,\\
        & \kk[x_1]/(x_1)^2 \oplus S / \maxideal_2  \oplus S / \maxideal_3, \quad S / \maxideal \oplus S / \maxideal_1 \oplus S / \maxideal_2 \oplus S / \maxideal_3
    \end{split}
    \]

\section{Summary of isomorphism classes up to permutations}

\subsection{Minimal border rank $1_*$-generic tensors}\label{classification_tensors}

We now translate the results from Subsection~\ref{classification_concise} into the tensor language. For sake of consistency we rename each tensor $M_i$ from Subsection~\ref{classification_local} to $M_{1, i}$. The tensor corresponding to $M_{s,i}$ will be denoted by $T_{s,i}$. We represent each $T_{s,i}$ as a space of matrices in variables $x_0, x_1, \dots, x_4$, where $x_0$ corresponds to the action of scalars. It is illustrated by the following example.

\begin{example}\label{example_tensor}
    Consider the $\kk[x_1, x_2, x_3]$-module $\kk[x, y] / (x^2, xy, y^3)$ from Example~\ref{ex1}. Recall that we chose the basis $1, x, y, y^2$ and assigned the corresponding endomorphisms to $1, x_1, x_2, x_3$. This yields the tensor represented by
    \[
    \begin{bmatrix}
        x_0 & 0 & 0 & 0 \\
        x_1 & x_0 & 0 & 0 \\
        x_2 & 0 & x_0 & 0 \\
        x_3 & 0 & x_2 & x_0 \\
    \end{bmatrix}.
    \]

    In the tensor notation, we can write this element of $A \otimes B \otimes C = \kk^4 \otimes \kk^4 \otimes \kk^4$ as
    \[
        a_1 \otimes (b_1 \otimes c_1 + \dots + b_4 \otimes c_4) + a_2 \otimes b_1 \otimes c_2 + a_3 \otimes (b_1 \otimes c_3 + b_3 \otimes c_4) + a_4 \otimes b_1 \otimes c_4.
    \]
    Note that this is the tensor considered in Example \ref{ex0}.
\end{example}

Now we present the classification of $1_*$-generic minimal border rank
tensors up to permutations, as declared in Theorem~\ref{classification_theorem}. It holds under the
assumption that $\kk$ is an algebraically closed field of $\charr \kk \neq 2$.
We arrange the tensors by $s$, the cardinality of the support of the
associated module, or, equivalently, the maximal number of parts in which the
tensor splits.

\subsubsection{Cardinality of the support $s=1$, that is, the local case.}
{\small\begin{equation*}
        T_{1,1} = \begin{bmatrix}
            x_0 & 0 & 0 & 0 & 0 \\
            x_1 & x_0 & 0 & 0 & 0 \\
            x_2 & x_1 & x_0 & 0 & 0 \\
            x_3 & x_2 & x_1 & x_0 & 0 \\
            x_4 & x_3 & x_2 & x_1 & x_0 \\
        \end{bmatrix}\quad
        T_{1,2} = \begin{bmatrix}
            x_0 & 0 & 0 & 0 & 0 \\
            x_1 & x_0 & 0 & 0 & 0 \\
            x_2 & 0 & x_0 & 0 & 0 \\
            x_3 & 0 & x_2 & x_0 & 0 \\
            x_4 & 0 & x_3 & x_2 & x_0 \\
        \end{bmatrix}\quad
        T_{1,3} = \begin{bmatrix}
            x_0 & 0 & 0 & 0 & 0 \\
            x_1 & x_0 & 0 & 0 & 0 \\
            x_2 & 0 & x_0 & 0 & 0 \\
            x_3 & 0 & x_2 & x_0 & 0 \\
            x_4 & x_1 & x_3 & x_2 & x_0 \\
        \end{bmatrix}
    \end{equation*}
    \begin{equation*}
        T_{1,4} = \begin{bmatrix}
            x_0 & 0 & 0 & 0 & 0 \\
            x_1 & x_0 & 0 & 0 & 0 \\
            x_2 & x_1 & x_0 & 0 & 0 \\
            x_3 & 0 & 0 & x_0 & 0 \\
            x_4 & 0 & 0 & x_3 & x_0 \\
        \end{bmatrix}\quad
        T_{1,5} = \begin{bmatrix}
            x_0 & 0 & 0 & 0 & 0 \\
            x_1 & x_0 & 0 & 0 & 0 \\
            x_2 & 0 & x_0 & 0 & 0 \\
            x_3 & 0 & x_2 & x_0 & 0 \\
            x_4 &  x_2 & x_1 & 0 & x_0 \\
        \end{bmatrix}\quad
        T_{1,6} = \begin{bmatrix}
            x_0 & 0 & 0 & 0 & 0 \\
            x_1 & x_0 & 0 & 0 & 0 \\
            x_2 & 0 & x_0 & 0 & 0 \\
            x_3 & 0 & 0 & x_0 & 0 \\
            x_4 & 0 & 0 & x_3 & x_0 \\
        \end{bmatrix}
    \end{equation*}
    \begin{equation*}
        T_{1,7} = \begin{bmatrix}
            x_0 & 0 & 0 & 0 & 0 \\
            x_1 & x_0 & 0 & 0 & 0 \\
            x_2 & 0 & x_0 & 0 & 0 \\
            x_3 & 0 & 0 & x_0 & 0 \\
            x_4 & 0 & x_3 & x_2 & x_0 \\
        \end{bmatrix}\quad
        T_{1,8} = \begin{bmatrix}
            x_0 & 0 & 0 & 0 & 0 \\
            x_1 & x_0 & 0 & 0 & 0 \\
            x_2 & 0 & x_0 & 0 & 0 \\
            x_3 & 0 & 0 & x_0 & 0 \\
            x_4 & x_2 & x_1 & x_3 & x_0 \\
        \end{bmatrix}\quad
        T_{1,9} = \begin{bmatrix}
            x_0 & 0 & 0 & 0 & 0 \\
            x_1 & x_0 & 0 & 0 & 0 \\
            x_2 & 0 & x_0 & 0 & 0 \\
            x_3 & 0 & 0 & x_0 & 0 \\
            x_4 & 0 & 0 & 0 & x_0 \\
        \end{bmatrix}
    \end{equation*}
    \begin{equation*}
        T_{1,10} = \begin{bmatrix}
            x_0 & 0 & 0 & 0 & 0 \\
            0 & x_0 & 0 & 0 & 0 \\
            0 & 0 & x_0 & 0 & 0 \\
            0 & x_1 & x_3 & x_0 & 0 \\
            0 & x_2 & x_4 & 0 & x_0 \\
        \end{bmatrix}\quad
        T_{1,11} = \begin{bmatrix}
            x_0 & 0 & 0 & 0 & 0 \\
            0 & x_0 & 0 & 0 & 0 \\
            0 & 0 & x_0 & 0 & 0 \\
            x_2 & x_3 & x_4 & x_0 & 0 \\
            x_1 & x_2 & x_3 & 0 & x_0 \\
        \end{bmatrix}\quad
        T_{1,12} = \begin{bmatrix}
            x_0 & 0 & 0 & 0 & 0 \\
            0 & x_0 & 0 & 0 & 0 \\
            0 & 0 & x_0 & 0 & 0 \\
            x_1 & 0 & x_4 & x_0 & 0 \\
            x_2 & x_3 & x_1 & 0 & x_0 \\
        \end{bmatrix}
    \end{equation*}
    \begin{equation*}
        T_{1,13} = \begin{bmatrix}
            x_0 & 0 & 0 & 0 & 0 \\
            0 & x_0 & 0 & 0 & 0 \\
            0 & 0 & x_0 & 0 & 0 \\
            x_1 & 0 & x_4 & x_0 & 0 \\
            x_2 & x_3 & 0 & 0 & x_0 \\
        \end{bmatrix}\quad
        T_{1,14} = \begin{bmatrix}
            x_0 & 0 & 0 & 0 & 0 \\
            0 & x_0 & 0 & 0 & 0 \\
            0 & 0 & x_0 & 0 & 0 \\
            x_1 & x_4 & 0 & x_0 & 0 \\
            x_2 & x_3 & x_4 & 0 & x_0 \\
        \end{bmatrix}\quad
        T_{1,15} = \begin{bmatrix}
            x_0 & 0 & 0 & 0 & 0 \\
            0 & x_0 & 0 & 0 & 0 \\
            0 & 0 & x_0 & 0 & 0 \\
            x_1 & x_3 & x_4 & x_0 & 0 \\
            x_2 & 0 & 0 & 0 & x_0 \\
        \end{bmatrix}
    \end{equation*}
    \begin{equation*}
        T_{1,16} = \begin{bmatrix}
            x_0 & 0 & 0 & 0 & 0 \\
            0 & x_0 & 0 & 0 & 0 \\
            x_1 & 0 & x_0 & 0 & 0 \\
            x_3 & 0 & 0 & x_0 & 0 \\
            x_2 & x_4 & x_1 & 0 & x_0 \\
        \end{bmatrix}\quad
        T_{1,17} = \begin{bmatrix}
            x_0 & 0 & 0 & 0 & 0 \\
            0 & x_0 & 0 & 0 & 0 \\
            x_1 & 0 & x_0 & 0 & 0 \\
            x_3 & x_1 & 0 & x_0 & 0 \\
            x_2 & x_4 & x_1 & 0 & x_0 \\
        \end{bmatrix}\quad
        T_{1,18} = \begin{bmatrix}
            x_0 & 0 & 0 & 0 & 0 \\
            0 & x_0 & 0 & 0 & 0 \\
            x_1 & 0 & x_0 & 0 & 0 \\
            x_3 & x_4 & 0 & x_0 & 0 \\
            x_2 & 0 & x_1 & 0 & x_0 \\
        \end{bmatrix}
    \end{equation*}
    \begin{equation*}
        T_{1,19} = \begin{bmatrix}
            x_0 & 0 & 0 & 0 & 0 \\
            0 & x_0 & 0 & 0 & 0 \\
            x_1 & 0 & x_0 & 0 & 0 \\
            x_3 & x_4 & 0 & x_0 & 0 \\
            x_2 & x_3 & x_1 & 0 & x_0 \\
        \end{bmatrix}
    \end{equation*}}
    \subsubsection{Cardinality of the support $s\geq 2$.}
{\small
    \begin{equation*}
            T_{2,1} = \begin{bmatrix}
            x_0 & 0 & 0 & 0 & 0 \\
            x_1 & x_0 & 0 & 0 & 0 \\
            x_2 & x_1 & x_0 & 0 & 0 \\
            0 & 0 & 0 & x_0 + x_4 & 0 \\
            0 & 0 & 0 & x_3 & x_0 + x_4 \\
        \end{bmatrix}\quad
        T_{2,2} = \begin{bmatrix}
            x_0 & 0 & 0 & 0 & 0 \\
            x_1 & x_0 & 0 & 0 & 0 \\
            x_2 & 0 & x_0 & 0 & 0 \\
            0 & 0 & 0 & x_0 + x_4 & 0 \\
            0 & 0 & 0 & x_3 & x_0 + x_4 \\
        \end{bmatrix}
    \end{equation*}
    \begin{equation*}
        T_{2,3} = \begin{bmatrix}
            x_0 & 0 & 0 & 0 & 0 \\
            x_1 & x_0 & 0 & 0 & 0 \\
            x_2 & x_1 & x_0 & 0 & 0 \\
            x_3 & x_2 & x_1 & x_0 & 0 \\
            0 & 0 & 0 & 0 & x_0 + x_4 \\
        \end{bmatrix}\quad
        T_{2,4} = \begin{bmatrix}
            x_0 & 0 & 0 & 0 & 0 \\
            x_1 & x_0 & 0 & 0 & 0 \\
            x_2 & 0 & x_0 & 0 & 0 \\
            x_3 & 0 & x_2 & x_0 & 0 \\
            0 & 0 & 0 & 0 & x_0 + x_4 \\
        \end{bmatrix}
    \end{equation*}
    \begin{equation*}
        T_{2,5} = \begin{bmatrix}
            x_0 & 0 & 0 & 0 & 0 \\
            x_1 & x_0 & 0 & 0 & 0 \\
            x_2 & 0 & x_0 & 0 & 0 \\
            x_3 & x_2 & x_1 & x_0 & 0 \\
            0 & 0 & 0 & 0 & x_0 + x_4 \\
        \end{bmatrix}\quad
        T_{2,6} = \begin{bmatrix}
            x_0 & 0 & 0 & 0 & 0 \\
            x_1 & x_0 & 0 & 0 & 0 \\
            x_2 & 0 & x_0 & 0 & 0 \\
            x_3 & 0 & 0 & x_0 & 0 \\
            0 & 0 & 0 & 0 & x_0 + x_4 \\
        \end{bmatrix}
    \end{equation*}
    \begin{equation*}
        T_{2,7} = \begin{bmatrix}
            x_0 & 0 & 0 & 0 & 0 \\
            0 & x_0 & 0 & 0 & 0 \\
            x_1 & x_2 & x_0 & 0 & 0 \\
            x_3 & -x_1 & 0 & x_0 & 0 \\
            0 & 0 & 0 & 0 & x_0 + x_4 \\
        \end{bmatrix}\quad
        T_{2,8} = \begin{bmatrix}
            x_0 & 0 & 0 & 0 & 0 \\
            0 & x_0 & 0 & 0 & 0 \\
            x_1 & x_2 & x_0 & 0 & 0 \\
            x_3 & 0 & 0 & x_0 & 0 \\
            0 & 0 & 0 & 0 & x_0 + x_4 \\
        \end{bmatrix}
    \end{equation*}
    \begin{equation*}
        T_{3,1} = \begin{bmatrix}
            x_0 & 0 & 0 & 0 & 0 \\
            x_1 & x_0 & 0 & 0 & 0 \\
            x_2 & x_1 & x_0 & 0 & 0 \\
            0 & 0 & 0 & x_0 + x_3 & 0 \\
            0 & 0 & 0 & 0 & x_0 + x_4 \\
        \end{bmatrix}\quad
        T_{3,2} = \begin{bmatrix}
            x_0 & 0 & 0 & 0 & 0 \\
            x_1 & x_0 & 0 & 0 & 0 \\
            x_2 & 0 & x_0 & 0 & 0 \\
            0 & 0 & 0 & x_0 + x_3 & 0 \\
            0 & 0 & 0 & 0 & x_0 + x_4 \\
        \end{bmatrix}
    \end{equation*}
    \begin{equation*}
        T_{3,3} = \begin{bmatrix}
            x_0 & 0 & 0 & 0 & 0 \\
            x_1 & x_0 & 0 & 0 & 0 \\
            0 & 0 & x_0 + x_3 & 0 & 0 \\
            0 & 0 & x_2 & x_0 + x_3 & 0 \\
            0 & 0 & 0 & 0 & x_0 + x_4 \\
        \end{bmatrix}\quad
        T_{4,1} = \begin{bmatrix}
            x_0 & 0 & 0 & 0 & 0 \\
            x_1 & x_0 & 0 & 0 & 0 \\
            0 & 0 & x_0 + x_2 & 0 & 0 \\
            0 & 0 & 0 & x_0 + x_3 & 0 \\
            0 & 0 & 0 & 0 & x_0 + x_4 \\
        \end{bmatrix}
    \end{equation*}
    \begin{equation*}
        T_{5,1} = \begin{bmatrix}
            x_0 & 0 & 0 & 0 & 0 \\
            0 & x_0 + x_1 & 0 & 0 & 0 \\
            0 & 0 & x_0 + x_2 & 0 & 0 \\
            0 & 0 & 0 & x_0 + x_3 & 0 \\
            0 & 0 & 0 & 0 & x_0 + x_4 \\
        \end{bmatrix}.
    \end{equation*}
}
For $m\leq 4$ the isomorphism classes up to permutations are listed in
Appendix~\ref{sec:smallm}.

For completeness below we recall the classification of $1$-degenerate minimal
border rank tensors from \cite[Theorem 1.7]{concise}. Note that this result,
unlike the classifications obtained in our article, was proved under the
additional assumption that $\kk = \CC$. Therefore the final classification of
minimal border rank tensors for $m \leq 5$ holds under the assumption that
$\kk = \CC$. For convenience, we also use the tensor $\Tdegtw{}$ rather than
the isomorphic tensor $\Tdeg{56}$.

\begin{equation*}
\begin{split}
    &\Tdeg{58} = \begin{bmatrix}
        x_0 & 0 & x_1 & x_2 & x_4 \\
        x_4 & x_0 & x_3 & -x_1 & 0 \\
        0 & 0 & x_0 & 0 & 0 \\
        0 & 0 & - x_4 & x_0 & 0 \\
        0 & 0 & 0 & x_4 & 0 \\
    \end{bmatrix},\quad
    \Tdeg{57} = \begin{bmatrix}
        x_0 & 0 & x_1 & x_2 & x_4 \\
        0 & x_0 & x_3 & -x_1 & 0 \\
        0 & 0 & x_0 & 0 & 0 \\
        0 & 0 & 0 & x_0 & 0 \\
        0 & 0 & 0 & x_4 & 0 \\
    \end{bmatrix},\\
    &\Tdegtw{} = \begin{bmatrix}
        x_0 & 0 & x_1 & x_2 & x_4 \\
        0 & x_0  & 0 & x_3 & x_4\\
        0 & 0 & x_0 & 0 & 0 \\
        0 & 0 & 0 & x_0 & 0 \\
        0 & 0 & 0 & x_4 & 0 \\
    \end{bmatrix},\quad
    \Tdeg{55} = \begin{bmatrix}
        x_0 & 0 & x_1 & x_2 & x_4 \\
        0 & x_0  & x_4 & x_3 & 0 \\
        0 & 0 & x_0 & 0 & 0 \\
        0 & 0 & 0 & x_0 & 0 \\
        0 & 0 & 0 & x_4 & 0 \\
    \end{bmatrix},\quad
    \Tdeg{54} = \begin{bmatrix}
        x_0 & 0 & x_1 & x_2 & x_4 \\
        0 & x_0  & 0 & x_3 & 0 \\
        0 & 0 & x_0 & 0 & 0 \\
        0 & 0 & 0 & x_0 & 0 \\
        0 & 0 & 0 & x_4 & 0 \\
    \end{bmatrix}
\end{split}
\end{equation*}

\subsubsection{Non End-closed tensors}
In this section we record the following side result.
\begin{proposition}\label{theorem_additional}
    Let $\kk$ be an algebraically closed field with $\charr \kk \neq 2$.
    Consider $1_*$-generic concise tensors in $\kk^m \otimes \kk^m \otimes
    \kk^m$ which are not End-closed (hence with border rank strictly greater
    than $m$). For $m\leq 4$ there are no such tensors.
    For $m = 5$ there are exactly two up to isomorphism and permutations:
\[
    T_{1,20} = \begin{bmatrix}
            x_0 & 0 & 0 & 0 & 0 \\
            0 & x_0 & 0 & 0 & 0 \\
            x_1 & 0 & x_0 & 0 & 0 \\
            x_3 & x_4 & 0 & x_0 & 0 \\
            0 & x_2 & x_1 & 0 & x_0
        \end{bmatrix}\quad
        T_{2,9} = \begin{bmatrix}
        x_0 & 0 & 0 & 0 & 0 \\
        0 & x_0 & 0 & 0 & 0 \\
        x_1 & x_2 & x_0 & 0 & 0 \\
        x_3 & x_4 & 0 & x_0 & 0 \\
        0 & 0 & 0 & 0 & x_0 + x_4
    \end{bmatrix}.
\]
\end{proposition}

\subsection{Limits of diagonalizable subspaces of matrices}\label{classification_subspaces}
In Subsection \ref{prelims_tensors}, we described a correspondence between $1_A$-generic minimal border rank tensors in $\kk^m \otimes \kk^m \otimes \kk^m$ and $m$-dimensional subspaces of $\End(\kk^m)$ which are limits of diagonalizable subspaces. Recall that a subspace $\mathcal{E} \subset \End(\kk^m)$ is \emph{diagonalizable} if there exists a basis of $\kk^m$ such that $\mathcal{E} \subset \mathrm{M}_{m\times m}(\kk)$ consists of diagonal matrices. Two subspaces of $\mathrm{M}_{m \times m}(\kk)$ are \emph{equivalent} if they correspond to the same subspace of $\End(\kk^m)$, i.e., they differ by a choice of basis.
We use the classification of tensors from Theorem~\ref{classification_theorem} to derive the classification of such subspaces. The result is summed up in Theorem~\ref{theorem_subspaces}.

\begin{proof}[Proof of Theorem~\ref{theorem_subspaces}]
    By the correspondence described in Subsection \ref{prelims_tensors}, limits of diagonalizable subspaces correspond to $1_A$-generic minimal border rank tensors and to concise End-closed modules. In Theorem \ref{classification_theorem}, we classified such tensors up to a permutation of the factors $B, C$. On the level of modules, it corresponds to identifying modules with their duals. Therefore we need to determine which modules are not self-dual. The final list of subspaces consists of the subspaces corresponding to tensors from Theorem \ref{classification_theorem} and the subspaces corresponding to tensors which come from these new dual modules. 

    Let $M$ be a local $S$-module of finite degree. We know that $\dim_\kk
    M^\vee / \maxideal M^\vee = \dim_\kk (0 : \maxideal)_M$, see proof of
    Lemma \ref{cyclic}. Therefore if $M$ is self-dual, then $\dim_\kk M /
    \maxideal M = \dim_\kk (0 \colon \maxideal)_M$. If $M$ is cyclic, then the
    other implication holds as well (because then both $M, M^\vee$ come from
    algebras and have equal annihilators). This observation enables us to
    easily calculate which local cyclic modules are self-dual; among the local
    ones these are exactly $M_1$, $M_3$, $M_8$, while among the non-local, the self-dual ones are
    $M_{2,1}$, $M_{2,3}$, $M_{2,5}$, $M_{3, 1}$, $M_{3, 3}$, $M_{4, 1}$,
    $M_{5, 1}$.
    The duality of non-cyclic local modules was determined in
    Subsection~\ref{ssec:m5}: $M_{10}$, $M_{16}$, $M_{17}$, $M_{19}$ are the
    self-dual ones. The only two remaining cases $M_{2,7}$,
    $M_{2,8}$ from Subsection \ref{classification_tensors}, which are
    self-dual, since $N_7, N_8$ from Subsection \ref{classification_local} are
    self-dual.

    We present the list of tensors which correspond to new subspaces. For $m \leq 2$ there are no such tensors. For $m = 3$ we have
    \[
    \begin{bmatrix}
        x_0 & x_1 & x_2 \\
        0 & x_0 & 0 \\
        0 & 0 & x_0 \\
    \end{bmatrix}.
    \]
    For $m = 4$ we have
    \[
    \begin{bmatrix}
        x_0 & x_1 & x_2 & x_3 \\
        0 & x_0 & 0 & 0  \\
        0 & 0 & x_0 & x_2 \\
        0 & 0 & 0 & x_0 \\
    \end{bmatrix},\quad
    \begin{bmatrix}
        x_0 & x_1 & x_2 & x_3 \\
        0 & x_0 & 0 & 0 \\
        0 & 0 & x_0 & 0 \\
        0 & 0 & 0 & x_0 \\
    \end{bmatrix},\quad
    \begin{bmatrix}
        x_0 & x_1 & x_2 & 0 \\
        0 & x_0 & 0 & 0\\
        0 & 0 & x_0 & 0\\
        0 & 0 & 0 & x_0 + x_3 \\
    \end{bmatrix}.
    \]
    For $m = 5$ we have
    \[
        T_{1,2}^{\top},\, T_{1,4}^{\top},\, T_{1,5}^{\top},\, T_{1,6}^{\top},\, T_{1,7}^{\top},\, T_{1,9}^{\top},\, T_{1,11}^{\top},\, T_{1,12}^{\top},\, T_{1,13}^{\top},\, T_{1,14}^{\top},\,\ T_{1,15}^{\top},\,\ T_{1,18}^{\top},\,
        T_{2,2}^{\top},\, T_{2,4}^{\top},\, T_{2,6}^{\top},\,
        T_{3,2}^{\top},
    \]
    where the superscript $(-)^{\top}$ denotes the transpose and the numbering is taken from Theorem \ref{classification_theorem}.
\end{proof}

\begin{remark}
    Modules corresponding to both tensors from Proposition \ref{theorem_additional} are self-dual, so subspaces corresponding to both tensors are isomorphic to its transposes.
\end{remark}

\newcommand{\marklevel}[1]{\textcolor{gray}{\texttt{#1}}}
\begin{figure}
\noindent\scalebox{0.8}{
\begin{tikzpicture}[node distance={20mm}, thick, main/.style = {draw, circle}]

    \node[main] (T51) {$T_{5,1}$};
    \node[main] (T41) [below of=T51] {$T_{4,1}$};
    \node[main] (T33) [below left of=T41] {$T_{3,3}$};
    \node[main] (T31) [below right of=T41] {$T_{3,1}$};
    \node[main] (T21) [below of=T33] {$T_{2,1}$};
    \node[main] (T23) [below of=T31] {$T_{2,3}$};
    \node[main] (T32) [below left of=T21] {$T_{3,2}$};
    \node[main] (T25) [below left of=T23] {$T_{2,5}$};
    \node[main] (T11) [below right of=T23]{$T_{1,1}$};
    \node[main] (T22) [below left of=T32]{$T_{2,2}$};
    \node[main] (T24) [below right of=T32]{$T_{2,4}$};
    \node[main] (T13) [below left of=T11]{$T_{1,3}$};
    \node[main] (T27) [below right of=T11]{$T_{2,7}$};
    \node[main] (T111) [below left=15mm of T22]{$T_{1,11}$};
    \node[main] (T119) [below right=15mm of T22]{$T_{1,19}$};
    \node[main] (T12) [below left=15mm of T13]{$T_{1,2}$};
    \node[main] (T14) [below right=15mm of T13]{$T_{1,4}$};
    \node[main] (T28) [below right=15mm of T27]{$T_{2,8}$};
    \node[main] (T112) [below of=T111]{$T_{1,12}$};
    \node[main] (T118) [below of=T119]{$T_{1,18}$};
    \node[main] (T15) [below of=T12]{$T_{1,5}$};
    \node[main] (T18) [below=27.5mm of T27]{$T_{1,8}$};
    \node[main] (T113) [below of=T112]{$T_{1,13}$};
    \node[main] (T117) [below of=T118]{$T_{1,17}$};
    \node[main] (T17) [below=27.5mm of T14]{$T_{1,7}$};
    \node[main] (T114) [below of=T113]{$T_{1,14}$};
    \node[main] (T116) [below of=T117]{$T_{1,16}$};
    \node[main] (T26) [left of=T114]{$T_{2,6}$};
    \node[main] (T115) [below=17.5mm of T114]{$T_{1,15}$};
    \node[main] (T16) [below=27.5mm of T17]{$T_{1,6}$};
    \node[main] (T110) [below=37.5mm of T26]{$T_{1,10}$};
    \node[main] (T19) [below=42.5mm of T16]{$T_{1,9}$};
    \node[main] (T58) [left=25mm of T111]{$T_{\mathcal{O}_{58}}$};
    \node[main] (T57) [left=25mm of T112]{$T_{\mathcal{O}_{57}}$};
    \node[main] (T56) [left=25mm of T113]{$T_{\widetilde{\mathcal{O}}_{56}}$};
    \node[main] (T55) [left=25mm of T114]{$T_{\mathcal{O}_{55}}$};
    \node[main] (T54) [below=7.5mm of T55]{$T_{\mathcal{O}_{54}}$};
    \node (stab16) [right= 5mm of T28]{\marklevel{16}};
    \node (stab15) [above=14mm of stab16]{\marklevel{15}};
    \node (stab14) [above=10.3mm of stab15]{\marklevel{14}};
    \node (stab13) [above=10.3mm of stab14]{\marklevel{13}};
    \node (stab12) [above=10.3mm of stab13]{\marklevel{12}};
    \node (stab11) [above=10.3mm of stab12]{\marklevel{11}};
    \node (stab10) [above=14mm of stab11]{\marklevel{10}};
    \node (stab17) [below=14mm of stab16]{\marklevel{17}};
    \node (stab18) [below=14mm of stab17]{\marklevel{18}};
    \node (stab19) [below=16mm of stab18]{\marklevel{19}};
    \node (stab20) [below=13mm of stab19]{\marklevel{20}};
    \node (stab21) [below=13mm of stab20]{\marklevel{21}};
    \node (stab22) [below=10mm of stab21]{\marklevel{22}};
    \node (stab26) [below=10mm of stab22]{\marklevel{26}};
    \node (infostab) [below=10mm of stab26]{\marklevel{stab.dim.}};

\draw[arrows={-Latex[length=8pt,bend,line width=0pt]}] (T111) edge (T112);
\draw[arrows={-Latex[length=8pt,bend,line width=0pt]}] (T112) edge (T113);
\draw[arrows={-Latex[length=8pt,bend,line width=0pt]}] (T113) edge (T114);
\draw[arrows={-Latex[length=8pt,bend,line width=0pt]}] (T114) edge (T115);
\draw[arrows={-Latex[length=8pt,bend,line width=0pt]}] (T114) edge (T110);
\draw[arrows={-Latex[length=8pt,bend,line width=0pt]}] (T119) edge (T118);
\draw[arrows={-Latex[length=8pt,bend,line width=0pt]}] (T118) edge (T117);
\draw[arrows={-Latex[length=8pt,bend,line width=0pt]}] (T117) edge (T116);
\draw[arrows={-Latex[length=8pt,bend,line width=0pt]}] (T119) edge (T112);
\draw[arrows={-Latex[length=8pt,bend,line width=0pt]}] (T118) edge (T113);
\draw[arrows={-Latex[length=8pt,bend,line width=0pt]}] (T117) edge (T114);
\draw[arrows={-Latex[length=8pt,bend,line width=0pt]}] (T116) edge (T115);
\draw[arrows={-Latex[length=8pt,bend,line width=0pt]}, in=50, out=180] (T13) edge (T119);
\draw[arrows={-Latex[length=8pt,bend,line width=0pt]},in=45,out=195] (T14) edge (T118);
\draw[arrows={-Latex[length=8pt,bend,line width=0pt]}] (T15) edge (T117);
\draw[arrows={-Latex[length=8pt,bend,line width=0pt]}] (T17) edge (T116);
\draw[arrows={-Latex[length=8pt,bend,line width=0pt]}] (T16) edge (T115);
\draw[arrows={-Latex[length=8pt,bend,line width=0pt]}, in=100, out=202] (T11) edge (T111);
\draw[arrows={-Latex[length=8pt,bend,line width=0pt]}] (T27) edge (T28);
\draw[arrows={-Latex[length=8pt,bend,line width=0pt]}] (T16) edge (T19);
\draw[arrows={-Latex[length=8pt,bend,line width=0pt]}] (T17) edge (T16);
\draw[arrows={-Latex[length=8pt,bend,line width=0pt]}] (T15) edge (T17);
\draw[arrows={-Latex[length=8pt,bend,line width=0pt]}] (T12) edge (T15);
\draw[arrows={-Latex[length=8pt,bend,line width=0pt]}] (T18) edge (T17);
\draw[arrows={-Latex[length=8pt,bend,line width=0pt]}] (T14) edge (T15);
\draw[arrows={-Latex[length=8pt,bend,line width=0pt]}] (T13) edge (T14);
\draw[arrows={-Latex[length=8pt,bend,line width=0pt]}, out=190, in=90] (T13) edge (T12);
\draw[arrows={-Latex[length=8pt,bend,line width=0pt]}, out=0, in=90] (T13) edge (T18);
\draw[arrows={-Latex[length=8pt,bend,line width=0pt]}] (T11) edge (T13);
\draw[arrows={-Latex[length=8pt,bend,line width=0pt]}] (T24)
.. controls +(230:2cm) and +(90:6cm).. (T26);
\draw[arrows={-Latex[length=8pt,bend,line width=0pt]}] (T24) edge (T12);
\draw[arrows={-Latex[length=8pt,bend,line width=0pt]}] (T24) edge (T14);
\draw[arrows={-Latex[length=8pt,bend,line width=0pt]}, out=340, in=160] (T24) edge (T28);
\draw[arrows={-Latex[length=8pt,bend,line width=0pt]}, in=45, out=205] (T28) edge (T118);
\draw[arrows={-Latex[length=8pt,bend,line width=0pt]}, out=200, in=35] (T27) edge (T119);
\draw[arrows={-Latex[length=8pt,bend,line width=0pt]}, out=320, in=180] (T26) edge (T16);
\draw[arrows={-Latex[length=8pt,bend,line width=0pt]}] (T25) edge (T24);
\draw[arrows={-Latex[length=8pt,bend,line width=0pt]}] (T32) edge (T22);
\draw[arrows={-Latex[length=8pt,bend,line width=0pt]}] (T41) edge (T33);
\draw[arrows={-Latex[length=8pt,bend,line width=0pt]}] (T51) edge (T41);
\draw[arrows={-Latex[length=8pt,bend,line width=0pt]}] (T41) edge (T31);
\draw[arrows={-Latex[length=8pt,bend,line width=0pt]}] (T33) edge (T21);
\draw[arrows={-Latex[length=8pt,bend,line width=0pt]}] (T33) edge (T23);
\draw[arrows={-Latex[length=8pt,bend,line width=0pt]}] (T31) edge (T21);
\draw[arrows={-Latex[length=8pt,bend,line width=0pt]}] (T31) edge (T23);
\draw[arrows={-Latex[length=8pt,bend,line width=0pt]}, bend left=15] (T31) edge (T32);
\draw[arrows={-Latex[length=8pt,bend,line width=0pt]}] (T23) edge (T25);
\draw[arrows={-Latex[length=8pt,bend,line width=0pt]}] (T23) edge (T11);
\draw[arrows={-Latex[length=8pt,bend,line width=0pt]}, bend left=30] (T23) edge (T27);
\draw[arrows={-Latex[length=8pt,bend,line width=0pt]}] (T21) edge (T11);
\draw[arrows={-Latex[length=8pt,bend,line width=0pt]}, out=350, in=135] (T22) edge (T12);
\draw[arrows={-Latex[length=8pt,bend,line width=0pt]}, out=220, in=30] (T12) edge (T112);
\draw[arrows={-Latex[length=8pt,bend,line width=0pt]}] (T25) edge (T13);
\draw[arrows={-Latex[length=8pt,bend,line width=0pt]}] (T32) edge (T24);
\draw[arrows={-Latex[length=8pt,bend,line width=0pt]}, bend right] (T21) edge (T22);
\draw[arrows={-Latex[length=8pt,bend,line width=0pt]}] (T12) edge (T118);
\draw[arrows={-Latex[length=8pt,bend,line width=0pt]}, out=200, in=27] (T116) edge (T54);
\draw[arrows={-Latex[length=8pt,bend,line width=0pt]}, out=200, in=27] (T117) edge (T55);
\draw[arrows={-Latex[length=8pt,bend,line width=0pt]}, out=200, in=27] (T118) edge (T56);
\draw[arrows={-Latex[length=8pt,bend,line width=0pt]}, out=200, in=27] (T119) edge (T57);
\draw[arrows={-Latex[length=8pt,bend,line width=0pt]}, out=202, in=70] (T11) edge (T58);
\draw[arrows={-Latex[length=8pt,bend,line width=0pt]}, out=205, in=27] (T12) edge (T57);
\draw[arrows={-Latex[length=8pt,bend,line width=0pt]}] (T58) edge (T57);
\draw[arrows={-Latex[length=8pt,bend,line width=0pt]}] (T57) edge (T56);
\draw[arrows={-Latex[length=8pt,bend,line width=0pt]}] (T56) edge (T55);
\draw[arrows={-Latex[length=8pt,bend,line width=0pt]}] (T55) edge (T54);

\end{tikzpicture}}
\caption{Degenerations of minimal border rank tensors in $\kk^5\otimes
\kk^5\otimes \kk^5$}\label{sec:diagram}
\end{figure}
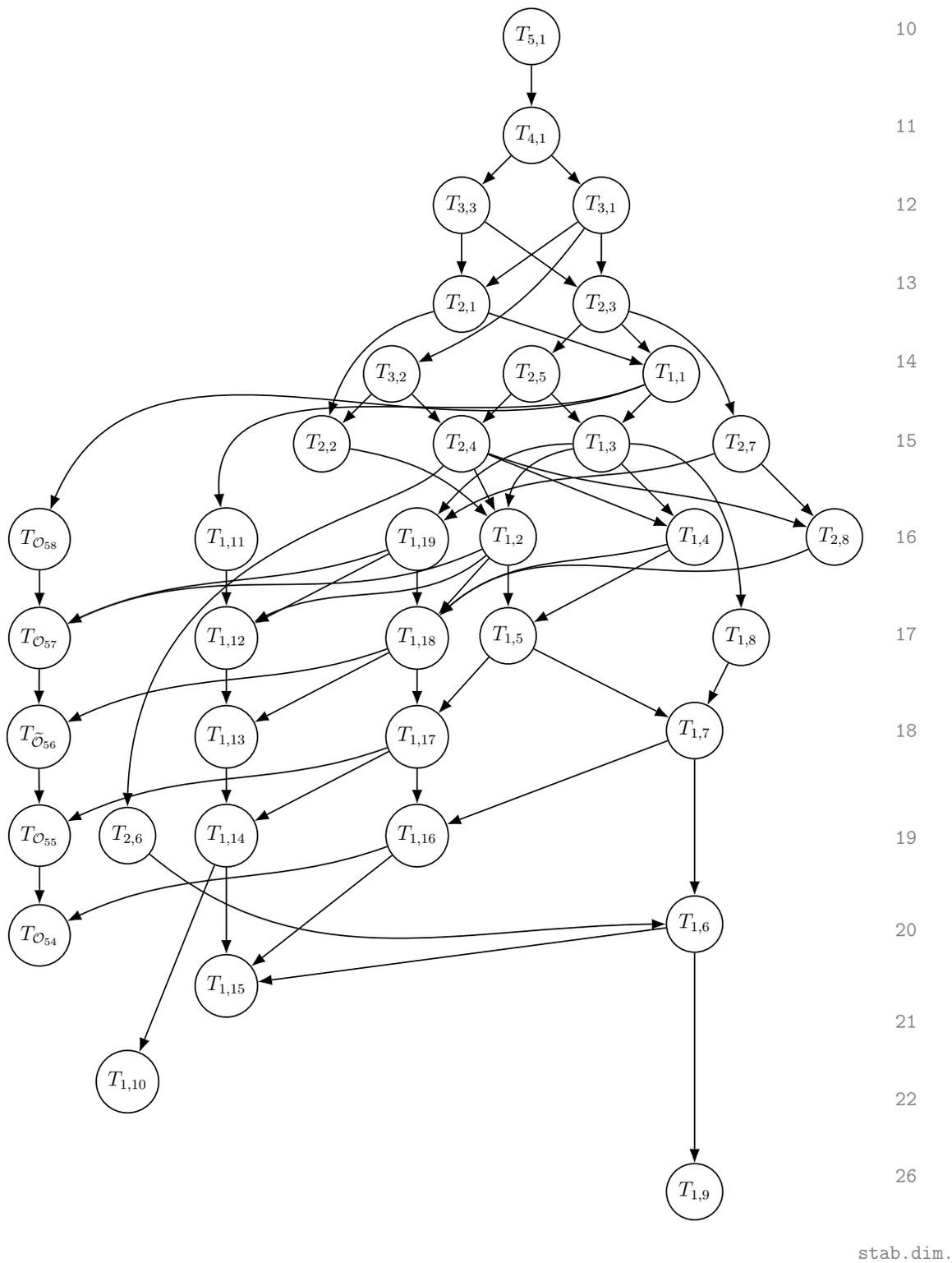

\begin{figure}
\scalebox{0.8}{
\noindent\begin{tikzpicture}[node distance={20mm}, thick, main/.style = {draw,
    circle}, every edge quotes/.style = {auto, font=\footnotesize,
    very near end}]

    \node[main] (T51) {$T_{5,1}$};
    \node[main] (T41) [below of=T51] {$T_{4,1}$};
    \node[main] (T33) [below left of=T41] {$T_{3,3}$};
    \node[main] (T31) [below right of=T41] {$T_{3,1}$};
    \node[main] (T21) [below of=T33] {$T_{2,1}$};
    \node[main] (T23) [below of=T31] {$T_{2,3}$};
    \node[main] (T32) [below left of=T21] {$T_{3,2}$};
    \node[main] (T25) [below left of=T23] {$T_{2,5}$};
    \node[main] (T11) [below right of=T23]{$T_{1,1}$};
    \node[main] (T22) [below left of=T32]{$T_{2,2}$};
    \node[main] (T24) [below right of=T32]{$T_{2,4}$};
    \node[main] (T13) [below left of=T11]{$T_{1,3}$};
    \node[main] (T27) [below right of=T11]{$T_{2,7}$};
    \node[main] (T111) [below left=15mm of T22]{$T_{1,11}$};
    \node[main] (T119) [below right=15mm of T22]{$T_{1,19}$};
    \node[main] (T12) [below left=15mm of T13]{$T_{1,2}$};
    \node[main] (T14) [below right=15mm of T13]{$T_{1,4}$};
    \node[main] (T28) [below right=15mm of T27]{$T_{2,8}$};
    \node[main] (T112) [below of=T111]{$T_{1,12}$};
    \node[main] (T118) [below of=T119]{$T_{1,18}$};
    \node[main] (T15) [below of=T12]{$T_{1,5}$};
    \node[main] (T18) [below=27.5mm of T27]{$T_{1,8}$};
    \node[main] (T113) [below of=T112]{$T_{1,13}$};
    \node[main] (T117) [below of=T118]{$T_{1,17}$};
    \node[main] (T17) [below=27.5mm of T14]{$T_{1,7}$};
    \node[main] (T114) [below of=T113]{$T_{1,14}$};
    \node[main] (T116) [below of=T117]{$T_{1,16}$};
    \node[main] (T26) [left of=T114]{$T_{2,6}$};
    \node[main] (T115) [below=17.5mm of T114]{$T_{1,15}$};
    \node[main] (T16) [below=27.5mm of T17]{$T_{1,6}$};
    \node[main] (T110) [below=37.5mm of T26]{$T_{1,10}$};
    \node[main] (T19) [below=42.5mm of T16]{$T_{1,9}$};
    \node[main] (T58) [left=25mm of T111]{$T_{\mathcal{O}_{58}}$};
    \node[main] (T57) [left=25mm of T112]{$T_{\mathcal{O}_{57}}$};
    \node[main] (T56) [left=25mm of T113]{$T_{\widetilde{\mathcal{O}}_{56}}$};
    \node[main] (T55) [left=25mm of T114]{$T_{\mathcal{O}_{55}}$};
    \node[main] (T54) [below=7.5mm of T55]{$T_{\mathcal{O}_{54}}$};
    \node (stab16) [right= 5mm of T28]{\marklevel{16}};
    \node (stab15) [above=14mm of stab16]{\marklevel{15}};
    \node (stab14) [above=10.3mm of stab15]{\marklevel{14}};
    \node (stab13) [above=10.3mm of stab14]{\marklevel{13}};
    \node (stab12) [above=10.3mm of stab13]{\marklevel{12}};
    \node (stab11) [above=10.3mm of stab12]{\marklevel{11}};
    \node (stab10) [above=14mm of stab11]{\marklevel{10}};
    \node (stab17) [below=14mm of stab16]{\marklevel{17}};
    \node (stab18) [below=14mm of stab17]{\marklevel{18}};
    \node (stab19) [below=16mm of stab18]{\marklevel{19}};
    \node (stab20) [below=13mm of stab19]{\marklevel{20}};
    \node (stab21) [below=13mm of stab20]{\marklevel{21}};
    \node (stab22) [below=10mm of stab21]{\marklevel{22}};
    \node (stab26) [below=10mm of stab22]{\marklevel{26}};
    \node (infostab) [below=10mm of stab26]{\marklevel{stab.dim.}};

\draw[arrows={-Latex[length=8pt,bend,line width=0pt]}, style=dashed,
in=90,out=270, color=red] (T32) edge node[pos=0.60, sloped, above,
font=\footnotesize] {\ref{ref:notoT119:cor}}(T119);
\draw[arrows={-Latex[length=8pt,bend,line width=0pt]}, style=dashed,
in=140,out=330, color=red] (T25) edge node[very near end, sloped, above,
font=\footnotesize] {\ref{ref:notoT27:lem}} (T27);
\draw[arrows={-Latex[length=8pt,bend,line width=0pt]}, style=dashed,
in=150,out=330, color=red] (T32) edge node[pos=0.9, sloped, below,
font=\footnotesize] {\ref{ref:notoT27:lem}}(T27);
\draw[arrows={-Latex[length=8pt,bend,line width=0pt]}, style=dashed,
in=90,out=180, color=red] (T32) edge node[pos=0.90, sloped, above,
font=\footnotesize] {\ref{ref:notoT111:cor}}(T111);
\draw[arrows={-Latex[length=8pt,bend,line width=0pt]}, style=dashed,
in=30,out=200, color=red] (T25) edge node[pos=0.95, sloped, above,
font=\footnotesize] {\ref{ref:notoT111:cor}}(T111);
\draw[arrows={-Latex[length=8pt,bend,line width=0pt]}, style=dashed,
in=20,out=210, color=red] (T27) edge node[pos=0.95, sloped, below,
font=\footnotesize] {\ref{ref:notoT111:cor}} (T111);
\draw[arrows={-Latex[length=8pt,bend,line width=0pt]}, style=dashed, in=15,out=200, color=red] (T28) edge node[pos=0.95, sloped, below,
font=\footnotesize] {\ref{ref:nofromT28toT112:lem}}(T112);
\draw[arrows={-Latex[length=8pt,bend,line width=0pt]}, style=dashed,
in=15,out=190, color=red] (T14) edge node[pos=0.95, sloped, above,
font=\footnotesize] {\ref{ref:nofromT14toT112:prop}}(T112);
\draw[arrows={-Latex[length=8pt,bend,line width=0pt]}, style=dashed, in=15,out=220, color=red] (T15) edge node[pos=0.90, sloped, below,
font=\footnotesize] {\ref{ref:nofromT15toT113:prop}}(T113);
\draw[arrows={-Latex[length=8pt,bend,line width=0pt]}, style=dashed,
in=0,out=270, color=red] (T18) edge node[pos=0.95, sloped, above,
font=\footnotesize] {\ref{ref:notoT110:cor}}(T110);
\draw[arrows={-Latex[length=8pt,bend,line width=0pt]}, style=dashed,
in=90,out=270, color=red] (T26) edge node[pos=0.80, sloped, below,
font=\footnotesize] {\ref{ref:notoT110:cor}}(T110);

\draw[arrows={-Latex[length=8pt,bend,line width=0pt]}, style=dashed,
in=45,out=225, color=red] (T26) edge node[pos=0.50, sloped, below,
font=\footnotesize] {\ref{ref:nofromT26toT54:lem}}(T54);
\draw[arrows={-Latex[length=8pt,bend,line width=0pt]}, style=dashed, in=60,out=220, color=red] (T111) edge node[pos=0.95, sloped, above,
font=\footnotesize] {\ref{ref:nodegstodegenerate:prop}}(T54);
\draw[arrows={-Latex[length=8pt,bend,line width=0pt]}, style=dashed, in=80,out=160, color=red] (T32) edge node[pos=0.90, sloped, above,
font=\footnotesize] {\ref{ref:nodegstodegenerate:prop}}(T58);
\draw[arrows={-Latex[length=8pt,bend,line width=0pt]}, style=dashed, in=55,out=200, color=red] (T25) edge node[pos=0.90, sloped, above,
font=\footnotesize] {\ref{ref:nodegstodegenerate:prop}}(T58);
\draw[arrows={-Latex[length=8pt,bend,line width=0pt]}, style=dashed, in=20,out=200, color=red] (T27) edge node[pos=0.95, sloped, above,
font=\footnotesize] {\ref{ref:nodegstodegenerate:prop}}(T58);
\draw[arrows={-Latex[length=8pt,bend,line width=0pt]}, style=dashed,
in=20,out=205, color=red] (T28) edge node[pos=0.95, sloped, below,
font=\footnotesize] {\ref{ref:nodegstodegenerate:prop}}(T57);
\draw[arrows={-Latex[length=8pt,bend,line width=0pt]}, style=dashed, in=25,out=200, color=red] (T14) edge node[pos=0.95, sloped, above,
font=\footnotesize] {\ref{ref:nodegstodegenerate:prop}}(T57);
\draw[arrows={-Latex[length=8pt,bend,line width=0pt]}, style=dashed, in=25,out=210, color=red] (T15) edge node[pos=0.95, sloped, above,
font=\footnotesize] {\ref{ref:nodegstodegenerate:prop}}(T56);
\draw[arrows={-Latex[length=8pt,bend,line width=0pt]}, style=dashed,
in=23,out=225, color=red] (T18) edge node[pos=0.90, sloped, above,
font=\footnotesize] {\ref{ref:nofromT18toT55:lem}}(T55);

\draw[arrows={-Latex[length=8pt,bend,line width=0pt]}, style=dashed,
in=60,out=290, color=red] (T27) edge node[pos=0.95, sloped, above,
font=\footnotesize] {\ref{ssec:being1Ageneric}}(T19);
\draw[arrows={-Latex[length=8pt,bend,line width=0pt]}, style=dashed,
in=180,out=350, color=red] (T32) edge node[pos=0.90, sloped, above,
font=\footnotesize] {\ref{ssec:being1Ageneric}}(T13);
\draw[arrows={-Latex[length=8pt,bend,line width=0pt]}, style=dashed,
in=180,out=350, color=red] (T32) edge node[pos=0.95, sloped, above,
font=\footnotesize] {\ref{ssec:being1Ageneric}}(T18);

\draw[arrows={-Latex[length=8pt,bend,line width=0pt]}, style=dashed,
in=90,out=245, color=red] (T33) edge node[pos=0.85, sloped, above,
font=\footnotesize] {\ref{ssec:severalParts}}(T32);
\draw[arrows={-Latex[length=8pt,bend,line width=0pt]}, style=dashed,
in=20,out=200, color=red] (T23) edge node[pos=0.10, sloped, above,
font=\footnotesize] {\ref{ssec:severalParts}}(T22);
\draw[arrows={-Latex[length=8pt,bend,line width=0pt]}, style=dashed,
in=150,out=340, color=red] (T22) edge node[pos=0.95, sloped, below,
font=\footnotesize] {\ref{ex:T22toT14}}(T14);

\end{tikzpicture}}
\caption{Non-degenerations of minimal border rank tensors in $\kk^5\otimes
\kk^5\otimes \kk^5$}\label{sec:nondiagram}
\end{figure}
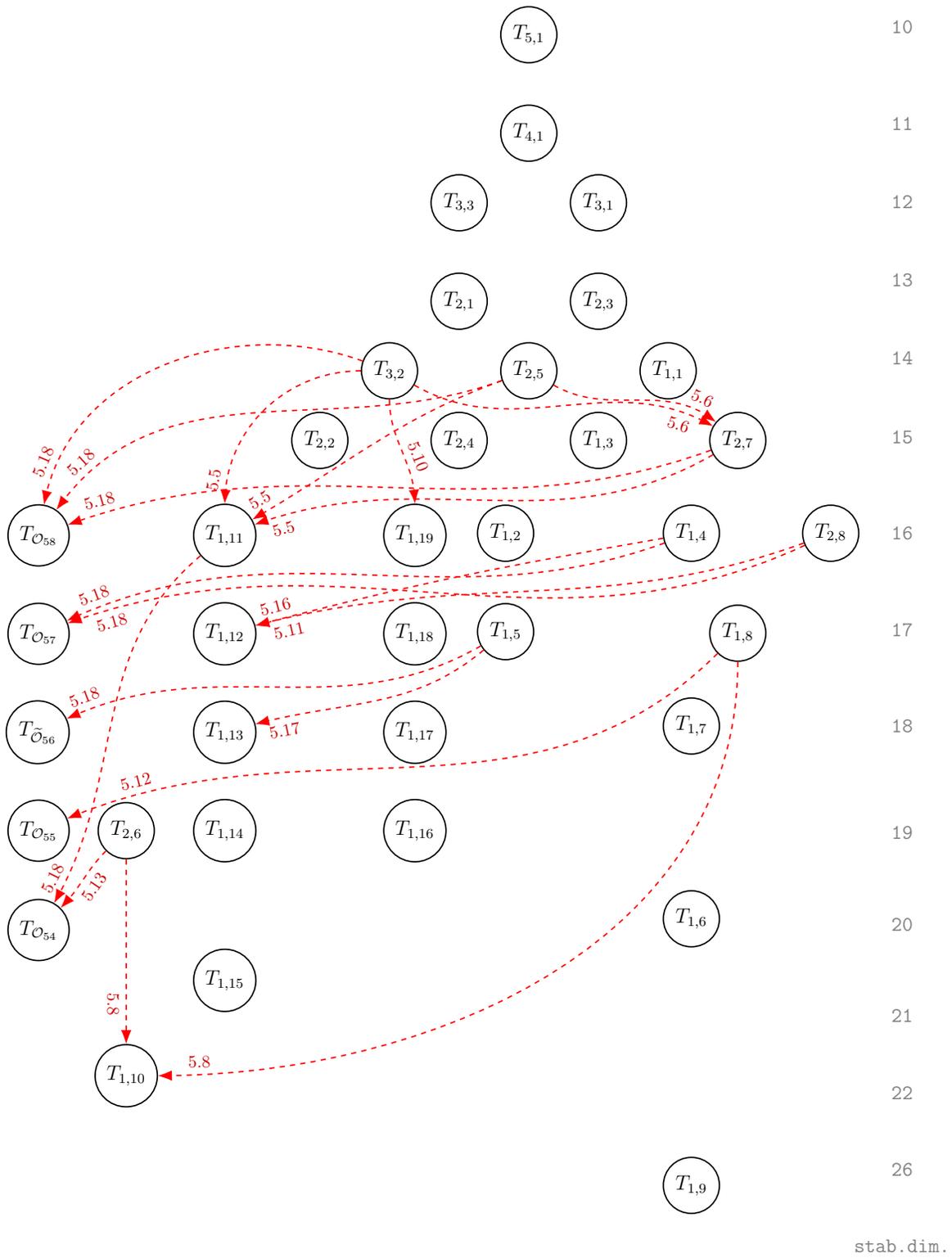

\section{Degenerations of tensors}\label{sec:degenerations}

In this section we prove the nonexistence of certain degenerations of minimal
border rank tensors. Together with the explicit degenerations described in
Appendix~\ref{sec:code} this yields the Diagram~\ref{sec:diagram} and proves
Theorem~\ref{degenerations_theorem}. An analysis of Diagram~\ref{sec:diagram}
yields also Corollary~\ref{ref:localscheme:cor}.

\subsection{On the existence of degenerations}

    To obtain the graph, we need to construct $66$ direct degenerations; every
    of them is constructed completely explicitly in Appendix~\ref{sec:code}. Of
    these, $31$ are degenerations between algebras
    and follow from~\cite{mazzola_generic_finite_schemes}. However, the
    translation of algebra degenerations to tensor requires significant work
    and sometimes results in intricate base changes,
    see for example $T_{1,1}\degto T_{1,3}$ in the code,
    Appendix~\ref{sec:code}. The existence
    of four degenerations $\Tdeg{58}\degto \Tdeg{57}\degto \Tdegtw{} \degto
    \Tdeg{55} \degto \Tdeg{54}$ is proven over $\CC$ in~\cite{concise}.
    The remaining $31$ degenerations are new. We construct them by
    degenerating (the generators of) the apolar modules described in
    Subsection~\ref{subsection_apolarity}, because in them the information
    about tensors is most compressed and hence handy to manipulate. Constructing some of the
    degenerations is almost trivial, as they amount to rescaling coordinates, while
    for some others it proved to be rather tricky and required a case-by-case approach.

    By construction, all degenerations exist over any field $\kk$ of
    characteristic not equal to $2$, $3$. This can be verified directly by
    inspecting the coefficients of the matrices: the only denominators which
    appear are $1$, $2$, $3$, or $18$, see function \texttt{tallyDegenerationDenominators DegenerationList}.

    \subsection{Notation and preliminaries}

    We gather some notation on degenerations. For details on tensors,
    \cite[Chapter~15]{BurgisserBook} is an excellent reference. For details on
    modules, a book on deformation theory, such
    as~\cite{fantechi_et_al_fundamental_ag}, is best.

    Let $T, T'\in A\otimes B\otimes C$ be tensors.
    We say that $T$ \emph{degenerates} to a tensor $T'$ if $T'$ lies in the
    closure of the $(\GL(A) \times \GL(B)\times \GL(C))$-orbit of $T$. We
    denote this by $T\degtopure T'$. We say
    that $T$ \emph{degenerates} to $T'$ \emph{up to permutations}, if there is a
    permutation $\sigma\in \Sigma_3$ such that $\sigma\cdot T$ degenerates to
    $T'$. We
    denote this by $T\degto T'$. Both $\degtopure$ and $\degto$ are partial
    orders on the set of isomorphism classes of tensors. They can also be
    characterised
    using tensors with coefficient in power series, as
    in~\cite[(15.19)]{BurgisserBook} and using flattenings as
    in~\cite[Theorem~4.3]{CGZ}.

    Let $M, M'$ be $S$-modules of degree $m$. We say that $M$ \emph{degenerates} to $M'$ if
    there is a finitely generated flat $S[\![t]\!]$-module $\mathcal{M}$ such that
    $\mathcal{M}/t\mathcal{M}$ is isomorphic to $M'$ while
    $\mathcal{M}_{t}$ is isomorphic to $M(\!(t)\!)$. We denote this by
    $M\degto M'$ and call any such $\mathcal{M}$ a \emph{degeneration} of $M$
    to $M'$. If $M$, $M'$ come from
    $m$-dimensional spaces of matrices $E, E'\subseteq \End(C)$ as
    in~\ref{ssec:modulesAnd1Ageneric}, then the above is equivalent
    to saying that $E'$ lies in the closure of the orbit of $E\in \Gr(m, \End(C))$ under
    $\GL(C)$. Understanding degenerations of modules is quite subtle and
    equivalent to understanding the topology of the so-called Quot scheme of
    points. See~\S\ref{ssec:degenerationsDictionary} for comparing different
    types of degenerations.

    \subsection{How to read Diagrams~\ref{sec:diagram}-\ref{sec:nondiagram}}

        On Diagram~\ref{sec:diagram} arrows correspond to
        degenerations, we allow permutations of factors. Only ``minimal'' degenerations are drawn: all others
        are obtained by transitivity: if $T\degto T'$ and $T'\degto T''$,
        then $T\degto T''$.
        On Diagram~\ref{sec:nondiagram} the dashed arrows corresponds to
        \emph{nonexistence} of degenerations.
        For clarity, two classes of such arrows are omitted:
        \begin{itemize}
            \item By~\S\ref{ssec:stabilizer} below, there cannot be any
                degenerations going horizontally or going up. (The stabilizer
                dimension is written on the side of the diagram.)
            \item For $1_*$-generic tensors, denoted $T_{a, b}$ on the
                diagram, the number $a$ denotes the maximal number of summands
                in a direct sum decomposition of $T_{a, b}$. By Proposition~\ref{ref:degenerationSplitting:prop} below,
                if $T_{a, b}\degto T_{a', b'}$, then $a'\leq a$.
        \end{itemize}
        Moreover, only ``minimal''
        non-degenerations are drawn, others can be obtained by transitivity as
        follows: if we know from Diagram~\ref{sec:nondiagram} that $T\notdegto T'''$ and
        additionally from Diagram~\ref{sec:diagram} we know that $T\degto T'$ and $T''\degto T'''$, then we
        infer that $T'\notdegto T''$, as otherwise $T'\degto T'\degto T''\degto T'''$ would yield
        a contradiction.  For example, once we know that $T_{2, 7}$ does not
        degenerate to $T_{1,9}$, we can infer that it cannot degenerate to
        $T_{1,2}$, $T_{1,5}$, $T_{1,8}$, \ldots and neither any of $T_{1,19}$,
        $T_{1,18}$, $T_{1,12}$, \ldots can degenerate to $T_{1,9}$, because
        each of them is a degeneration of $T_{2, 7}$.

    \subsection{Preliminary results}\label{ssec:degenerationsPrelims}

In this section we gather several basic (and well-known) observations
regarding degenerations of tensors. We will refer to them in subsequent
arguments.

\subsubsection{Stabilizer Lie algebra}\label{ssec:stabilizer}

Let $T\in A\otimes B\otimes C$.  The \emph{stabilizer Lie algebra} consists of
triples $(X,Y,Z)\in \End(A) \times \End(B) \times \End(C)$, such that
\[
    X\cdot T + Y\cdot T + Z\cdot T = 0,
\]
where $X\cdot T := (X\otimes \id_B \otimes \id_C)(T)$, etc. This algebra is
equal to the kernel of the linear map
\[
    \End(A) \oplus \End(B)\oplus \End(C)\ni (X,Y,Z)\mapsto X\cdot T + Y\cdot T
    + Z\cdot T \in A \otimes B\otimes C,
\]
hence its dimension is (upper)-semicontinuous: for every $r$, the
locus of tensors having $(\geq r)$-dimensional stabilizer is closed.

\subsubsection{Being $1_A$-generic}\label{ssec:being1Ageneric}

    Recall that $\dim B = \dim C = m$.
    A tensor $T\in A\otimes B\otimes C$ is $1_A$-generic if and only if $\det
    T_A$ is not identically zero. It follows that being $1_A$-generic is open:
    no $1_A$-degenerate tensor can degenerate to a $1_A$-generic one.

    For example, $T_{2,7}$ is $1_A$-generic, but no $1_B$-, $1_C$-generic, so
    even after permuting factors it cannot degenerate to $T_{1, 9}$ which is
    both $1_A$- and $1_B$-generic. Similarly, $T_{3,2}$ is $1_A$- and
    $1_C$-generic, yet not $1_B$-generic, while $T_{1,3}$ and $T_{1,8}$ are
    $1_A$-, $1_B$-, $1_C$-, so $T_{3,2}$ cannot degenerate to $T_{1,3}$ or
    $T_{1, 8}$.

\subsubsection{Minimal number of generators}\label{ssec:mingens}

Let $M$ be an $S$-module of degree $m$. Let $\maxideal =
(x_1, \ldots ,x_{m-1})$.
Assume that there exists an integer $D\geq 0$ such that $\maxideal^DM = 0$.
In this case, Nakayama's Lemma~\ref{ref:Nakayama:lem} tells us that the minimal number of generators
of the $S$-module $M$ is equal to $\dim_{\kk} M/\maxideal M$. Suppose that
$M'$ is another $S$-module of degree $m$ with $\maxideal^D M' = 0$ for $D\gg 0$
and that $M$ degenerates to $M'$. Then
\begin{equation}\label{eq:semicontinuity}
    \dim_{\kk} M'/\maxideal M' \geq \dim_{\kk} M/\maxideal M.
\end{equation}
\newcommand{\Mbar}{\overline{\mathcal{M}}}%
Indeed, let $\mathcal{M}$ be a degeneration of $M$ to $M'$. Consider
$\Mbar :=\mathcal{M}/\maxideal \mathcal{M}$. This is a
finitely generated $\kk[\![t]\!]$-module such that $\Mbar/t \Mbar \simeq M'/\maxideal M'$
and $\dim_{\kk(\!(t)\!)} \Mbar_t = \dim_{\kk} M/\maxideal M$. The
inequality~\eqref{eq:semicontinuity} follows from the classification of
finitely generated $\kk[\![t]\!]$-modules: since $\kk[\![t]\!]$ is a principal
ideal domain, every such module has the form
$\kk[\![t]\!]^{\oplus r} \oplus \bigoplus_{j=1}^s
\frac{\kk[\![t]\!]}{(t^{e_j})}$ for some $r, s\geq 0$ and $e_1, \ldots
,e_s\geq 1$. The left-hand-side of~\eqref{eq:semicontinuity} is equal to $r+s$,
while the right-hand-side is $r$.

\subsubsection{Degenerations of $1_*$-generic tensors up to
$\Sigma_3$-action}\label{ssec:degenerationsDictionary}

    Let $T, T'\in A\otimes B\otimes C$ be $1_A$-generic and satisfy
    $A$-Strassen's equations. To prove that $T\notdegto T'$, in principle we need to consider $6$
    permutations $\sigma\in \Sigma_3$ and prove that $\sigma\cdot T\notdegtopure
    T'$ for each of them. Fortunately, it is not
    so:
    \begin{itemize}
        \item If $T$ is $1_B$-degenerate and $1_C$-degenerate, then by
            Subsection~\ref{ssec:being1Ageneric} the only possible
            degenerations are
            $T\degtopure T'$ and $\sigma\cdot T\degtopure T'$, where $\sigma\in
            \Sigma_3$ switches $B$ and $C$ coordinates. Let $M$, $M'$ be the
            modules associated to $T$ and $T'$, as in
            Section~\ref{prelims_tensors}. Then the two possible degenerations
            above translate to degenerations of $M\degtopure M'$ and
            $M^{\vee}\degtopure M'$ of modules. To have a degeneration
            $M^{\vee} \degtopure M'$ is the same as to have a degeneration
            $M\degtopure (M')^{\vee}$, because $M^{\vee\vee}$ is isomorphic to
            $M$.
        \item If $T$ is $1_B$-generic and $1_C$-degenerate, then in principle
            we could also swap $A$ and $B$ coordinates. However, $T$
            corresponds to a commutative algebra (see Subsection~\ref{prelims_tensors}), so such a swap does
            nothing. Hence also in this case we need to consider the
            degenerations as above.
        \item If $T$ is $1_A-$, $1_B-$, $1_C$-generic, then it corresponds to
            a Gorenstein algebra, so $T$ is
            isomorphic to a symmetric tensor~\cite[Proposition~5.6.2.1]{landsberg_complexity}, hence
            we need to consider only $T\degtopure T'$.
    \end{itemize}
    \begin{remark}
        The above considerations of course fail when $T'$ is 1-degenerate and
        indeed one of our degenerations is $\sigma\cdot T_{1,2}\degtopure \Tdeg{57}$,
        where $\sigma\in \Sigma_3$ is a three-cycle $A\to C\to B\to A$.
    \end{remark}

    \subsubsection{Degenerations of modules supported in several maximal
    ideals}\label{ssec:severalParts}

    \begin{proposition}[{\cite[THEOREM,
        p.291]{mazzola_generic_finite_schemes}}]\label{ref:degenerationSplitting:prop}
            Let $N = N_{1}\oplus \ldots \oplus N_{r}$ be a
            direct sum of $S$-modules, where the supports of $N_i$ and $N_j$
            are disjoint for $i\neq j$.

            Suppose that $M$ degenerates to $N$. Then there exists direct sum
            decomposition of $S$-modules $M = M_1\oplus \ldots \oplus
            M_{r}$ such that $M_i$ and $M_j$ have disjoint supports and $M_i$
            degenerates to $N_i$ for every $i=1,2, \ldots ,r$.
        \end{proposition}

        For example, $T_{2,3}$ corresponds to the cyclic module
        $\kk[x_1]/(x_1^4) \oplus \kk[x_4]/(x_4-1)$, where summands have degree
        $4$ and $1$, see~\S\ref{sssec:s2}, while
        $T_{2,2}$ corresponds to $\kk[x_1, x_2]/(x_1, x_2)^2 \oplus
        \kk[x_3]/(x_3 - 1)^2$, where summands have degree $3$ and $2$, so by
        Proposition~\ref{ref:degenerationSplitting:prop}
        the tensor $T_{2,3}$ cannot degenerate to $T_{2, 2}$. Similarly,
        $T_{3, 3}$ cannot degenerate to $T_{3, 2}$.

    \subsubsection{Submodules}

        In several important cases below, we rule our a degeneration by
        considering submodules of the module associated to a $1_A$-generic
        tensor $T$ that satisfies $A$-Strassen's equations.

        We review the construction here, in a somewhat naive way, which is
        sufficient for our purposes. Recall that to a tensor $T$ and full rank
        matrix $T(\alpha)$ we associate~\eqref{eq:Espace}
        a space of matrices $\Espace{T} \subseteq \End(C)$
        and the module $\modC$. A submodule of $\modC$ is a subspace
        $V\subseteq C$ closed under the action of elements of $\Espace{T}$.

        \begin{proposition}\label{ref:submodules:prop}
            Let $T\in A\otimes B\otimes C$ be concise $1_A$-generic and satisfy $A$-Strassen's equations.
            Suppose that $T$ degenerates to another concise $1_A$-generic
            $T'$. Let $M$, $M'$ be the modules associated to $T$, $T'$,
            respectively. Let $N\subseteq M$ be a submodule of degree $r$.
            Then there exists a submodule $N'\subseteq M'$ of degree $r$ and a degeneration
            $N\degtopure N'$.
        \end{proposition}
        We stress that the degeneration in the statement does not allow for
        any permutations; we require that $T'$ is in the closure of the
        $(\GL(A)\times \GL(B)\times \GL(C))$-orbit of $T$.
        \begin{proof}
            \def\curve{\mathcal{C}}
            \def\curvebar{\overline{\curve}}
            Let $\OO$ denote the orbit of $T$. The map $\GL(A) \times \GL(B)
            \times \GL(C)\to \OO$ is surjective.
            The tensor $T$ degenerates to $T'$, so we may pick a smooth curve $\curve\to
            \GL(A) \times \GL(B)\times \GL(C)$ such that $T'$ lies in the
            closure of $\curve \cdot T$.

            View $T(A^{\vee})\subseteq B\otimes C$ as a
            point of the Grassmannian $\Gr(m, B\otimes C)$. The group
            $\GL(B)\times \GL(C)$ acts on this Grassmannian and
            $[T'(A^{\vee})]$ lies in the closure of $\curve\cdot [T(A^{\vee})]$,
            where $\curve$ acts only by its $\GL(B)\times \GL(C)$ part.

            Pick an element $\alpha\in A^{\vee}$ such that
            $T'(\alpha)$ has full rank. By semicontinuity, for a nonempty open
            subset of $x\in \curve$, the element $(x \cdot T)(\alpha)$ has
            full rank as well. We restrict the curve $\curve$ to this open
            subset.
            Recall from~\eqref{eq:Espace} the space
            \[
                \Espace{T'} = T'(A^{\vee})T'(\alpha)^{-1}\subseteq \End(C)
            \]
            and its counterpart for $T$.  The point
            $[\Espace{T'}]\in \Gr(m, \End(C))$ lies in the closure of
            $\curve\cdot \Espace{T}$. Observe that here $\curve$ acts only by
            the $\GL(C)$-part, the $\GL(B)$-part of the action has cancelled
            out. Let $\curvebar\to \Gr(m, \End(C))$ be the smooth projective curve
            extending $\curve$ and suppose that $0\in \curvebar$ maps to
            $[\Espace{T'}]$.

            Consider now the module $N$ and view it as an element of $\Gr(r,
            C)$. We have a map $\curve \to \GL(C)$ and $\GL(C)$ acts on
            $\Gr(r, C)$, so we can  associate an element
            $[x\cdot N]\in \Gr(r, C)$ to every $x\in \curve$.
            The
            Grassmannian $\Gr(r, C)$ is projective and so the map $\curve \to
            \Gr(r, C)$ extends to a map from a smooth projective curve. By
            uniqueness it is
            $\curvebar\to \Gr(r, C)$.
            Let $[N']\in \Gr(r, C)$ be the image of $0$.

            For every $x\in \curve$, the subspace $x\cdot N$ is closed
            under the action of the matrices $x\cdot [\Espace{T}]$.
            By semicontinuity, the space $N'$
            is closed under the action of $\Espace{T'}$, so it is a submodule
            of $M'$.
        \end{proof}

\subsection{Obstructions to degenerations coming from submodules}

In this section we rule out degenerations using the following observation: if
a module $M$ degenerates to a module $M'$ and $N\subseteq M$ is a submodule,
then there exists a submodule $N'\subseteq M'$ which is a degeneration of $N$,
see Proposition~\ref{ref:submodules:prop}. If $M$ admits such an $N$ which has large annihilator
or requires many generators, then the same is true for $N'$. But such an $N'$
cannot exist in the cases below.

\begin{lemma}[{obstruction for $T_{1,11}$}]\label{ref:T11submodules:lem}
    Let $M_{11}$ be the module corresponding to the tensor $T_{1,11}$. Let
    $N\subseteq M_{11}^{\vee}$ be any submodule of degree $\dim_{\kk} N = 4$. Then $N$ is
    cyclic. Let $N\subseteq M_{11}$ be any submodule of degree four. Then
    $\ann(N)\cap \langle x_1, \ldots ,x_5\rangle$ is at most one-dimensional.
\end{lemma}

\begin{proof}
    We provide an elementary proof. The module $N$ corresponds to a
    $4$-dimensional subspace $V\subseteq \kk^{\oplus 5}$, which is closed
    under the action of matrices coming from $T_{1,11}^{\top}$.
    If this subspace
    contains $e_2$, then it also contains $x_4(e_2) = e_3$, $x_3(e_2) = e_4$,
    $x_2(e_2) = e_5$, so $N = \spann{e_2, e_3, e_4, e_5}$ is generated by
    $e_2$.
    Suppose that the subspace does not contain $e_2$. By dimension reasons, it
    intersects $\spann{e_1, e_2}$, so it contains an element $e_1 + \lambda
    e_2$ for some $\lambda\in \kk$. Then it also contains elements
    $x_5(e_1+\lambda e_2) = e_5$, $x_4(e_1 + \lambda e_2) = e_4 + \lambda
    e_5$, $x_3(e_1 + \lambda e_2) = e_3 + \lambda e_4$, so $N$ is generated by
    $e_1 + \lambda e_2$.

    The part for $N\subseteq M_{11}$ is quite similar. The submodule $N$
    intersects $\spann{e_1, e_2, e_3}$ in at least a $2$-dimensional subspace.
    If this subspace is $\spann{e_1, e_2}$ then the claim holds by direct
    check. If not, then $N$ contains an element of the form $e_3 +
    \lambda_1e_1 + \lambda_2e_2$, so its annihilator is contained in the
    annihilator of this element, hence
    \[
        \ann(N)\cap \langle x_1, \ldots ,x_5\rangle \subseteq \spann{x_1,
        x_2}.
    \]
    Analysing the intersection $N\cap \spann{e_1, e_2}$, we check that the
    containment is strict.
\end{proof}

\begin{corollary}\label{ref:notoT111:cor}
    There are no degenerations $T_{3,2} \degto T_{1,11}$, $T_{2,5} \degto
    T_{1,11}$, $T_{2,7}\degto T_{1,11}$.
\end{corollary}
\begin{proof}
    The tensors $T_{2,5}$, $T_{2,7}$ correspond to self-dual
    modules, so it is enough to prove that there are no degenerations
    $M_{2,5}$, $M_{2,7}$ to $M_{11}^{\vee}$. Consider $M_{2,5}$,
    which comes from an algebra (see~\eqref{eq:modulesM23M26}), and its submodule given by the maximal ideal
    $\maxideal$.
    By Proposition~\ref{ref:submodules:prop} this submodule degenerates to a degree four submodule
    $N$ of $M_{11}^{\vee}$. By Lemma~\ref{ref:T11submodules:lem}, the module
    $N$ is cyclic. By semicontinuity of minimal number of
    generators~\ref{ssec:mingens} also the maximal ideal of
    $M_{2,5}$ is cyclic, but this is not so, a contradiction. Same argument
    works for $M_{2,7}$ and its
    distinguished degree four submodule defined in Subsection~\ref{deg4}.

    The case $T_{3,2}$ is slightly different, since $T_{3,2}$ is not
    self-dual. This tensor corresponds to an algebra $A$ with a noncyclic maximal
    ideal, so as above we prove that there is no degeneration of $A$ to
    $M_{11}^{\vee}$. To prove that there is no degeneration of $A$ to
    $M_{11}$ take again the maximal ideal $\maxideal\subseteq A$. It is
    annihilated by two-dimensional space of variables. But it degenerates to a
    degree four submodule $N\subseteq M_{11}$ which by
    Lemma~\ref{ref:T11submodules:lem} does not have this property. A
    contradiction with semicontinuity of annihilators.
\end{proof}

\begin{lemma}[{obstruction for $T_{2,7}$}]\label{ref:notoT27:lem}
    There are no degenerations $T_{3,2} \degto T_{2,7}$, $T_{2,5} \degto
    T_{2,7}$.
\end{lemma}
\begin{proof}
    The module $M_{2,7} \simeq \kk \times N$ corresponding to $T_{2,7}$ is self-dual, so it is
    enough to show nonexistence of degenerations $M_{3,2}\degtopure
    M_{2,7}^{\vee}$, $M_{2,5}\degtopure M_{2,7}^{\vee}$. By
    Proposition~\ref{ref:degenerationSplitting:prop} these
    degenerations would come from degenerations of $\kk\times \kk[x, y]/(x,
    y)^2$ or $\kk[x,y]/(x^2, y^2)$ to $N$. To disprove their existence, we
    argue as above: reasoning as in
    Lemma~\ref{ref:T11submodules:lem}, we prove that every degree $3$
    submodule of $N$ is cyclic, and use semicontinuity and maximal ideals in
    the above algebras.
\end{proof}

\begin{lemma}[{obstruction for $T_{1,10}$}]\label{ref:T10submodules:lem}
    Let $M_{10}$ be the module corresponding to the tensor $T_{1,10}$.
    This is a self-dual module.
    Let
    $N\subseteq M_{10}$ be any submodule of
    degree $\dim_{\kk} N = 4$. Then $N$ is generated by at most two elements.
\end{lemma}
\begin{proof}
    Recall that $\kk e_1 \subseteq M_{10}$ is also a submodule. If
    $e_1$ does not lie in $N$, then $N$ is isomorphic to $M_{10}/\kk e_1$,
    which is generated by two elements, $e_2$, $e_3$.
    Suppose that $e_1$ is an element of $N$ and take it as a generator.
    The module $N/\kk e_1$ is cyclic by the same argument as in
    Lemma~\ref{ref:T11submodules:lem}. Let $\bar{v}\in N/\kk e_1$ be its generator and let
    $v\in N$ be any lift, then $N$ is generated by $e_1$ and $v$.
\end{proof}

\begin{corollary}\label{ref:notoT110:cor}
    There are no degenerations $T_{1,8} \degto T_{1,10}$, $T_{2,6} \degto
    T_{1,10}$.
\end{corollary}
\begin{proof}
    Let $M_{10}$ be the module corresponding to $T_{1,10}$. This module is
    self-dual, so it is enough to prove nonexistence of degenerations
    $M_{8}\degtopure M_{10}$ and $M_{2,6}\degtopure M_{10}$. The module $M_8$
    corresponds to the algebra $\kk \times \kk[x, y, z]/(x, y, z)^2$, while
    $M_{10}$ corresponds to the algebra $\kk[x, y, z]/(xy,yz,zx,
    x^2-y^2,x^2-z^2)$. They both have ideals which are not generated by two
    elements. This yields a contradiction with semicontinuity and
    Lemma~\ref{ref:T10submodules:lem}.
\end{proof}

\begin{lemma}[{obstruction for $T_{1,19}$}]\label{ref:T19submodules:lem}
    Let $M_{19}$ be the module corresponding to the tensor $T_{1,19}$. This is
    a self-dual module. Let
    $N\subseteq M_{19}$ be any submodule of
    degree $\dim_{\kk} N = 4$, then the subspace $\ann(N)\cap \langle x_1, \ldots
    ,x_5\rangle$ is at most one-dimensional.
\end{lemma}
\begin{proof}
    This is a case-by-case analysis. Consider $N\subseteq M_{19}$.
    Then $N$ intersects the subspace $\spann{e_1, e_2}$. If for some
    $\lambda\in \kk$ the element $e_1-\lambda e_2$ belongs to $N$, then $\ann
    N$ is contained in $\ann(e_1-\lambda e_2) = x_4$. If not, then $N$
    contains $e_2$, hence also $x_4(e_2) = e_4$ and $x_3(e_2) = e_5$. We
    assumed that it does not contain $e_1$, so it contains some element $e_3 -
    \mu e_1$ and as a result, we get $\ann(N)\subseteq \ann(e_2)\cap \ann(e_3
    - \mu e_1) = \spann{x_2}$.
\end{proof}

\begin{corollary}\label{ref:notoT119:cor}
    There is no degeneration $T_{3,2} \degto T_{1,19}$.
\end{corollary}
\begin{proof}
    The ideal $\kk\times \kk\times \spann{x, y} \subseteq \kk\times \kk\times
    \kk[x, y]/(x, y)^2$ is a submodule of $M_{3,2}$ and this submodule is
    annihilated by a two-dimensional space of variables. Coupled with
    Lemma~\ref{ref:T19submodules:lem} and semicontinuity, this proves that no
    degeneration $M_{3,2}\degtopure M_{19}$ can happen. Since $M_{19}$ is
    self-dual, this yields the claim.
\end{proof}

\begin{lemma}\label{ref:nofromT28toT112:lem}
    There is no degeneration $T_{2,8} \degto T_{1,12}$.
\end{lemma}
\begin{proof}
    The case is similar to the above. The module $M_{2,8}$ is self-dual, so it
    is enough to prove non-existence of degeneration $M_{2,8}\degtopure
    M_{1,12}$ of modules. The subspace $N = \spann{e_2, \ldots
    ,e_5}\subseteq M_{2,8}$ is a degree four submodule annihilated by a
    $2$-dimensional space of variables. Arguing as in
    Lemma~\ref{ref:T11submodules:lem}, we check that no submodule of
    $M_{1,12}$ has this property.
\end{proof}

\begin{remark}\label{ex:T22toT14}
    In the $1_A$-, $1_B$-generic case, instead of submodules, we may also
    consider subspaces with multiplication (also known as non-unital algebras). This is useful in one case: the tensor
    $T_{2,2}$ corresponds to an algebra $\kk[x, y]/(x, y)^2 \times
    \kk[z]/(z^2)$, which has a $3$-dimensional subspace $\spann{x, y, z}$ with
    zero multiplication. The tensor $T_{1,4}$ corresponds to the algebra
    $\kk[x, y]/(x^3, xy, y^3)$ that admits no such subspace, hence
    $T_{2,2}$ does not degenerate to $T_{1, 4}$.
\end{remark}

\subsection{Obstructions to degenerations coming from many low rank matrices}

\newcommand{\inters}[3]{d_{#1, #2}^{#3}}

    For a concise tensor $T\in A\otimes B\otimes C$ and a fixed integer $1\leq
    r\leq 4$, we may consider a projective subspace
    $\mathbb{P}(T(A^{\vee}))\subseteq \mathbb{P}(B\otimes C)$ and its
    intersection with the projective variety of matrices of rank at most $r$,
    that is, with the $r$-th secant to the Segre variety in
    $\mathbb{P}(B\otimes C)$. Let $\inters{T}{r}{A}-1$ denote the dimension of this
    intersection, so that $\inters{T}{r}{A}$ denotes the dimension of the
    intersection on the affine level. We define $\inters{T}{r}{B}$,
    $\inters{T}{r}{C}$ analogically. By
    semicontinuity, for $T\degtopure T'$, we have $\inters{T'}{r}{A}\geq
    \inters{T}{r}{A}$
    and same for two other coordinates.

    \begin{lemma}\label{ref:nofromT18toT55:lem}
        There is no degeneration $T_{1,8}\degto \Tdeg{55}$.
    \end{lemma}
    \begin{proof}
        We compute directly that $\inters{T_{1,8}}{3}{\star}$ is equal to $4$ for
        every $\star\in\{A, B,C\}$, while $\inters{\Tdeg{55}}{3}{\star}$ is equal
        to $3$ for every  $\star\in\{A, B,C\}$. This violates semicontinuity
        of $\inters{-}{3}{-}$,
        even after permuting factors.
    \end{proof}

    \begin{lemma}\label{ref:nofromT26toT54:lem}
        There is no degeneration $T_{2,6}\degto \Tdeg{54}$.
    \end{lemma}
    \begin{proof}
        We compute directly that the triple
        \[
            \{\inters{T_{2,6}}{2}{\star}\ |\ \star\in \left\{ A,B,C \right\}\}
        \]
        is
        $\left\{3,1,3\right\}$, while the corresponding triple for $\Tdeg{54}$
        is $\left\{ 2,2,2 \right\}$. This violates semicontinuity of
        $\inters{-}{2}{-}$,
        even after permuting factors.
    \end{proof}

    \subsection{Obstructions to degenerations coming from the
        \BBname{} decomposition}

        \newcommand{\inn}{\operatorname{in}}
        \newcommand{\gr}{\operatorname{gr}}
        \newcommand{\Quot}{\OpQuot_{5}^r}

        In this section we show the non-existence of two degenerations which
        we find to be the hardest to discard. In the case $T_{1,5}\degto T_{1,13}$ none of
        the invariants known to us prohibits degeneration. In the case
        $T_{1,4}\degto T_{1,12}$ the degeneration is prohibited by considering
        the non-semisimple part of the stabilizer Lie algebra (we thank Joseph
        Landsberg for this observation).
        We handle both cases using the method of \BBname{} decomposition.
        In essence, it says that if a degeneration existed, if would have a
        particularly easy shape (called the associated graded), which is then
        possible to rule out by hand, see
        Proposition~\ref{ref:noeasydegeneration:prop}.
        Finding a nice invariant that rules out both cases would be a very
        useful simplification.

        The \BBname{} decomposition is a tool of moduli spaces in algebraic
        geometry and it is quite intricate. Below we try to summarize it,
        however we apologise for being brief.
        To use the \BBname{} decomposition, we fix a standard grading on $S$, where $\deg(x_i) = 1$.
        This yields a grading on $S^{\oplus r}$, for every $r$. For an element
        $k\in S^{\oplus r}$ we can decompose it into homogeneous parts, $k =
        k_0 + k_1 + \ldots$. The grading corresponds to an action of the torus
        $\Gmult := \Spec(\kk[t^{\pm 1}])$, where $t\cdot k_i = t^{-i}k_i$.

        The \emph{initial form} of $k$ is $\inn(k) := k_i$, where
        $i$ is the largest index such that $k_i\neq0$. For example, the
        initial form of $x_0^2 + x_1$ is $x_0^2$. By convention, $\inn(0) =
        0$.

        For $K\subseteq S^{\oplus
        r}$ a submodule, the \emph{initial module} $\inn(K)$ is the
        $\kk$-vector space spanned by initial forms of elements of $K$. It is
        always an $S$-submodule. For
        $M= S^{\oplus r}/K$ a quotient module, the \emph{associated graded
        module} is $\gr(M) := S^{\oplus r}/\inn(K)$. The notation is a bit
        abusive in that the associated
        graded module depends not only on $M$ but on its presentation as
        $S^{\oplus r}/K$. In terms of the
        $\Gmult$-action, the associated graded module is the limit of the
        $\Gmult$-orbit of $[K]$ at zero, where $[K]$ is the point of the $\Quot$
        scheme.

        The associated graded module can also be described in a different
        manner. Let $M$ and $K$ be as above and let $M_{\geq i} :=
        (S^{\oplus r})_{\geq i} + K$. By construction, the multiplication by
        every variable $x_j$ sends $M_{\geq i}$ to $M_{\geq i+1}$.
        The associated graded module of $M$ can be identified with the vector space
        \[
            \gr(M) := \bigoplus_i \frac{M_{\geq i}}{M_{\geq i+1}}
        \]
        which is an $S$-module via the maps $x_j\colon M_{\geq i}/M_{\geq
        i+1}\to M_{\geq i+1}/M_{\geq i+2}$ for every $j$.

        \begin{proposition}[no associated-graded
            degeneration]\label{ref:noeasydegeneration:prop}
            Let $N\subseteq M_{1,4}^{\vee}$ be a submodule of degree two. Then the
            associated graded module $\gr\left(M_{1,4}^{\vee}\right)$ is
            isomorphic to $M_{1, 15}$ as a graded module.
        \end{proposition}
        \begin{proof}
            The multiplication on $M_{1,4}^{\vee}$ is determined by the
            transpose of $T_{1,4}$. From the matrices, it follows that the
            only submodules of degree two are $\spann{e_1, e_2}$ and
            $\spann{e_1, e_4}$; in the notation of~\eqref{eq:localalgebras} these are
            $\spann{1^*, x^*}$ or $\spann{1^*, y^*}$. Both choices are
            equivalent: swapping $x$ and $y$ interchanges them. Suppose we
            took $\spann{1^*, x^*}$. In the associated graded, the only
            nonzero multiplications are $x^2\cdot (x^2)^* = 1^*$, $x\cdot
            (x^2)^* = x^*$, $y\cdot y^* = 1^*$ and $y^2\cdot (y^2)^* = 1^*$.
            It follows that we get a module isomorphic to $M_{1,15}$.
        \end{proof}

        The following is the key general result that will allow us to restrict
        to degenerations given by the associated graded construction.
        \begin{proposition}[{\cite[Chapter~5]{components}, see also~\cite{jelisiejew_sienkiewicz__BB}}]\label{ref:BBmain:prop}
            Let $M = S^{\oplus r}/K$ be a zero-dimensional $S$-module and
            assume that $K$ is homogeneous. Assume additionally that
            \begin{equation}\label{eq:nopositivetangents}
                \Hom_S(K, M)_{\geq 0} = 0.
            \end{equation}
            Then there exists an open subset $U\subseteq
            \Quot$ and a $\Gmult$-invariant morphism $p\colon U\to U^{\Gmult}$ that sends
            every module $[M]\in U$ to $[\gr(M)]\in U^{\Gmult}$.
        \end{proposition}

        \begin{proposition}\label{ref:nofromT14toT112:prop}
            There is no degeneration $T_{1,4}\degto T_{1,12}$.
        \end{proposition}
        \begin{proof}
            \def\curve{X}
            First, assume that $M_{1,4}$ degenerates to $M_{1,12}$. The
            maximal ideal of $M_{1,4}$ is a submodule of degree four annihilated by a two-dimensional
            space of variables. The module $M_{1,12}$ admits no such
            submodule, as we already asserted in
            Lemma~\ref{ref:nofromT28toT112:lem}, so a degeneration cannot
            exist.

            We now proceed to disprove the existence of
            $M_{1,4}^{\vee}\degtopure M_{1,12}$. This requires more care: the analogue of the argument
            above does not work, because $M_{1,12}/\kk e_5$ is a degree four quotient module
            annihilated by a two-dimensional space of variables.

            The module $M_{1,12}$ is generated by $e_1$, $e_2$, $e_3$, which
            means that it is isomorphic to $S^{\oplus 3}/K$.
            The kernel $K$ is homogeneous and in fact
            $H_{M_{1,12}} = (3, 2)$. A direct computation shows
            that~\eqref{eq:nopositivetangents} is satisfied. Let $U$ be as in
            Proposition~\ref{ref:BBmain:prop}. We shrink $U$ if necessary, so
            that it contains concise modules only.

            Assume that $M_{1,4}^{\vee}$ does degenerate to $M_{1,12}$. A
            degeneration yields a pointed curve $f\colon (\curve,0)\to \Quot$ which
            sends each point except zero to a module isomorphic to
            $M_{1,4}^{\vee}$ and sends $0$ to $M_{1,4}^{\vee} = S^{\oplus
            3}/K$. After replacing the curve by $f^{-1}(U)$, we get a curve
            $f\colon (\curve, 0) \to U$. Composing it with $p\colon U\to
            U^{\Gmult}$, we get a curve $p\circ f\colon (\curve, 0)\to
            U^{\Gmult}$.

            Consider any $x\in \curve$. The point $(p\circ f)(x)$ is a concise graded
            module with Hilbert function $(3, 2)$.
            The closure of the $\Gmult$-orbit in $p^{-1}((p\circ f)(x))$.
            By
            Proposition~\ref{ref:BBmain:prop}, it is a degeneration of
            $M_{1,4}^{\vee}$, in particular it is isomorphic to one of the
            modules $M_{1,12}$, $M_{1,13}$, $M_{1,14}$.

            By semicontinuity of the stabilizer, there is an open
            subset of $x$ which yield $M_{1,12}$. By semicontinuity, a
            general point $x$ corresponds to $M_{1,12}$. Fix any such $x$.
            The $\Gmult$-orbit of $x$ yields degeneration of
            $M_{1,4}^{\vee}$ to a module isomorphic to $M_{1,12}$ given by the associated graded
            construction.

            Denote by $M_{\geq i}\subseteq M_{1,12}$ the
            elements of the filtration.
            The module $M_{1,12}$ has Hilbert function
            $(3, 2)$, so it follows that
            \[
                \dim_{\kk} M_{\geq 0}/M_{\geq 1} = 3,\quad \dim_{\kk}
                M_{\geq 1}/M_{\geq 2} = 2,\quad \dim_{\kk} M_{\geq
                2}/M_{\geq 3} = 0,
            \]
            hence $M_{\geq 0} = M$, the submodule $M_{\geq 1}$ has degree two and $M_{\geq
            2} = 0$. But precisely such a degeneration was ruled out in
            Proposition~\ref{ref:noeasydegeneration:prop}.
        \end{proof}

        \begin{proposition}\label{ref:nofromT15toT113:prop}
            There is no degeneration $T_{1,5}\degto T_{1,13}$.
        \end{proposition}
        \begin{proof}
            The proof is analogous to
            Proposition~\ref{ref:nofromT14toT112:prop}. First, a degeneration
            $M_{1,5}\degtopure M_{1,13}$ does not exist, because the maximal ideal
            of the algebra $M_{1,5}$ is annihilated by a $3$-dimensional space
            of variables, which $M_{1,13}$ admits no degree four submodule
            annihilated by such a large subspace.

            To rule out the degeneration $M_{1,5}\degtopure M_{1,13}$, we check
            that $M_{1,13}$ is isomorphic to $S^{\oplus 3}/K$ for $K$
            homogeneous, that it satisfies~\eqref{eq:nopositivetangents} and
            has Hilbert function $(3, 2)$. As above, we reduce to proving that
            there is no degeneration $M_{1,5}\degtopure M_{1,13}$ given by the
            associated graded construction. As in
            Proposition~\ref{ref:noeasydegeneration:prop}, we check that there
            are only two degree two submodules, hence only two possible
            degenerations. One of them yields $M_{1,15}$ and the other yields
            a non-concise module.
        \end{proof}

    \subsection{Obstruction to degenerations to $\Tdeg{58}$, \ldots ,
        $\Tdeg{54}$.}

        Consider a degeneration $T\degtopure T'$ of minimal border rank tensors in
        $A\otimes B\otimes C$. This can be interpreted as a
        family $T_t\in A\otimes B\otimes C$. Assume
        additionally that both $T$ and $T'$ are 111-sharp, that is, that they
        both have \emph{exactly} $m$-dimensional 111-algebras. In this case, the
        degeneration induces a degeneration of 111-algebras inside
        $\End(A)\times \End(B)\times \End(C)$, see~\cite[(1.3)]{concise},
        and consequently, a degeneration
        of modules $\modA$, $\modB$, $\modC$.
        By~\cite[\S1.4.1]{concise} all minimal border rank tensors for $m =
        \dim_{\kk} A \leq 5$ are 111-sharp.

        When $T$, $T'$ come from concise modules $M$, $M'$, the obtained degenerations
        of modules are quite
        tautological: we obtain degenerations $M\degtopure M'$,
        $M^{\vee}\degtopure
        (M')^{\vee}$ and additionally a degeneration of algebras
        $S/\ann(M)\degtopure
        S/\ann(M')$, because $\modA \simeq S/\ann(M)$ for $T$ an similarly for
        $T'$, by~\cite[Theorem 5.3]{concise}. In contrast, for a 1-degenerate
        $T'$, the above becomes very
        helpful.

        In the table below, we list the modules coming from the five
        1-degenerate tensors. One can compute them by hand, or using our
        package, see Appendix~\ref{sec:code}.
        \begin{equation}\label{eq:modules}
            \begin{tabular}[<+position+>]{ll c c c}
                && $\modA$ & $\modB$ & $\modC$\\
                \midrule
                $\Tdeg{58}$ && $M_{11}$ & $M_{11}$ & $M_{11}$\\
                $\Tdeg{57}$ && $M_{15}$ & $M_{12}$ & $M_{12}$\\
                $\Tdegtw{}$ && $M_{15}$ & $M_{13}$ & $M_{15}$\\
                $\Tdeg{55}$ && $M_{15}$ & $M_{15}$ & $M_{15}$\\
                $\Tdeg{54}$ && $M_{15}$ & $M_{15}$ & $M_{15}$
            \end{tabular}
        \end{equation}

        \begin{proposition}\label{ref:nodegstodegenerate:prop}
            There are no degenerations $T_{3,2}\degto \Tdeg{58}$,
            $T_{2,5}\degto \Tdeg{58}$, $T_{2,7}\degto \Tdeg{58}$,
            $T_{1,4}\degto \Tdeg{57}$, $T_{2,8}\degto \Tdeg{57}$,
            $T_{1,5}\degto \Tdegtw{}$, $T_{1,11}\degto \Tdeg{54}$.
        \end{proposition}

        \begin{proof}
            The tensor $T_{3,2}$ comes from an algebra $A$, so all its coordinate
            modules are isomorphic to $A$ or $A^{\vee}$. By
            Corollary~\ref{ref:notoT111:cor}, none of these modules
            degenerates to $T_{1,11}$, which are the coordinate modules of
            $\Tdeg{58}$, see~\eqref{eq:modules}. This proves that no degeneration
            $T_{3, 2}\degto \Tdeg{58}$ exists.

            The proof for $T_{2,5}\degto
            \Tdeg{58}$ is the same.
            For $T_{1,4}\degto \Tdeg{57}$ and $T_{1,5}\degto \Tdegtw{}$, the
            proof is again the same, using
            Propositions~\ref{ref:nofromT14toT112:prop}-\ref{ref:nofromT15toT113:prop}
            to get non-degenerations of modules.

            For $T_{2,7}\degto \Tdeg{58}$ the proof is
            similar, the only subtlety is that $T_{2,7}$ corresponds to a
            module $M_{2,7}$, so its coordinate modules are $M_{2,7}$,
            $M_{2,7}^{\vee}$, and $S/\ann(M_{2,7}) \simeq M_{2,6}$. None of
            these degenerates to $M_{11}$. The same argument works for
            $T_{2,8}\degto \Tdeg{57}$, using
            Lemma~\ref{ref:nofromT28toT112:lem}.

            To prove that no $T_{1,11}\degto \Tdeg{54}$ exists, we recall that
            the module $M_{11}$ is annihilated by the square of the maximal
            ideal, so that $S/\ann(M_{11})$ is isomorphic to $M_{9}$. This is a
            coordinate module of $T_{1,11}$, but it does not degenerate to
            $M_{15}$, because of stabilizer dimension. This concludes the
            proof.
        \end{proof}

        \section{Refined classification of minimal border rank
        tensors}\label{ssec:upToIsomorphism}

        In this section we complete the proof of Theorem~\ref{theorem_final} by giving the
classification of minimal border rank tensors for $m\leq 5$ over $\mathbb{C}$ up to action of
$\GL(A)\times \GL(B)\times \GL(C)$, that is, without allowing permutations of
factors.
\begin{proof}[Final part of proof of Theorem~\ref{theorem_final}]
    We do the full proof for the case $m = 5$, while the smaller cases can be
    deduced from it and are much easier anyway.

    We begin with the $1_*$-generic tensors. For this case, we rely on the proof of classification of subspaces, as given in
    Subsection~\ref{classification_subspaces}. A $1_A$-generic tensor is
    isomorphic to a symmetric tensor if and only if it is $1_B$-, $1_C$-generic. There are $10$ such
    cases $T_{5, 1}$, $T_{4, 1}$, $T_{3,1}$, $T_{3, 3}$, $T_{2,1}$,
    $T_{2,3}$, $T_{2,5}$, $T_{1,1}$, $T_{1,3}$, $T_{1,8}$, they yield $10$
    isomorphism types. A $1_A$-,
    $1_B$-generic tensor $T$ which is not $1_C$-generic comes from a multiplication in an algebra, hence is
    isomorphic to $T$ with swapped first two coordinates. There are $10$ such
    tensors $T_{3,2}$, $T_{2,2}$, $T_{2,4}$, $T_{2,6}$, $T_{1,2}$, $T_{1,4}$,
    $T_{1, 5}$, $T_{1,6}$, $T_{1, 7}$, $T_{1, 9}$. The orbits of $\Sigma_3$
    acting on them have three elements, so these yield $10\cdot 3$ isomorphism types.
    We are left with $12$ cases of $1_A$-generic, but not $1_B$- nor
    $1_C$-generic tensors. Exactly six of them come from self-dual modules
    (see Subsection~\ref{classification_subspaces}), so
    they are (up to isomorphism) invariant under transposing and yield $3\cdot
    6$ isomorphism types. The other six yield $6\cdot 6$ isomorphism types. In
    total we obtain $94$ isomorphism types of $1_*$-generic tensors.

    We have to deal with five $1$-degenerate tensors. It is convenient to
    keep Table~\eqref{eq:modules} in mind. The
    tensor~$\Tdeg{58}$ is the unique among them which is isomorphic to a symmetric one,
    see~\cite[p.2478]{concise}. The tensors $\Tdeg{57}$, $\Tdeg{55}$,
    $\Tdeg{54}$ are easily seen to be isomorphic to their transpositions. They
    are not symmetric, so each of them yields $3$ isomorphism types. From
    Table~\eqref{eq:modules} it follows that $\Tdegtw{}$ admits at most a
    transposition symmetry.  From \cite[Theorem~7.3]{concise} it follows that
    after we fix one coordinate, there are exactly two isomorphism types that
    after permutation yield $\Tdegtw{}$. This can happen only if
    $\Tdegtw{}$ indeed admits a transposition symmetry. It follows that the
    $1$-degenerate tensors contribute $1+4\cdot 3 = 13$ isomorphism classes,
    which gives in total $107$ isomorphism classes.

    For $m=4$ we get $6$ symmetric tensors, $3$ tensors which are $1_A$-,
    $1_B$-generic and $2$ tensors which are only $1_A$-generic, corresponding
    to self-dual modules. This yields $6 + 3\cdot 3 + 2\cdot 3 = 21$
    isomorphism classes. For $m=3$ we obtain $3$ symmetric and one $1_A$-,
    $1_B$-generic tensors, so $3+1\cdot 3 = 6$ isomorphism types.
\end{proof}

\section{Existence of $1$-degenerate tensors}\label{degenerate}
In this section, we use the correspondence between tensors of minimal border rank and bilinear maps between modules, based on the 111-algebra introduced in \cite{concise}, to translate the claim that there are no 1-degenerate tensors of minimal border rank in $\kk^m \otimes \kk^m \otimes \kk^m$ for $m \leq 4$ into the claim that there no maps bilinear maps between module satisfying certain conditions. Then we use the auxiliary classification of concise local $S$-modules of degree $m \leq 4$ from Subsection \ref{classification_local} to prove this.

Recall that the classification in question states that:
\begin{itemize}
    \item[$m = 1:$] There are only cyclic modules. They are simultaneously cocyclic. 
    \item[$m = 2:$] There are only cyclic modules. They are simultaneously cocyclic. 
    \item[$m = 3:$] There are some cyclic and some cocyclic modules. 
    \item[$m = 4:$] There are some cyclic and some cocyclic modules. There are also some self-dual modules that are minimally generated by 2 elements.
\end{itemize}

The following lemma and its corollary show that the maps obtained from concise 111-abundant tensors decompose well into maps of local modules for $m \leq 4$.

\begin{lemma}
    Let $\mathcal{A}$ be a commutative unital $\kk$-algebra of degree $m$ and let $M$ be a concise $\mathcal{A}$-module of degree $m$. Choose a surjection $S \to \mathcal{A}$ which maps $S_{\leq 1}$ bijectively onto $\mathcal{A}$. If $m \leq 4$, then for each maximal ideal $\maxideal \subset S$ the $\mathcal{A}_\maxideal$-module $M_\maxideal$ is concise and $\dim_\kk \mathcal{A}_\maxideal = \dim_\kk M_\maxideal$.
\end{lemma}
\begin{proof}
    If $\dim_\kk M_\maxideal = 4$, then $\dim_\kk \mathcal{A}_\maxideal \leq \dim_\kk \mathcal{A} = 4 = \dim_\kk M_\maxideal$. If $\dim_\kk M_\maxideal \leq 3$, then $M_\maxideal$ is cyclic or cocyclic by the classification, so $\dim_\kk \End M_\maxideal =\dim_\kk M_\maxideal$. The $\mathcal{A}_\maxideal$-module $M_\maxideal$ is concise, so $\dim_\kk \mathcal{A}_\maxideal \leq \dim_\kk \End M_\maxideal =  \dim_\kk M_\maxideal$. Therefore $\dim_\kk \mathcal{A}_\maxideal \leq \dim_\kk M_\maxideal$ for each maximal ideal. We also know that $\sum_\maxideal \dim_\kk \mathcal{A}_\maxideal = \dim_\kk \mathcal{A} = \dim_\kk M = \sum_\maxideal \dim_\kk M_\maxideal$, so $\dim_\kk \mathcal{A}_\maxideal = \dim_\kk M_\maxideal$ for each maximal ideal.
\end{proof}

\begin{corollary}\label{degenerate_dimension}
    Let $\varphi \colon M \otimes_S N \to P$ be a surjective non-degenerate map corresponding to a concise $111$-abundant tensor $T_\varphi$. By Lemma \ref{maps_decomposition} it decomposes as a direct sum of surjective non-degenerate maps $\varphi_\maxideal\colon M_\maxideal \otimes_S N_\maxideal \to P_\maxideal$. If $m \leq 4$, then $\dim_\kk M_\maxideal = \dim_\kk N_\maxideal = \dim_\kk P_\maxideal$ for each $\maxideal \subset S$ and $M_\maxideal, N_\maxideal, P_\maxideal$ are concise.
\end{corollary}

We will use general results introduced in Subsection \ref{maps} and the classification of local concise modules to prove Theorem \ref{theorem_degenerate}. The result holds for any algebraically closed field $\kk$ with $\charr \kk \neq 2$.

\begin{proof}[Proof of Theorem~\ref{theorem_degenerate}]
    Let $M, N, P$ be concise $S$-modules of degree $m$ and let $\varphi \colon M \otimes_S N \to P$ be a surjective non-degenerate map. The module $P$ is a local module of degree 4 or it decomposes as a direct sum of at most one local module of degree 3 and local modules of degrees at most 2. 

\begin{enumerate}
    \item In the first case it follows from the classification that each $M, N, P$ is cyclic, cocyclic, or minimally generated by 2 elements. If at least one of $M, N, P$ is cyclic or cocyclic then we conclude by Corollary \ref{maps_symmetry} and Corollary \ref{degenerate_dimension}. We will show that the other case cannot hold.

    Assume that $M, N, P$ are minimally generated by 2 elements. Let $e_1, e_2$ and $e_3, e_4$ be minimal generators of $M$ and $N$. By Lemma \ref{maps_nakayama} and Lemma \ref{maps_cyclic} the map $\overline{\varphi}(e_1, -) \colon N / \maxideal N \to P / \maxideal P$ cannot be surjective. We know that $\dim_\kk P / \maxideal P = 2$, so $\overline{\varphi}(e_1, e_3), \overline{\varphi}(e_1, e_4)$ must be linearly dependent. Applying the same argument for $e_2, e_3, e_4$ and in case of need $e_1 + e_2$ or $e_3 + e_4$ shows that in fact all $\overline{\varphi}(e_1, e_3), \overline{\varphi}(e_1, e_4), \overline{\varphi}(e_2, e_3), \overline{\varphi}(e_2, e_4)$ are linearly dependent. It follows that the image of $\varphi$ is at most $3$-dimensional, so $\varphi$ is not surjective.
    
    \item In the second case the classification combined with Lemma \ref{maps_(co)cyclic} imply that $P$ is cyclic or cocyclic because local modules of degree 3 are cyclic or cocyclic and local modules of degrees at most 2 are simultaneously cyclic and cocyclic. We conclude by Corollary \ref{maps_symmetry}.
\end{enumerate}. 
\end{proof}

\appendix
\section{Code}\label{sec:code}

Macaulay2 computations are included with the arXiv submission of this paper,
as an auxiliary file \texttt{SmallMinimalBorderRankTensors.m2}. This is a
Macaulay2 package, which can be loaded using
\texttt{loadPackage("SmallMinimalBorderRankTensors")}.

The variable \texttt{TensorList} contains a list of tensors (in matrix
notation) together with their names. Additionally, the table
\texttt{TensorInMatrixForm} allows for quick access to a given tensor, for
example to get $\Tdegtw$, use
\begin{verbatim}
TensorInMatrixForm_{56}
\end{verbatim}
note the curly braces. The function \texttt{matrixFormToTensor} yields the
tensor form, for example \texttt{matrixFormToTensor TensorInMatrixForm\_{56}}
yields
{\small\begin{verbatim}
a b c  + a b c  + a b c  + a b c  + a b c  + a b c  + a b c  + a b c  + a b c  + a b c
 1 1 1    1 2 2    2 1 3    1 3 3    3 1 4    4 2 4    1 4 4    5 5 4    5 1 5    5 2 5
\end{verbatim}}
The function \texttt{matrixFormToModule} applied to one of our $1_A$-generic
tensors, yields the corresponding module. Some important invariants of tensors
are obtained using the functions
\begin{itemize}
    \item \texttt{stabilizerDimension}, which yields the stabilizer dimension,
    \item \texttt{oneoneonematrixspace}, which yields the 111-algebra inside
        $\End(A) \times \End(B) \times \End(C)$,
    \item \texttt{coordinateTensorsInMatrixForm}, which yields the matrix
        forms of the multiplication tensors of $\modA$, $\modB$, $\modC$,
    \item \texttt{coordinateModules}, which yields $\modA$, $\modB$, $\modC$,
        as modules.
\end{itemize}

The variable \texttt{DegenerationList} includes every degeneration in
the diagram~\ref{sec:diagram}. A desired degeneration can be looked up by name
parameter (\texttt{degName}), for example by

\begin{verbatim}
mydeg = first select(DegenerationList, el -> el#degName == {{1,2}, {1,18}})
\end{verbatim}

Once obtained, the variables \texttt{mydeg\#Laction}, \texttt{mydeg\#Raction},
\texttt{mydeg\#Vaction} contain the matrices corresponding to the \textbf{L}eft action
(on the $B$ coordinate), the \textbf{R}ight action (on the $C$ coordinate) and the
action on \textbf{V}ariables (on the $A$ coordinate), respectively. To get the
family itself, use \texttt{degenerationAsFamily}.

\section{Classification and degenerations for $m\leq 4$}\label{sec:smallm}

In this appendix we present a list of isomorphism types and degenerations of
minimal border rank tensors in $\kk^m\otimes \kk^m \otimes \kk^m$ for $m\leq
4$.

Case $m=2$.
\[
    \begin{bmatrix}
        x_0 & 0 \\
        x_1 & x_0 \\
    \end{bmatrix},\quad
    \begin{bmatrix}
        x_0 & 0 \\
        0 & x_0 + x_1 \\
    \end{bmatrix}
    \]
Case $m=3$.
    \[
    \begin{split}
        \begin{bmatrix}
            x_0 & 0 & 0 \\
            x_1 & x_0 & 0 \\
            x_2 & x_1 & x_0 \\
        \end{bmatrix},\quad
        \begin{bmatrix}
            x_0 & 0 & 0 \\
            x_1 & x_0 & 0 \\
            x_2 & 0 & x_0 \\
        \end{bmatrix},\quad
        \begin{bmatrix}
            x_0 & 0 & 0 \\
            x_1 & x_0 & 0 \\
            0 & 0 & x_0 + x_2 \\
        \end{bmatrix},\quad
        \begin{bmatrix}
            x_0 & 0 & 0 \\
            0 & x_0 + x_1 & 0 \\
            0 & 0 & x_0 + x_2 \\
        \end{bmatrix}
    \end{split}
    \]
\newcommand{\ten}[1]{U_{#1}}
Case $m=4$.
{\scriptsize\[
    \begin{split}
        &\ten{2,3} = \begin{bmatrix}
        x_0 & 0 & 0 & 0  \\
        x_1 & x_0 & 0 & 0 \\
        x_2 & x_1 & x_0 & 0 \\
        x_3 & x_2 & x_1 & x_0 \\
    \end{bmatrix},\ 
    \ten{2,4} =\begin{bmatrix}
        x_0 & 0 & 0 & 0 \\
        x_1 & x_0 & 0 & 0  \\
        x_2 & 0 & x_0 & 0 \\
        x_3 & 0 & x_2 & x_0 \\
    \end{bmatrix},\ 
    \ten{2,5} = \begin{bmatrix}
        x_0 & 0 & 0 & 0 \\
        x_1 & x_0 & 0 & 0 \\
        x_2 & 0 & x_0 & 0 \\
        x_3 & x_2 & x_1 & x_0 \\
    \end{bmatrix},\ 
    \ten{2,6} =\begin{bmatrix}
        x_0 & 0 & 0 & 0 \\
        x_1 & x_0 & 0 & 0 \\
        x_2 & 0 & x_0 & 0 \\
        x_3 & 0 & 0 & x_0 \\
    \end{bmatrix},\\
    &\ten{2,7} = \begin{bmatrix}
        x_0 & 0 & 0 & 0 \\
        0 & x_0 & 0  & 0 \\
        x_1 & x_2 & x_0 & 0 \\
        x_3 & -x_1 & 0 & x_0 \\
    \end{bmatrix},\ 
    \ten{2,8} = \begin{bmatrix}
        x_0 & 0 & 0 & 0 \\
        0 & x_0 & 0 & 0 \\
        x_1 & x_2 & x_0 & 0 \\
        x_3 & 0 & 0 & x_0  \\
    \end{bmatrix},\ 
    \ten{3,1} = \begin{bmatrix}
        x_0 & 0 & 0 & 0\\
        x_1 & x_0 & 0 & 0\\
        x_2 & x_1 & x_0 & 0\\
        0 & 0 & 0 & x_0 + x_3 \\
    \end{bmatrix},\ 
    \ten{3,2} = \begin{bmatrix}
        x_0 & 0 & 0 & 0 \\
        x_1 & x_0 & 0 & 0\\
        x_2 & 0 & x_0 & 0\\
        0 & 0 & 0 & x_0 + x_3 \\
    \end{bmatrix},\\
    &\ten{3,3} = \begin{bmatrix}
        x_0 & 0 & 0 & 0 \\
        x_1 & x_0 & 0 & 0 \\
        0 & 0 & x_0 + x_3 & 0 \\
        0 & 0 & x_2 & x_0 + x_3 \\
    \end{bmatrix},\
    \ten{4,1} = \begin{bmatrix}
        x_0 & 0 & 0 & 0 \\
        x_1 & x_0 & 0 & 0 \\
        0 & 0 & x_0 + x_2 & 0 \\
        0 & 0 & 0 & x_0 + x_3 \\
    \end{bmatrix},\
    \ten{5,1} = \begin{bmatrix}
        x_0 & 0 & 0 & 0 \\
        0 & x_0 + x_1 & 0 & 0 \\
        0 & 0 & x_0 + x_2 & 0 \\
        0 & 0 & 0 & x_0 + x_3 \\
    \end{bmatrix}
    \end{split}
\]}
The degeneration diagram for $m=4$ is obtained directly from~\ref{sec:diagram} using
Proposition~\ref{ref:degenerationSplitting:prop}. The diagrams for $m=2,3$ can
be obtained from the one below by the same method and are easy to get anyway.
\noindent\begin{center}{\scalebox{0.8}{
\begin{tikzpicture}[node distance={20mm}, thick, main/.style = {draw, circle}]

    \node[main] (U51) {$U_{5,1}$};
    \node[main] (U41) [below of=U51] {$U_{4,1}$};
    \node[main] (U33) [below left of=U41] {$U_{3,3}$};
    \node[main] (U31) [below right of=U41] {$U_{3,1}$};
    \node[main] (U23) [below of=U31] {$U_{2,3}$};
    \node[main] (U25) [below left of=U23] {$U_{2,5}$};
    \node[main] (U32) [left=18mm of U25] {$U_{3,2}$};
    \node[main] (U24) [below right of=U32]{$U_{2,4}$};
    \node[main] (U27) [below right=27.5mm of U23]{$U_{2,7}$};
    \node[main] (U28) [below right=15mm of U27]{$U_{2,8}$};
    \node[main] (U26) [below left=15mm of U24]{$U_{2,6}$};

\draw[arrows={-Latex[length=8pt,bend,line width=0pt]}] (U27) edge (U28);
\draw[arrows={-Latex[length=8pt,bend,line width=0pt]}, out=225, in=60] (U24) edge (U26);
\draw[arrows={-Latex[length=8pt,bend,line width=0pt]}, out=340, in=160] (U24) edge (U28);
\draw[arrows={-Latex[length=8pt,bend,line width=0pt]}] (U25) edge (U24);
\draw[arrows={-Latex[length=8pt,bend,line width=0pt]}] (U41) edge (U33);
\draw[arrows={-Latex[length=8pt,bend,line width=0pt]}] (U51) edge (U41);
\draw[arrows={-Latex[length=8pt,bend,line width=0pt]}] (U41) edge (U31);
\draw[arrows={-Latex[length=8pt,bend,line width=0pt]}] (U33) edge (U23);
\draw[arrows={-Latex[length=8pt,bend,line width=0pt]}] (U31) edge (U23);
\draw[arrows={-Latex[length=8pt,bend,line width=0pt]}, bend left=15] (U31) edge (U32);
\draw[arrows={-Latex[length=8pt,bend,line width=0pt]}] (U23) edge (U25);
\draw[arrows={-Latex[length=8pt,bend,line width=0pt]}, bend left=30] (U23) edge (U27);
\draw[arrows={-Latex[length=8pt,bend,line width=0pt]}] (U32) edge (U24);

\end{tikzpicture}}}\end{center}

\small
\newcommand{\etalchar}[1]{$^{#1}$}


\begin{thebibliography}{CGLV22}

\bibitem[AM69]{Atiyah_Macdonald}
M.~F. Atiyah and I.~G. Macdonald.
\newblock {\em Introduction to commutative algebra}.
\newblock Addison-Wesley Publishing Co., Reading, Mass.-London-Don Mills, Ont.,
  1969.

\bibitem[BB13]{Ballico_Bernardi}
Edoardo Ballico and Alessandra Bernardi.
\newblock Stratification of the fourth secant variety of {V}eronese varieties
  via the symmetric rank.
\newblock {\em Adv. Pure Appl. Math.}, 4(2):215--250, 2013.

\bibitem[BB21]{buczynski}
Weronika Buczy\'{n}ska and Jaros{\l}aw Buczy\'{n}ski.
\newblock Apolarity, border rank, and multigraded {H}ilbert scheme.
\newblock {\em Duke Math. J.}, 170(16):3659--3702, 2021.

\bibitem[BCS13]{BurgisserBook}
Peter B{\"u}rgisser, Michael Clausen, and Mohammad~A Shokrollahi.
\newblock {\em Algebraic complexity theory}, volume 315 of Grundlehren der
  mathematischen Wissenschaften.
\newblock Springer Science \& Business Media, 2013.

\bibitem[BL14]{third}
Jaros{\l}aw Buczy\'{n}ski and J.~M. Landsberg.
\newblock On the third secant variety.
\newblock {\em J. Algebraic Combin.}, 40(2):475--502, 2014.

\bibitem[BL16]{Blaser_Lysikov}
Markus Bl\"aser and Vladimir Lysikov.
\newblock On degeneration of tensors and algebras.
\newblock In {\em 41st {I}nternational {S}ymposium on {M}athematical
  {F}oundations of {C}omputer {S}cience}, volume~58 of {\em LIPIcs. Leibniz
  Int. Proc. Inform.}, pages Art. No. 19, 11 pages. Schloss Dagstuhl.
  Leibniz-Zent. Inform., Wadern, 2016.

\bibitem[BO11]{bates-oeding}
Daniel~J. Bates and Luke Oeding.
\newblock Toward a salmon conjecture.
\newblock {\em Exp. Math.}, 20(3):358--370, 2011.

\bibitem[CGLS24]{CGLS}
Matthias Christandl, Fulvio Gesmundo, Vladimir Lysikov, and Vincent Steffan.
\newblock Partial degeneration of tensors.
\newblock {\em SIAM J. Matrix Anal. Appl.}, 45(1):771--800, 2024.

\bibitem[CGLV22]{Conner_Gesmundo_Landsberg_Ventura}
Austin Conner, Fulvio Gesmundo, Joseph~M. Landsberg, and Emanuele Ventura.
\newblock Rank and border rank of {K}ronecker powers of tensors and
  {S}trassen's laser method.
\newblock {\em Comput. Complexity}, 31(1):Paper No. 1, 40, 2022.

\bibitem[CGZ23]{CGZ}
Matthias Christandl, Fulvio Gesmundo, and Jeroen Zuiddam.
\newblock A gap in the subrank of tensors.
\newblock {\em SIAM J. Appl. Algebra Geom.}, 7(4):742--767, 2023.

\bibitem[CHL23]{Conner_Huang_Landsberg}
Austin Conner, Hang Huang, and J.~M. Landsberg.
\newblock Bad and good news for {S}trassen's laser method: border rank of
  {${\rm perm}_3$} and strict submultiplicativity.
\newblock {\em Found. Comput. Math.}, 23(6):2049--2087, 2023.

\bibitem[CVZ19]{Christandl__barriers}
Matthias Christandl, P\'eter Vrana, and Jeroen Zuiddam.
\newblock Barriers for fast matrix multiplication from irreversibility.
\newblock In {\em 34th {C}omputational {C}omplexity {C}onference}, volume 137
  of {\em LIPIcs. Leibniz Int. Proc. Inform.}, pages Art. No. 26, 17. Schloss
  Dagstuhl. Leibniz-Zent. Inform., Wadern, 2019.

\bibitem[Eis95]{eisenbud}
David Eisenbud.
\newblock {\em Commutative algebra}, volume 150 of {\em Graduate Texts in
  Mathematics}.
\newblock Springer-Verlag, New York, 1995.
\newblock With a view toward algebraic geometry.

\bibitem[FGI{\etalchar{+}}05]{fantechi_et_al_fundamental_ag}
Barbara Fantechi, Lothar G{\"o}ttsche, Luc Illusie, Steven~L. Kleiman, Nitin
  Nitsure, and Angelo Vistoli.
\newblock {\em Fundamental algebraic geometry}, volume 123 of {\em Mathematical
  Surveys and Monographs}.
\newblock American Mathematical Society, Providence, RI, 2005.
\newblock Grothendieck's FGA explained.

\bibitem[FO14]{Friedland_Ottaviani}
Shmuel Friedland and Giorgio Ottaviani.
\newblock The number of singular vector tuples and uniqueness of best rank-one
  approximation of tensors.
\newblock {\em Found. Comput. Math.}, 14(6):1209--1242, 2014.

\bibitem[Fri13]{friedland_4}
Shmuel Friedland.
\newblock On tensors of border rank {$l$} in {$\Bbb{C}^{m\times n\times l}$}.
\newblock {\em Linear Algebra Appl.}, 438(2):713--737, 2013.

\bibitem[HJMS22]{Homs_Jelisiejew_Michalek_Seynnaeve}
Roser Homs, Joachim Jelisiejew, Mateusz Micha{\l}ek, and Tim Seynnaeve.
\newblock Bounds on complexity of matrix multiplication away from
  {C}oppersmith-{W}inograd tensors.
\newblock {\em J. Pure Appl. Algebra}, 226(12):Paper No. 107142, 16 pages,
  2022.

\bibitem[IK99]{Iarrobino_Kanev}
Anthony Iarrobino and Vassil Kanev.
\newblock {\em Power sums, {G}orenstein algebras, and determinantal loci},
  volume 1721 of {\em Lecture Notes in Mathematics}.
\newblock Springer-Verlag, Berlin, 1999.
\newblock Appendix C by Anthony Iarrobino and Steven L. Kleiman.

\bibitem[JLP23]{concise}
Joachim Jelisiejew, J.~M. Landsberg, and Arpan Pal.
\newblock Concise tensors of minimal border rank.
\newblock {\em Mathematische Annalen}, Feb 2023.

\bibitem[JS19]{jelisiejew_sienkiewicz__BB}
Joachim Jelisiejew and {\L{}}ukasz Sienkiewicz.
\newblock {B}ia{\l{}}ynicki-{B}irula decomposition for reductive groups.
\newblock {\em Journal de Mathématiques Pures et Appliquées}, 131:290 -- 325,
  2019.

\bibitem[J{\v{S}}22]{components}
Joachim Jelisiejew and Klemen {\v{S}}ivic.
\newblock Components and singularities of {Q}uot schemes and varieties of
  commuting matrices.
\newblock {\em J. Reine Angew. Math.}, 788:129--187, 2022.

\bibitem[Kun11]{Kunte__Gorenstein_modules_of_finite_length}
Michael Kunte.
\newblock Gorenstein modules of finite length.
\newblock {\em Math. Nachr.}, 284(7):899--919, 2011.

\bibitem[Lan12]{landsberg_tensors}
J.~M. Landsberg.
\newblock {\em Tensors: geometry and applications}, volume 128 of {\em Graduate
  Studies in Mathematics}.
\newblock American Mathematical Society, Providence, RI, 2012.

\bibitem[Lan17]{landsberg_complexity}
J.~M. Landsberg.
\newblock {\em Geometry and complexity theory}, volume 169 of {\em Cambridge
  Studies in Advanced Mathematics}.
\newblock Cambridge University Press, Cambridge, 2017.

\bibitem[LM17]{abelian}
J.~M. Landsberg and Mateusz Micha{\l}ek.
\newblock Abelian tensors.
\newblock {\em J. Math. Pures Appl. (9)}, 108(3):333--371, 2017.

\bibitem[LT10]{Landsberg_Teitler}
J.~M. Landsberg and Zach Teitler.
\newblock On the ranks and border ranks of symmetric tensors.
\newblock {\em Found. Comput. Math.}, 10(3):339--366, 2010.

\bibitem[Maz80]{mazzola_generic_finite_schemes}
Guerino Mazzola.
\newblock Generic finite schemes and {H}ochschild cocycles.
\newblock {\em Comment. Math. Helv.}, 55(2):267--293, 1980.

\bibitem[MR18]{Moschetti_Ricolfi}
Riccardo Moschetti and Andrea~T. Ricolfi.
\newblock On coherent sheaves of small length on the affine plane.
\newblock {\em J. Algebra}, 516:471--489, 2018.

\bibitem[MZ14]{Mroz_Zwara}
Andrzej Mr{\'o}z and Grzegorz Zwara.
\newblock Combinatorial algorithms for computing degenerations of modules of
  finite dimension.
\newblock {\em Fund. Inform.}, 132(4):519--532, 2014.

\bibitem[Poo08]{poonen}
Bjorn Poonen.
\newblock Isomorphism types of commutative algebras of finite rank over an
  algebraically closed field.
\newblock In {\em Computational arithmetic geometry}, volume 463 of {\em
  Contemp. Math.}, pages 111--120. Amer. Math. Soc., Providence, RI, 2008.

\bibitem[ST03]{perestanovochnye}
D.~A. Suprunenko and R.~I. Tyshkevich.
\newblock {\em { Perestanovochnye matritsy}}.
\newblock \`Editorial URSS, Moscow, second edition, 2003.

\bibitem[Str83]{strassen}
V.~Strassen.
\newblock Rank and optimal computation of generic tensors.
\newblock {\em Linear Algebra Appl.}, 52/53:645--685, 1983.

\bibitem[{Woj}24]{Wojtala}
Maciej {Wojtala}.
\newblock Iarrobino's decomposition for self-dual modules.
\newblock arXiv:2405.13829, 2024.

\end{thebibliography}
\end{document}